\documentclass[11pt]{article}

\usepackage{amssymb,amsmath,amscd,amsthm,xy,enumitem,mathabx,mathrsfs,color,caption,stmaryrd}
\usepackage{amsxtra,pstricks,pst-node}
\xyoption{all}

\newcounter{svcnt}

\newtheorem{thm}{Theorem}[section]
\newtheorem{prp}[thm]{Proposition}
\newtheorem{lmm}[thm]{Lemma}

\newtheorem{crl}[thm]{Corollary}

\theoremstyle{definition}
\newtheorem{dfn}[thm]{Definition}
\newtheorem{eg}[thm]{Example}

\theoremstyle{remark}
\newtheorem{rmk}[thm]{Remark}

\numberwithin{equation}{section}

\def\lra{\longrightarrow}

\def\BE#1{\begin{equation}\label{#1}}
\def\EE{\end{equation}}
\def\lr#1{\langle#1\rangle}
\def\flr#1{\left\lfloor{#1}\right\rfloor}
\def\blr#1{\big\langle#1\big\rangle}

\def\wt#1{\widetilde{#1}}
\def\ov#1{\overline{#1}}
\def\eref#1{(\ref{#1})}
\def\tn#1{\textnormal{#1}}
\def\sf#1{\textsf{#1}}

\def\ch#1{\check{#1}}

\def\smsize#1{\begin{small}#1\end{small}}

\def\De{\Delta}
\def\Ga{\Gamma}
\def\La{\Lambda}

\def\Si{\Sigma}

\def\al{\alpha}
\def\be{\beta}

\def\eps{\epsilon}
\def\ga{\gamma}
\def\io{\iota}

\def\la{\lambda}
\def\na{\nabla}
\def\om{\omega}
\def\si{\sigma}

\def\ups{\upsilon}

\def\vph{\varphi}

\def\vt{\vartheta}

\def\cA{\mathcal A}

\def\C{\mathbb C}
\def\cC{\mathcal C}  
\def\cD{\mathcal{D}}
\def\bE{\mathbb E}
\def\E{\tn{E}}

\def\cH{\mathcal H}

\def\fI{\mathfrak i}
\def\cJ{\mathcal J}
\def\fJ{\mathfrak j}
\def\cK{\mathcal K}

\def\bL{\mathbb L}
\def\cL{\mathcal L}

\def\cM{\mathcal M}
\def\fM{\mathfrak M}

\def\cN{\mathcal N}
\def\cO{\mathcal O}
\def\fO{\mathfrak O}
\def\P{\mathbb P}

\def\Q{\mathbb Q}
\def\R{\mathbb{R}}

\def\bT{\mathbb{T}}
\def\cT{\mathcal T}
\def\cU{\mathcal U}

\def\cV{\mathcal V}
\def\nV{\tn{V}}

\def\fX{\mathfrak X}
\def\Z{\mathbb{Z}}
\def\cZ{\mathcal Z}

\def\a{\mathbf a}

\def\b{\mathbf b}

\def\c{\mathbf c}
\def\fc{\mathfrak c}
\def\fd{\mathfrak d}
\def\e{\mathbf e}
\def\h{\mathbf h}

\def\ff{\mathfrak f}
\def\g{\mathfrak g}

\def\fm{\mathfrak m}
\def\p{\mathbf p}
\def\fs{\mathfrak s}

\def\u{\mathbf u}
\def\x{\mathbf x}
\def\y{\mathbf y}
\def\Y{\mathbf Y}
\def\z{\mathbf z}

\def\Aut{\tn{Aut}}
\def\Cntr{\tn{Cntr}}
\def\tnd{\textnormal{d}}
\def\Edg{\tn{Edg}}
\def\ev{\tn{ev}}
\def\F{\tn{F}}
\def\Fix{\tn{Fix}}

\def\GW{\tn{GW}}

\def\Im{\tn{Im}}
\def\id{\textnormal{id}}
\def\Id{\tn{Id}}
\def\nd{\tn{nd}}

\def\pt{\tn{pt}}
\def\PD{\tn{PD}}
\def\rk{\tn{rk}\,}
\def\Sym{\tn{Sym}}

\def\top{\textnormal{top}}
\def\Ver{\tn{Ver}}
\def\vrt{\tn{vrt}}
\def\val{\tn{val}}

\def\0{\mathbf 0}
\def\1{\mathbf 1}

\def\dbar{\bar\partial}
\def\prt{\partial}
\def\eset{\emptyset}
\def\i{\infty}

\def\bu{\bullet}

\begin{document}

\title{Real Gromov-Witten Theory in All Genera and\\ 
Real Enumerative Geometry: Computation}
\author{Penka Georgieva\thanks{Partially supported by ERC grant STEIN-259118} $~$and 
Aleksey Zinger\thanks{Partially supported by NSF grants DMS 1500875 and MPIM}}
\date{\today}
\maketitle

\begin{abstract}
\noindent
The first part of this work constructs real positive-genus   Gromov-Witten invariants
of real-orientable symplectic manifolds of odd ``complex" dimensions;
the second part studies the orientations on the moduli spaces of real maps
used in constructing these invariants.
The present paper applies the results of the latter to 
obtain quantitative and qualitative conclusions about the invariants defined in the former.
After describing large collections of real-orientable symplectic manifolds,  
we show that the  real genus~1  Gromov-Witten invariants
of sufficiently positive almost Kahler threefolds are signed counts of real genus~1 curves only
and thus provide direct lower bounds for the counts of these curves in such targets.
We specify real orientations on the real-orientable complete intersections in projective spaces;
the real Gromov-Witten invariants they determine 
are in a sense canonically determined by the complete intersection itself, 
(at least) in most cases.
We also obtain equivariant localization data that computes the real invariants 
of projective spaces and 
determines the contributions from many torus fixed loci for other complete intersections.
Our results confirm Walcher's predictions for the vanishing of these invariants in certain cases
and for the localization data in other cases.
\end{abstract}

\tableofcontents

\section{Introduction}
\label{intro_sec}

\noindent
The theory of $J$-holomorphic maps plays prominent roles in symplectic topology,
algebraic geometry, and string theory.
The foundational work of~\cite{Gr,McSa94,RT,LT,FO} has 
established the theory of (closed) Gromov-Witten invariants,
i.e.~counts of $J$-holomorphic maps from closed Riemann surfaces to symplectic manifolds.
In~\cite{RealGWsI}, we introduce the notion of \sf{real orientation} on 
a real symplectic manifold $(X,\om,\phi)$. 
A real orientation on a real symplectic $2n$-manifold $(X,\om,\phi)$ with $n\!\not\in\!2\Z$
induces orientations on the moduli spaces of real $J$-holomorphic maps from arbitrary genus~$g$
symmetric surfaces to~$(X,\phi)$ commuting 
with the involutions on the domain and the target
and thus gives rise
to arbitrary-genus real Gromov-Witten invariants,  
i.e.~counts of such maps, 
for  \sf{real-orientable} symplectic manifolds of odd ``complex" dimension~$n$.
These orientations are studied in detail in~\cite{RealGWsII}.
The present paper applies the results of~\cite{RealGWsII} to 
the computation of these invariants.\\

\noindent
Propositions~\ref{CYorient_prp}  and~\ref{CIorient_prp} 
provide large collections of real-orientable symplectic manifolds;
they include the odd-dimensional projective spaces and the quintic threefold,
which plays a special role in the interactions of Gromov-Witten theory with string theory.
We also show that the  genus~1  real Gromov-Witten invariants
of sufficiently positive almost Kahler threefolds are signed counts of real genus~1 curves only;
see Theorem~\ref{g1EG_thm}.
The ``classical" Gromov-Witten invariants in contrast include genus~0 contributions 
and generally are not integer; see~\eref{GWvsEnum_e}.
Thus, the  genus~1  real Gromov-Witten invariants provide direct lower bounds for the counts of
real genus~1  curves in many almost Kahler threefolds.\\

\noindent
An explicit real orientation on each complete intersection~$X_{n;\a}$ 
of Proposition~\ref{CIorient_prp} is specified in Section~\ref{RealOrientCI_subs}.
This real orientation may depend on 
the parametrization of the ambient projective space with its involution,
on the ordering of the line bundles corresponding to~$X_{n;\a}$,
and on the section~$s_{n;\a}$ cutting out~$X_{n;\a}$
(i.e.~negating some components of~$s_{n;\a}$ may change the induced real orientation).
However, we show that the real Gromov-Witten invariants of
a fixed complete intersection~$X_{n;\a}$ determined by this real orientation 
are independent of all these choices in most cases and 
vanish in the subset of these cases predicted in~\cite{Wal}; see Theorem~\ref{GWsCI_thm}.
In the cases not covered by the independence statement of this theorem,
the real Gromov-Witten invariants are expected to vanish as well.
These invariants are also preserved by linear inclusions of the ambient projective space
into larger projective spaces; see Proposition~\ref{CROprop_prp}.
Furthermore, the signed count of real lines through a pair of conjugate 
points in  $\P^{2m-1}$ with the standard conjugation is $+1$ with respect to these orientations:
\BE{d1tau_e0} \blr{H^{2m-1}}_{0,1}^{\P^{2m-1},\tau_{2m}}=+1.\EE
Section~\ref{LocData_subs} describes the equivariant localization data that computes 
the contributions to these invariants from many torus fixed loci. 
If $X_{n;\a}\!=\!\P^{2m-1}$ (but $g$ is arbitrary) or $g\!=\!0$ (but $X_{n;\a}$ is arbitrary),
this description covers all fixed loci and thus completely determines 
the real Gromov-Witten invariants in these cases.
Theorem~\ref{EquivLocal_thm} is used in~\cite{NZapp} to compute the real genus~$g$ degree~$d$
Gromov-Witten invariants of~$\P^3$ with $d$~conjugate pairs of point constraints
for $g\!\le\!5$ and $d\!\le\!8$.
The equivariant localization data of Section~\ref{LocData_subs} agrees with \cite[(3.22)]{Wal}.
This  implies that the spin structure on the real locus 
of the quintic threefold $X_{5;(5)}$ used, but not explicitly specified, in~\cite{PSW} 
is the spin structure associated with our real orientation 
on $(X_{5;(5)},\tau_{5;(5)})$.

\subsection{Real-orientable symplectic manifolds}
\label{RealGWth_subs}

\noindent
An \textsf{involution} on a topological space~$X$ is a homeomorphism
$\phi\!:X\!\lra\!X$ such that $\phi\!\circ\!\phi\!=\!\id_X$.
By an \textsf{involution on a manifold}, we will mean a smooth involution. 
Let
$$X^\phi=\big\{x\!\in\!X\!:~\phi(x)\!=\!x\big\}$$
denote the fixed locus.
An \sf{anti-symplectic involution~$\phi$} on a symplectic manifold $(X,\om)$
is an involution $\phi\!:X\!\lra\!X$ such that $\phi^*\om\!=\!-\om$.
For example, the~maps
\BE{tauetadfn_e}
\begin{aligned}
\tau_n\!:\P^{n-1}&\lra\P^{n-1}, &\quad [Z_1,\ldots,Z_n]&\lra[\ov{Z}_1,\ldots,\ov{Z}_n],\\
\eta_{2m}\!: \P^{2m-1}& \lra\P^{2m-1},&\quad
[Z_1,\ldots,Z_{2m}]&\lra 
\big[\ov{Z}_2,-\ov{Z}_1,\ldots,\ov{Z}_{2m},-\ov{Z}_{2m-1}\big],
\end{aligned}\EE
are anti-symplectic involutions with respect to the standard Fubini-Study symplectic
forms~$\om_n$ on~$\P^{n-1}$ and~$\om_{2m}$ on $\P^{2m-1}$, respectively.
If 
$$k\!\ge\!0, \qquad \a\equiv(a_1,\ldots,a_k)\in(\Z^+)^k\,,$$
and $X_{n;\a}\!\subset\!\P^{n-1}$ is a complete intersection of multi-degree~$\a$
preserved by~$\tau_n$,  then $\tau_{n;\a}\!\equiv\!\tau_n|_{X_{n;\a}}$
is an anti-symplectic involution on $X_{n;\a}$ with respect to the symplectic form
$\om_{n;\a}\!=\!\om_n|_{X_{n;\a}}$. 
Similarly, if $X_{2m;\a}\!\subset\!\P^{2m-1}$ is preserved by~$\eta_{2m}$, then
$\eta_{2m;\a}\!\equiv\!\eta_{2m}|_{X_{2m;\a}}$
is an anti-symplectic involution on $X_{2m;\a}$ with respect to the symplectic form
$\om_{2m;\a}\!=\!\om_{2m}|_{X_{2m;\a}}$.
A \sf{real symplectic manifold} is a triple $(X,\om,\phi)$ consisting 
of a symplectic manifold~$(X,\om)$ and an anti-symplectic involution~$\phi$.\\

\noindent
Let $(X,\phi)$ be a topological space with an involution.
A \sf{conjugation} on a complex vector bundle $V\!\lra\!X$ 
\sf{lifting} an involution~$\phi$ is a vector bundle homomorphism 
$\wt\phi\!:V\!\lra\!V$ covering~$\phi$ (or equivalently 
a vector bundle homomorphism  $\wt\phi\!:V\!\lra\!\phi^*V$ covering~$\id_X$)
such that the restriction of~$\wt\phi$ to each fiber is anti-complex linear
and $\wt\phi\!\circ\!\wt\phi\!=\!\id_V$.
A \sf{real bundle pair} $(V,\wt\phi)\!\lra\!(X,\phi)$   
consists of a complex vector bundle $V\!\lra\!X$ and 
a conjugation~$\wt\phi$ on $V$ lifting~$\phi$.
For example, 
$$(X\!\times\!\C^n,\phi\!\times\!\fc)\lra(X,\phi),$$
where $\fc\!:\C^n\!\lra\!\C^n$ is the standard conjugation on~$\C^n$,
is a real bundle pair.
If $X$ is a smooth manifold, then $(TX,\tnd\phi)$ is also a real bundle pair over~$(X,\phi)$.
For any real bundle pair $(V,\wt\phi)$ over~$(X,\phi)$, 
we denote~by
$$\La_{\C}^{\top}(V,\wt\phi)=(\La_{\C}^{\top}V,\La_{\C}^{\top}\wt\phi)$$
the top exterior power of $V$ over $\C$ with the induced conjugation.
Direct sums, duals, and tensor products over~$\C$ of real bundle pairs over~$(X,\phi)$
are again real bundle pairs over~$(X,\phi)$.

\begin{dfn}[{\cite[Definition~5.1]{RealGWsI}}]\label{realorient_dfn4}
Let $(X,\phi)$ be a topological space with an involution and 
$(V,\vph)$ be a real bundle pair over~$(X,\phi)$.
A \sf{real orientation} on~$(V,\vph)$ consists~of 
\begin{enumerate}[label=(RO\arabic*),leftmargin=*]

\item\label{LBP_it2} a rank~1 real bundle pair $(L,\wt\phi)$ over $(X,\phi)$ such that 
\BE{realorient_e4}
w_2(V^{\vph})=w_1(L^{\wt\phi})^2 \qquad\hbox{and}\qquad
\La_{\C}^{\top}(V,\vph)\approx(L,\wt\phi)^{\otimes 2},\EE

\item\label{isom_it2} a homotopy class of isomorphisms of real bundle pairs in~\eref{realorient_e4}, and

\item\label{spin_it2} a spin structure~on the real vector bundle
$V^{\vph}\!\oplus\!2(L^*)^{\wt\phi^*}$ over~$X^{\phi}$
compatible with the orientation induced by~\ref{isom_it2}.\\ 

\end{enumerate}
\end{dfn}

\noindent
An isomorphism in~\eref{realorient_e4} restricts to an isomorphism 
\BE{realorient2_e3}\La_{\R}^{\top}V^{\vph}\approx (L^{\wt\phi})^{\otimes2}\EE
of real line bundles over~$X^{\phi}$.
Since the vector bundles $(L^{\wt\phi})^{\otimes2}$ and $2(L^*)^{\wt\phi^*}$ are canonically oriented, 
\ref{isom_it2} determines orientations on $V^{\vph}$ and $V^{\vph}\!\oplus\! 2(L^*)^{\wt\phi^*}$.
By the first assumption in~\eref{realorient_e4}, the real vector bundle
$V^{\vph}\!\oplus\!2(L^*)^{\wt\phi^*}$ over~$X^{\phi}$ admits a spin structure.\\

\noindent
Let $(X,\om,\phi)$ be a real symplectic manifold.
A \sf{real orientation on~$(X,\om,\phi)$} is a real orientation on the real bundle pair $(TX,\tnd\phi)$.
We call $(X,\om,\phi)$ \sf{real-orientable} if it admits a real orientation.
The next three statements, which are established in Section~\ref{RealOrientMfld_subs},
 describe large collections of real-orientable symplectic manifolds.

\begin{prp}\label{CYorient_prp}
Let $(X,\om,\phi)$ be a real symplectic manifold with $w_2(X^{\phi})\!=\!0$.
If 
\begin{enumerate}[label=(\arabic*),leftmargin=*]

\item
$H_1(X;\Q)\!=\!0$ and $c_1(X)\!=\!2(\mu\!-\!\phi^*\mu)$ for some $\mu\!\in\!H^2(X;\Z)$ or

\item\label{KahlerOrient_it} $X$ is  compact Kahler, $\phi$ is anti-holomorphic, and 
$\cK_X\!=\!2([D]\!+\![\ov{\phi_*D}])$ for some divisor $D$ on~$X$,

\end{enumerate}
then $(X,\om,\phi)$ is a real-orientable symplectic manifold.
\end{prp}

\begin{crl}\label{CIorient_crl}
Let $n\!\in\!\Z^+$ and $\a\!\equiv\!(a_1,\ldots,a_{n-4})\!\in\!(\Z^+)^{n-4}$ be such~that 
$$a_1\!+\!\ldots\!+\!a_{n-4}\equiv n \mod4\,.$$
If $X_{n;\a}\!\subset\!\P^{n-1}$ is a complete intersection of multi-degree~$\a$
preserved by~$\tau_n$, then  
$(X_{n;\a},\om_{n;\a},\tau_{n;\a})$ is a real-orientable symplectic manifold.
\end{crl}

\begin{prp}\label{CIorient_prp}
Let $m,n\!\in\!\Z^+$, $k\!\in\!\Z^{\ge0}$, and $\a\!\equiv\!(a_1,\ldots,a_k)\!\in\!(\Z^+)^k$.
\begin{enumerate}[label=(\arabic*),leftmargin=*]
\item\label{CItau_it} 
If $X_{n;\a}\!\subset\!\P^{n-1}$ is a complete intersection of multi-degree~$\a$
preserved by~$\tau_n$,
\BE{CIorient_e1}
a_1\!+\!\ldots\!+\!a_k\equiv n \mod2, \quad\hbox{and}\quad
a_1^2\!+\!\ldots\!+\!a_k^2
\equiv a_1\!+\!\ldots\!+\!a_k \mod4,\EE
then $(X_{n;\a},\om_{n;\a},\tau_{n;\a})$ is a real-orientable symplectic manifold.

\item\label{CIeta_it}  If $X_{2m;\a}\!\subset\!\P^{2m-1}$ is a complete intersection of 
multi-degree~$\a$ preserved by~$\eta_{2m}$ and
$$a_1\!+\!\ldots\!+\!a_k\equiv 2m \mod4,$$
then $(X_{2m;\a},\om_{2m;\a},\eta_{2m;\a})$ is a real-orientable symplectic manifold.\\
\end{enumerate}
\end{prp}

\noindent
As indicated by the proof of Corollary~\ref{CIorient_crl} in Section~\ref{RealOrientMfld_subs}, 
the condition $w_2(X^{\phi})\!=\!0$ in Proposition~\ref{CYorient_prp} is redundant if $X$ is a threefold.
In particular, every real compact Kahler Calabi-Yau threefold is real-orientable.\\

\noindent
By Proposition~\ref{CIorient_prp}, $(\P^{2m-1},\tau_{2m})$ and $(\P^{4m-1},\eta_{4m})$
are  real-orientable symplectic manifolds.
So are the   complete intersection Calabi-Yau threefolds
$$X_{5;(5)}\subset\P^4, \qquad  X_{6;(3,3)}\subset\P^5,
\quad\hbox{and}\quad X_{8;(2,2,2,2)}\subset\P^7$$
preserved by the standard conjugation on the ambient space.
Proposition~\ref{CIorient_prp}\ref{CItau_it} does not apply to 
the two remaining   projective  complete intersection Calabi-Yau threefolds,
$$X_{6;(2,4)}\subset\P^5 \qquad\hbox{and}\qquad X_{7;(2,2,3)}\subset\P^6,$$
as they do not satisfy the second conditions in~\eref{CIorient_e1}.
These two threefolds are instead real-orientable by Corollary~\ref{CIorient_crl},
whenever they are preserved by the standard conjugation on the ambient space.
In contrast to Proposition~\ref{CIorient_prp}\ref{CItau_it}, 
Corollary~\ref{CIorient_crl} does not endow the symplectic manifolds 
to which it applies with a natural real orientation.

\subsection{Gromov-Witten and~enumerative invariants}
\label{RealGWs_subs}

\noindent
A \sf{symmetric surface} $(\Si,\si)$ is a closed  oriented (possibly nodal) surface~$\Si$  
with an orientation-reversing involution~$\si$.
If $\Si$ is smooth, the fixed locus~$\Si^{\si}$ of~$\si$ is a disjoint union of circles.
If in addition $(X,\phi)$ is a manifold with an involution, 
a \sf{real map} 
$$u\!:(\Si,\si)\lra(X,\phi)$$ 
is a smooth map $u\!:\Si\!\lra\!X$ such that $u\!\circ\!\si=\phi\!\circ\!u$.\\

\noindent
For a symplectic manifold $(X,\om)$, we denote~by $\cJ_{\om}$
the space of $\om$-compatible almost complex structures on~$X$.
If $\phi$ is  an anti-symplectic involution on~$(X,\om)$, let 
$$\cJ_{\om}^{\phi}=\big\{J\!\in\!\cJ_{\om}\!:\,\phi^*J\!=\!-J\big\}.$$
For a genus~$g$ symmetric surface~$(\Si,\si)$, possibly nodal and disconnected,
we similarly denote by $\cJ_{\Si}^{\si}$
the space of complex structures~$\fJ$ on~$\Si$ compatible with the orientation such that 
$\si^*\fJ\!=\!-\fJ$.
For $J\!\in\!\cJ_{\om}^{\phi}$, a \sf{real $J$-holomorphic} map consists
of a symmetric surface~$(\Si,\si)$, $\fJ\!\in\!\cJ_{\Si}^{\si}$, 
and a real $(J,\fJ)$-holomorphic map $u\!:\Si\!\lra\!X$.\\

\noindent 
Let $(X,\om,\phi)$ be a real-orientable symplectic $2n$-manifold with $n\!\not\in\!2\Z$,
$g,l\!\in\!\Z^{\ge0}$, $B\!\in\!H_2(X;\Z)$, and $J\!\in\!\cJ_{\om}^{\phi}$.
We denote by $\ov\fM_{g,l}(X,B;J)^{\phi}$ 
the  moduli space of equivalence classes of stable real degree~$B$  $J$-holomorphic maps
from genus~$g$ symmetric (possibly nodal) surfaces with $l$ conjugate pairs of  marked points. 
By \cite[Theorem~1.4]{RealGWsI}, a real orientation on~$(X,\om,\phi)$ 
determines an orientation on this compact space, 
endows it with a virtual fundamental class, and  
thus gives rise to real genus~$g$ GW-invariants of $(X,\om,\phi)$
that are independent of the choice of~$J\!\in\!\cJ_{\om}^{\phi}$.
If $n\!=\!3$ and $c_1(X)$ is divisible by~4, a real orientation on~$(X,\om,\phi)$ also determines 
a count of real genus~1 $J$-holomorphic curves with conjugate and real point insertions;
see \cite[Theorem~1.5]{RealGWsI}.\\

\noindent
The real genus~$g$  GW-invariants of \cite[Theorem~1.4]{RealGWsI}  are 
in general combinations of counts of real curves of genus~$g$ and 
counts of real curves of lower genera, just as happens in the ``classical" complex setting.
In light of \cite[Theorems~1A/B]{g1diff} and \cite[Theorem~1.5]{FanoGV}, 
it seems plausible that integral counts of genus~$g$ real curves alone can  
be extracted from these GW-invariants to directly provide lower bounds for
enumerative counts of real curves in good situations.
This would typically involve delicate obstruction analysis.
However, the situation is fairly simple if $g\!=\!1$ and $n\!=\!3$.
The statement below
applies whenever it makes sense in the context of \cite[Theorems~1.4,1.5]{RealGWsI}; 
see the beginning of Section~\ref{g1EG_sec} and~\eref{GWvsEnumR_e}.

\begin{thm}\label{g1EG_thm}
Let $(X,\om,\phi)$ be a compact real-orientable 6-manifold 
and $J\!\in\!\cJ_{\om}^{\phi}$ be an almost  complex structure which is
genus~1 regular in the sense of \cite[Definition~1.4]{g1comp}.
The genus~1 real GW-invariants of $(X,\om,\phi)$ are then equal to 
the corresponding signed counts of real $J$-holomorphic curves
and thus provide lower bounds for the number of real genus~1 irreducible curves in~$(X,J,\phi)$. 
\end{thm}

\noindent
In contrast, the complex genus~1 degree~$d$ enumerative and GW-invariants of~$\P^3$
are related by the formula
\BE{GWvsEnum_e}  \E_{1,d}=\GW_{1,d}+\frac{2d\!-\!1}{12}\GW_{0,d}\,.\EE
This formula, originally announced as Theorem~A in~\cite{Getzler},
is established as a special case of \cite[Theorem~1.1]{g1comp2},
comparing standard and ``reduced'' GW-invariants 
(the latter do not ``see" the genus~0 curves in sufficiently positive cases).\\

\noindent
The real genus~$g$ GW-invariants of $(X,\om,\phi)$ with $g\!\ge\!2$, 
 include counts of lower genus curves with rational coefficients,
even if the real dimension of~$X$ is~6 and $X$ is very positive.
For sufficiently positive almost Kahler real-orientable manifolds $(X,\om,\phi)$
of real dimension~6, these contributions are determined in~\cite{NZ}.


\subsection{Real invariants of complete intersections}
\label{VanCompPrp_subs}

\noindent
Explicit real orientations on the real-orientable complete intersections 
of Proposition~\ref{CIorient_prp} are described in Section~\ref{RealOrientCI_subs}.
They in general depend on some auxiliary choices made. 
By studying the effect of these choices on the resulting orientations of
the moduli spaces of real maps,
we establish the following statement in Section~\ref{CROprop_subs}.

\begin{thm}\label{GWsCI_thm}
Let $(X,\phi)$ denote an odd-dimensional 
complete intersection $(X_{n;\a},\om_{n;\a},\tau_{n;\a})$ or $(X_{2m;\a},\om_{2m;\a},\eta_{2m;\a})$
satisfying the assumptions of Proposition~\ref{CIorient_prp}
and $g,d\!\in\!\Z^{\ge0}$.
\begin{enumerate}[label=(\arabic*),leftmargin=*]

\item\label{CIindep_it}   If either $\phi\!=\!\eta_{n;\a}$, or $kg\!=\!0$, or
$g\!\not\in\!2\Z$ and $d\!\in\!2\Z$, or 
$a_i\!\not\in\!2\Z$ for all~$i$ and $d\!-\!g\!\not\in\!2\Z$, then
the genus~$g$ degree~$d$ real GW-invariants of $(X,\phi)$
determined by the real orientation of Section~\ref{RealOrientCI_subs} 
with insertions from~$\P^{n-1}$
are independent of the real parametrization of the ambient projective space, 
the ordering of the line bundles associated with the complete intersection,
and the section~$s_{n;\a}$ cutting out the complete intersection~$(X,\phi)$.

\item\label{CIvan_it} 
Suppose the genus~$g$ degree~$d$ real GW-invariants of $(X,\phi)$ with insertions from~$\P^{n-1}$
are independent of the real parametrization of the ambient projective space, 
the ordering of the line bundles associated with the complete intersection,
and the section~$s_{n;\a}$.
If either  $d\!-\!g\!\in\!2\Z$ or $a_i\!\in\!2\Z$ for some~$i$ and $g\!\in\!2\Z$, then these
invariants vanish.\\
\end{enumerate} 
\end{thm}

\noindent
This theorem refers to the real GW-invariants of \cite[Theorems~1.4,1.5]{RealGWsI}.
The proof we give in Section~\ref{CROprop_subs} addresses the case of 
conjugate pairs of primary insertions
(i.e.~pullbacks of cohomology classes on~$\P^{n-1}$), but it is easily adaptable to
the descendant insertions (i.e.~$\psi$-classes).
It also applies to the real genus~1 GW-invariants with real point insertions
under the conditions of \cite[Theorem~1.5]{RealGWsI},
because the number of real point insertions is then even and 
the appropriate analogue of the middle vertical arrow in~\eref{CROdiag_e}
is thus still orientation-preserving.
Theorems~\ref{g1EG_thm} and~\ref{GWsCI_thm}\ref{CIvan_it} 
imply that the genus~1 odd-degree complex enumerative 
invariants of~$\P^3$ with pairs of the same insertions are even,
as claimed in \cite[Corollary~2.6]{RealGWsI}.\\

\noindent
The last assumption in Theorem~\ref{GWsCI_thm}\ref{CIvan_it} is that the real invariants
of the complete intersections cut out by two different transverse real sections~$s_{n;\a}$
and~$s_{n;\a}'$ are the same (even if the complete intersections are~not).
The last conclusion of Theorem~\ref{GWsCI_thm}\ref{CIindep_it} is weaker,
as the complete intersection~$(X,\phi)$ is fixed.
By the next paragraph, the stronger conclusion holds if $\phi\!=\!\eta_{n;\a}$ or $kg\!=\!0$.
We expect that it holds in all cases; see the next paragraph and Remark~\ref{Cobordism_rmk}.
The GW-invariants of~$(X,\phi)$ are predicted in~\cite{Wal}  to vanish
in all cases not covered by the assumptions of Theorem~\ref{GWsCI_thm}\ref{CIindep_it}.
Conversely, the independence of the real GW-invariants of~$(X,\phi)$
of all choices and in all cases would confirm the vanishing predictions of~\cite{Wal}
in all cases.\\

\noindent
In the case of the complete intersections $(X_{2m;\a},\eta_{2m;\a})$,
the space of transverse sections is path-connected;
see the proof of Theorem~\ref{GWsCI_thm}.
The construction of real orientations of Section~\ref{RealOrientCI_subs} 
lifts over such paths.
Thus, the real orientations on these complete intersections 
and the resulting real GW-invariants 
are independent 
of the choice of the section~$s_{2m;\a}$ cutting out~$X_{2m;\a}$.
On the other hand, 
the topology of the real locus of a (smooth) complete intersection $(X_{n;\a},\tau_{n;\a})$
may depend on the choice of the section~$s_{n;\a}$ cutting out~$X_{n;\a}$.
Thus, the real orientations of Section~\ref{RealOrientCI_subs} 
for different complete intersections of multidegree~$\a$ in $(\P^{n-1},\tau_n)$ 
are not comparable in general.
However, the resulting real GW-invariants can still be compared.
In the case of the primary insertions arising from~$\P^{n-1}$,
these invariants are expected to be related to the sheaf $\cV_{n;\a}^{\wt\phi_{n;\a}}$ in~\eref{cVdfn_e}
in a manner independent of the section~$s_{n;\a}$.
This is indeed the case for the $g\!=\!0$ invariants, as shown in~\cite{PSW} 
for~$(X_{5;(5)},\tau_{5;(5)})$. 
Whenever the real GW-invariants of~$X_{n;\a}$ are related to~$\cV_{n;\a}^{\wt\phi_{n;\a}}$ as expected,
they are also independent of the choice of the section~$s_{n;\a}$.\\

\noindent
The projective spaces $\P^{n-1}$ and~$\P^{2m-1}$ admit natural, ``half"-dimensional torus actions
that are compatible with the involutions~$\tau_n$ and $\eta_{2m}$, respectively;
these are the torus actions appearing in \cite[Section~3]{Wal}.
We determine the torus fixed loci of the induced actions on the moduli spaces of real
maps and their normal bundles in Sections~\ref{FixedLoci_subs} and~\ref{FLCsign_subs}, respectively.
The resulting equivariant localization data, which is described in Section~\ref{LocData_subs},
determines the real genus~$g$  GW-invariants of $(\P^{2m-1},\tau_{2m})$ 
and $(\P^{4m-1},\tau_{4m})$ 
 with conjugate pairs of constraints 
via the virtual equivariant localization theorem of~\cite{GP}.
Since the moduli spaces of real maps into these targets are smooth in the genus~0 case
and all torus fixed loci are contained in the smooth locus in the genus~1 case,
the classical equivariant localization theorem of~\cite{AB} suffices in these cases.
We also determine the equivariant contributions to the real GW-invariants of 
the real-orientable complete intersections~$X_{n;\a}$ that arise from
the torus fixed loci containing no maps with contracted components of positive genus;
see Theorem~\ref{EquivLocal_thm}.
In contrast to the usual approach in real GW-theory of choosing a half-graph,
our computational algorithm allows to pick  the non-fixed edges and vertices
at random (one from each conjugate pair); see Remark~\ref{EquivLocal_rmk}.
Our equivariant localization data is consistent with \cite[(3.22)]{Wal}.
We also obtain the two types of cancellations of contributions 
from some fixed loci predicted in the Calabi-Yau cases in \cite[Sections~3.2,3.3]{Wal};
see Corollary~\ref{EquivLocal_crl} and the second case of Lemma~\ref{VanV_lmm2}.\\

\noindent
By Theorem~\ref{g1EG_thm}, the genus~1 real GW-invariants of $(\P^3,\om_4,\tau_4)$ 
and $(\P^3,\om_4,\eta_4)$ are lower bounds for the enumerative counts of such curves in 
$(\P^3,J_0,\tau_4)$ and $(\P^3,J_0,\eta_4)$, respectively.
The lower bound for the number of real genus~1 degree~$d$ curves passing through  $d$~pairs of 
conjugate point insertions obtained from the equivariant localization computation
of Section~\ref{EquivLoc_sec} is~0 for $d\!=\!2$, 1 for $d\!=\!4$, and~4 for $d\!=\!6$;
see Examples~\ref{d2_eg} and~\ref{g1d4_eg} and~\cite{RealGWsApp}.
The $d\!=\!2$ number is as expected, since there are no connected degree~2 curves 
of any kind passing through two generic pairs of conjugate points in~$\P^3$.
The $d\!=\!4$ number is also not surprising, since there is only one genus~1 degree~4
curve passing through 8~generic points in~$\P^3$; see the first three paragraphs
of \cite[Section~1]{Kollar}.
By~\eref{GWvsEnum_e} and \cite{growi},  the number  of genus~1 degree~6 curves
passing through 12~generic points in~$\P^3$ is~2860.
Our signed count of~$-4$ for the real genus~1 degree~6 curves
through 6~pairs of conjugate points in~$\P^3$ is thus consistent with the complex count
and provides a non-trivial lower bound for the number of real genus~1 degree~6 curves
with 6~pairs of conjugate point insertions.\\

\noindent
In Section~\ref{EquivLocApp_subs}, 
we use the equivariant localization data of Section~\ref{LocData_subs}
to give an alternative proof of Theorem~\ref{GWsCI_thm}\ref{CIvan_it}
in the case of projective spaces
and to compare the real GW-invariants for the two involution types;
see Proposition~\ref{GWsPn_prp} below.
The genus~0 and~1 cases of both of these statements appear in
\cite[Theorem~1.9]{Teh} and \cite[Theorem~7.2]{Teh2}, respectively
(the latter presumes that genus~1 real GW-invariants of~$\P^3$ can be defined).

\begin{prp}\label{GWsPn_prp}
The genus~$g$ real GW-invariants of $(\P^{4m-1},\om_{4m},\tau_{4m})$ and 
$(\P^{4m-1},\om_{4m},\eta_{4m})$
with only conjugate pairs of insertions differ by the factor of~$(-1)^{g-1}$.
\end{prp}

\noindent 
The proof of Proposition~\ref{GWsPn_prp} extends directly 
to the real GW-invariants of the complete intersections~$(X,\phi)$ 
of Theorem~\ref{GWsCI_thm} whenever they are related 
to the sheaf $\cV_{n;\a}^{\wt\phi_{n;\a}}$ in~\eref{cVdfn_e} as expected
(in particular for $g\!=\!0$).

\subsection{Outline of the paper}
\label{introend_subs}

\noindent
Propositions~\ref{CYorient_prp} and~\ref{CIorient_prp}  are straightforward to prove;
this is done in Section~\ref{RealOrientMfld_subs}.
Explicit real orientations on the real-orientable complete intersections 
of Proposition~\ref{CIorient_prp} are described in Section~\ref{RealOrientCI_subs}.
We establish Theorem~\ref{GWsCI_thm} and note additional properties of
the real GW-invariants determined by these real orientations 
in Section~\ref{CROprop_subs}.
Theorem~\ref{g1EG_thm} is proved in Section~\ref{g1EG_sec}.
Section~\ref{EquivSetup_subs} introduces the equivariant setting relevant to the present situation.
The resulting equivariant localization data is described in Section~\ref{LocData_subs}; 
see Theorem~\ref{EquivLocal_thm}.
Its use is illustrated in Section~\ref{EquivLocApp_subs}.
Theorem~\ref{EquivLocal_thm}  is proved in Section~\ref{EquivLocalPf_sec}.

\section{Some examples}
\label{RealOrientEg_sec}

\noindent
This section establishes Propositions~\ref{CYorient_prp} and~\ref{CIorient_prp}
and thus
provides large collections of real-orientable symplectic manifolds.
We also describe specific real orientations on the   real-orientable symplectic manifolds
of Propositions~\ref{CIorient_prp}.

\subsection{Real orientable symplectic manifolds}
\label{RealOrientMfld_subs}

\noindent
We begin by deducing  Proposition~\ref{CYorient_prp} from
topological observations made in~\cite{Teh}.

\begin{proof}[{\bf\emph{Proof of Proposition~\ref{CYorient_prp}}}]
Let $(X,\om,\phi)$ be a real symplectic manifold such that $w_2(X^{\phi})\!=\!0$.
We follow the proof of \cite[Proposition~1.5]{Teh}.\\

\noindent 
Suppose $c_1(X)\!=\!2(\mu\!-\!\phi^*\mu)$ for some $\mu\!\in\!H^2(X;\Z)$.
Let $L'\!\lra\!X$ be a complex line bundle such that $c_1(L')\!=\!\mu$
and $(L,\wt\phi)$ be the real bundle pair over~$(X,\phi)$ given~by
\BE{CYorient_e3} L=L'\!\otimes_{\C}\!\ov{\phi^*L'}, \qquad
\wt\phi(v\!\otimes\!w)=w\!\otimes\!v.\EE
A continuous orientation on the fibers of $L^{\wt\phi}\!\lra\!X^{\phi}$ is 
given by the elements $v\!\otimes\!v$ with $v\!\in\!L'|_{X^{\phi}}$ nonzero.
Thus, the first requirement in~\eref{realorient_e4} is satisfied.
Since $c_1(X)\!=\!2(\mu\!-\!\phi^*\mu)$, the complex line bundles $\La_{\C}^{\top}TX$
and $L^{\otimes2}$ are isomorphic.
If in addition $H_1(X;\Q)\!=\!0$, \cite[Lemma~2.7]{Teh} then implies that 
the second requirement in~\eref{realorient_e4} is also satisfied.\\

\noindent
Suppose $X$ is Kahler, $\phi$ is anti-holomorphic, and 
$\cK_X\!=\!2([D]\!+\![\ov{\phi_*D}])$ for some divisor $D$ on~$X$.
Let $L'\!=\![-D]$ be the dual of the holomorphic line bundle corresponding to the divisor~$D$
and $(L,\wt\phi)$ be the real bundle pair over~$(X,\phi)$ defined as in~\eref{CYorient_e3}.
The first requirement in~\eref{realorient_e4} is again satisfied.
In this case, $(L,\wt\phi)$  is a holomorphic line bundle with an anti-holomorphic conjugation.
Since $\cK_X\!=\!2([D]\!+\![\ov{\phi_*D}])$, the holomorphic line bundles $\La_{\C}^{\top}TX$
and $L^{\otimes2}$ are isomorphic.
If in addition $X$ is compact, \cite[Lemma~2.6]{Teh} then implies that 
the second requirement in~\eref{realorient_e4} is also satisfied.
\end{proof}

\begin{proof}[{\bf\emph{Proof of Corollary~\ref{CIorient_crl}}}]
The assumptions imply that $X_{n;\a}\!\subset\!\P^{n-1}$ is a threefold such that  
$$w_1\big(X_{\tau_{n;\a}}^{\tau_{n;\a}}\big)=0,  \quad
\cK_{X_{n;\a}}=2\big([D]\!+\![\ov{\{\tau_{n;\a}\}_*D}]\big)
~~\hbox{with}~~ D=\cO_{\P^{n-1}}\big((n\!-\!a_1\!-\!\dots\!-\!a_{n-4})/4\big).$$
By Wu's relations \cite[Theorem~11.14]{MiSt} and the first statement above,
$w_2(X_{\tau_{n;\a}}^{\tau_{n;\a}})\!=\!0$.
The claim now follows from Proposition~\ref{CYorient_prp}\ref{KahlerOrient_it}.
\end{proof}

\noindent
We now turn to the setting of Proposition~\ref{CIorient_prp}.
We will denote by~$\fc$ the standard conjugation on~$\C^n$.
Define another conjugation  on~$\C^n$ by 
$$\fc_{\tau}'\big(v_1,\ldots,v_n\big)= 
\begin{cases}
(\bar{v}_2,\bar{v}_1,\ldots,\bar{v}_n,\bar{v}_{n-1})
,&\hbox{if}~n\!\in\!2\Z;\\
(\bar{v}_2,\bar{v}_1,\ldots,\bar{v}_{n-1},\bar{v}_{n-2},\bar{v}_n)
,&\hbox{if}~n\!\not\in\!2\Z.
\end{cases}$$
We also define a $\C$-antilinear automorphism of~$\C^{2m}$ by 
$$\fc_{\eta}\big(v_1,\ldots,v_{2m}\big)=
\big(\bar{v}_2,-\bar{v}_1,\ldots,\bar{v}_{2m},-\bar{v}_{2m-1}\big).$$ 
This automorphism has order~4.\\

\noindent
For the purposes of equivariant localization computations, 
it is convenient to consider the involution
\BE{tauprdfn_e}
\tau_n'\!: \P^n\lra\P^n, \qquad 
\tau_n'\big([Z_1,\ldots,Z_n]\big)=\big[\fc_{\tau}'(Z_1,\ldots,Z_n)\big].\EE
It is  equivalent to the involution~$\tau_n$ defined in Section~\ref{RealGWth_subs}
under the biholomorphic automorphism of~$\P^n$ given~by
\BE{PnAutdn_e}
\big[Z_1,\ldots,Z_{2m}\big]\lra
\begin{cases}
[Z_1\!+\!\fI Z_2,Z_1\!-\!\fI Z_2,\ldots,Z_{n-1}\!+\!\fI Z_n,Z_{n-1}\!-\!\fI Z_n],
&\hbox{if}~n\!\in\!2\Z;\\
\big[Z_1\!+\!\fI Z_2,Z_1\!-\!\fI Z_2,\ldots,Z_{n-2}\!+\!\fI Z_{n-1},Z_{n-2}\!-\!\fI Z_{n-1},Z_n\big],
&\hbox{if}~n\!\not\in\!2\Z.
\end{cases}\EE\\

\noindent
The involutions~$\tau_n$ and $\tau_n'$ lift to conjugations on the tautological line bundle
\BE{gandfn_e}\ga_{n-1}=\cO_{\P^{n-1}}(-1)\equiv
\big\{(\ell,v)\!\in\!\P^{n-1}\!\times\!\C^n\!:\,v\!\in\!\ell\!\subset\!\C^n\big\}\EE
as
$$\wt\tau_n(\ell,v)= \big(\tau_n(\ell),\fc(v)\big),\qquad
\wt\tau_n'\big(\ell,v\big)=\big(\tau_n'(\ell),\fc_{\tau}'(v)\big).$$
For $a\!\in\!\Z$,
we denote the induced conjugations on $\cO_{\P^{n-1}}(a)\!\equiv\!(\ga_{n-1}^*)^{\otimes a}$
by $\wt\tau_{n;1}^{(a)}$ and $\wt\tau_{n;1}'^{(a)}$, respectively,
omitting~$(a)$ for $a\!=\!1$.
The composition of $2\wt\tau_{n;1}'^{(a)}$ with the involution
$$2\cO_{\P^{n-1}}(a)\lra 2\cO_{\P^{n-1}}(a), \qquad (x,y)\lra(y,x),$$
is again an involution on $2\cO_{\P^{n-1}}(a)$; we denote it by $\wt\tau_{n;1,1}'^{(a)}$.\\

\noindent
For $a\!\in\!\Z^+$, the involution $\phi\!=\!\eta_{2m}$ lifts to 
a conjugation on $2\ga_{2m-1}^{\otimes a}$ as 
$$\wt\eta_{2m;1,1}^{(-a)}\big(\ell,v^{\otimes a},w^{\otimes a}\big)
=\big(\eta_{2m}(\ell),(\fc_{\eta}(w))^{\otimes a},
(-\fc_{\eta}(v))^{\otimes a}\big).$$
We denote the induced conjugations on 
$$2\cO_{\P^{2m-1}}(a)=\big(2\ga_{2m-1}^{\otimes a}\big)^*
\qquad\hbox{and}\qquad
\cO_{\P^{2m-1}}(2a) \equiv \La_{\C}^2\big(2\cO_{\P^{2m-1}}(a)\big)$$
by $\wt\eta_{2m;1,1}^{(a)}$ and $\wt\eta_{2m;1}^{(2a)}$, respectively.
We note that $\wt\eta_{2m;1,1}^{(2a)}\!=\!2\wt\eta_{2m;1}^{(2a)}$.\\

\noindent
For $\phi\!=\!\tau_n',\eta_{2m}$, we define
\BE{baridfn_e}
\phi:\{1,\ldots,n\}\lra\{1,\ldots,n\}
\qquad\hbox{by}\quad
\phi(i)=\begin{cases}
n,&\hbox{if}~i\!=\!n\!\not\in\!2\Z;\\
3i\!-\!4\flr{\frac{i}2}\!-\!1,&\hbox{otherwise};
\end{cases}\EE
the second case interchanges each odd integer $2k\!-\!1$ with its successor~$2k$.
For $\phi\!=\!\tau_n$, we take the bijection in~\eref{baridfn_e} to be the identity.
Let
\BE{absphi_e} |\phi|=\begin{cases} 0,&\hbox{if}~\phi\!=\!\tau_n,\tau_n';\\
1,&\hbox{if}~\phi\!=\!\eta_{2m}.\end{cases}\EE

\begin{lmm}\label{Pnses_lmm}
Let $\phi$ denote either the involution $\tau_n$ or~$\tau_n'$ on~$\P^{n-1}$
or the involution~$\eta_{2m}$ on~$\P^{2m-1}$.
Euler's exact sequence of holomorphic vector bundles 
\BE{Pnses_e}  0\lra \P^{n-1}\!\times\!\C \stackrel{f}{\lra} 
n\cO_{\P^{n-1}}(1) \stackrel{h}{\lra} T\P^{n-1}\lra 0\EE
over $\P^{n-1}$ commutes with the conjugation $\phi\!\times\!\fc$ on the first term,
the conjugation $\tnd\phi$ on the last term, and 
the conjugation~$\wt\phi_n^*$ on the middle term given~by
\begin{enumerate}[label=$\bu$,leftmargin=*]

\item $n\wt\tau_{n;1}$ if $\phi\!=\!\tau_n$,
\item $m\wt\tau_{n;1,1}'$ if $\phi\!=\!\tau_n'$ with $n\!=\!2m$,
\item $m\wt\tau_{n;1,1}'\!\oplus\!\wt\tau_{n;1}'$ if $\phi\!=\!\tau_n'$ with $n\!=\!2m\!+\!1$,
\item $m\wt\eta_{n;1,1}$ if $\phi\!=\!\eta_n$ with $n\!=\!2m$.

\end{enumerate}
\end{lmm}

\begin{proof}
For $i\!=\!1,\ldots,n$, define
\begin{gather*}
\ch{Z}_i\in H^0\big(\P^{n-1};\cO_{\P^{n-1}}(1)\big) \qquad\hbox{by}\quad
\big\{\ch{Z}_i(\ell)\big\}\big(\ell,(v_1,\ldots,v_n)\big)=v_i
~~~\forall\,(v_1,\ldots,v_n)\!\in\!\ell, \\
\z_i\!=\!\!\big(\id,(z_{i1},\ldots,z_{in})\big)\!\!:
\cU_i\!\equiv\!\{[Z_1,\ldots,Z_n]\!\in\!\P^{n-1}\!:\,Z_i\!\neq\!0\}\lra \ga_{n-1}|_{\cU_i}\,,
\quad z_{ij}=\frac{Z_j}{Z_i}=\ch{Z}_j(\z_i).
\end{gather*}
Thus, $\z_i$ is a section of~$\ga_{n-1}$ over~$\cU_i$,  $(z_{ij})_{j\neq i}$
is a  holomorphic chart on~$\cU_i$, and  
\BE{PnCoord_e}\frac{\prt}{\prt z_{ij}}
=\sum_{j'\neq i'}\frac{\prt z_{i'j'}}{\prt z_{ij}}\frac{\prt}{\prt z_{i'j'}}
=\begin{cases}
z_{ii'}^{-1}\frac{\prt}{\prt z_{i'j}}, &\hbox{if~}j\!\neq\!i';\\
-z_{ii'}^{-2}\Big(
\frac{\prt}{\prt z_{i'i}}+\sum\limits_{j'\neq i,i'}z_{ij'}
\frac{\prt}{\prt z_{i'j'}}\Big), &\hbox{if~}j\!=\!i';
\end{cases}\EE
on $\cU_i\!\cap\!\cU_{i'}$ with $i\!\neq\!i'$.
The homomorphisms~$f$ and~$h$ in~\eref{Pnses_e} are defined~by
\begin{gather*}
f\big(\ell,\la)=\big(\la\ch{Z}_1\big|_{\ell},\ldots,\la\ch{Z}_n\big|_{\ell}\big) 
\quad\forall~(\ell,\la)\!\in\!\P^{n-1}\!\times\!\C,\\
h(p_1,\ldots,p_n)=\sum_{j\neq i}
\big(p_j(\z_i(\ell))-z_{ij}(\ell)p_i(\z_i(\ell))\big) \frac\prt{\prt z_{ij}}\bigg|_{\ell}
\quad\forall\,p_1,\ldots,p_n\!\in\!\cO_{\P^{n-1}}(1)|_{\ell}\,,\ell\!\in\!\cU_i\,.
\end{gather*}
It is straightforward to check that the last homomorphism is independent
of the choice of~$i$ and that the sequence~\eref{Pnses_e} is indeed exact.\\

\noindent
Denote by $\wt\phi$ the conjugation $\wt\tau_n$ if $\phi\!=\!\tau_n$, 
the conjugation $\wt\tau_n'$ if $\phi\!=\!\tau_n'$,
and the real bundle automorphism
\BE{wtetadfn_e}\wt\phi\!:\ga_{2m-1}\lra \ga_{2m-1}, \qquad
\wt\phi(\ell,v\big)=
\big(\eta_{2m}(\ell),\fc_{\eta}(v)\big),\EE
if $\phi\!=\!\eta_{2m}$; the square of~\eref{wtetadfn_e} is~$-\Id$ on each fiber.
The effect of the involutions $\phi$ on the coordinate charts is described~by 
the relations 
\BE{phiPn_e}
\ch{Z}_{\phi(i)}\!\circ\!\wt\phi=(-1)^{|\phi|i}\fc\!\circ\!\ch{Z}_i, \qquad
\z_{\phi(i)}\!\circ\!\phi=(-1)^{|\phi|i}\wt\phi\!\circ\!\z_i.\EE
Thus, 
\BE{dphiPn_e}
z_{\phi(i)\phi(j)}\!\circ\!\phi=(-1)^{|\phi|(i+j)}\fc\!\circ\!z_{ij}, \quad
\tnd_{\ell}\phi\bigg(\frac{\prt}{\prt z_{ij}}\bigg|_{\ell}\bigg) 
=(-1)^{|\phi|(i+j)}\frac{\prt}{\prt z_{\phi(i)\phi(j)}}\bigg|_{\phi(\ell)}. \EE\\

\noindent
Denote by $\wt\phi_n$ the conjugation on $\ga_{n-1}$ dual to 
the conjugation~$\wt\phi_n^*$ in the statement of the lemma.
Thus,
\begin{equation*}
\big(\big\{(p_1,\ldots,p_n)\big\}
\big(\wt\phi_n(v_1,\ldots,v_n)\big)\big)_{\phi(i)}
=(-1)^{|\phi|i}p_{\phi(i)}\big(\wt\phi(v_i)\big)
\quad\begin{aligned}
\forall~&p_1,\ldots,p_n\!\in\!\cO_{\P^{n-1}}(1)\big|_{\ell}\,,\\
~&v_1,\ldots,v_n\!\in\!\ell.
\end{aligned}\end{equation*}
This identity is equivalent~to
$$\big((p_1,\ldots,p_n)\!\circ\!\wt\phi_n\big)_i
=(-1)^{|\phi|i}p_{\phi(i)}\!\circ\!\wt\phi
\quad\forall~p_1,\ldots,p_n\!\in\!\cO_{\P^{n-1}}(1)\big|_{\ell}.$$
The first equation in~\eref{phiPn_e} is equivalent~to
$$f\big(\phi(\ell),\bar\la\big)=
\fc\!\circ\!f(\ell,\la)\!\circ\!\wt\phi_n\!:\,n\phi(\ell)\lra\C^n
\qquad\forall~(\ell,\la)\!\in\!\P^{n-1}\!\times\!\C\,.$$
Since $\wt\phi^2\!=\!(-1)^{|\phi|}\Id$,
the remaining  equation in~\eref{phiPn_e} and~\eref{dphiPn_e} are equivalent~to
$$h\big(\fc\!\circ\!(p_1,\ldots,p_n)\!\circ\!\wt\phi_n\big)
=\tnd_{\ell}\phi\big(h(p_1,\ldots,p_n)\big)
\qquad\forall~p_1,\ldots,p_n\!\in\!\cO_{\P^{n-1}}(1)|_{\ell}\,,\ell\!\in\!\P^n.$$
The last two identities imply that $f$ and $h$ commute with the specified conjugations.
\end{proof}

\noindent
Let $\phi$ and~$\wt\phi_n^*$  be as in Lemma~\ref{Pnses_lmm} and define
$$\wt\phi_{\cK}=\begin{cases} \wt\tau_{n;1}^{(n)}, &\hbox{if}~\phi\!=\!\tau_n;\\
 \wt\tau_{n;1}'^{(n)}, &\hbox{if}~\phi\!=\!\tau_n';\\
 \wt\eta_{2m;1}^{(2m)}, &\hbox{if}~\phi\!=\!\eta_{2m}.\\
\end{cases}$$
The short exact sequence~\eref{Pnses_e} induces an isomorphism
\BE{Pnses_e2}\begin{split}
\La_{\C}^{\top}\big(T\P^{n-1},\tnd\phi\big)
&=\La_{\C}^{\top}\big(\P^{n-1}\!\times\!\C,\phi\!\times\!\fc\big)
\otimes\La_{\C}^{\top}\big(T\P^{n-1},\tnd\phi\big)\\
&\approx \La_{\C}^{\top}\big(n\cO_{\P^{n-1}}(1),\wt\phi_n^*\big)
=\big(\cO_{\P^{n-1}}(n),\wt\phi_{\cK}\big)
\end{split}\EE
of real bundle pairs over $(\P^{n-1},\phi)$.

\begin{proof}[{\bf\emph{Proof of Proposition~\ref{CIorient_prp}}}]
Under the numerical assumptions in Proposition~\ref{CIorient_prp}(1),
\BE{CIorient_e3}\begin{split}
&\binom{n}{2}-n\sum_{i=1}^ka_i+\sum_{i=1}^ka_i^2+\sum_{i<j}a_ia_j
\equiv \binom{n}{2}+0+\frac12\big(|\a|^2-|\a|\big)\\
&\qquad\qquad\qquad\equiv\frac14\Big(
\big(n\!-\!|\a|\big)^2+\big(n\!+\!|\a|\big)\big(n\!+\!|\a|\!-\!2\big)\Big)
\equiv\bigg(\frac{n\!-\!|\a|}{2}\bigg)^2 \mod2,
\end{split}\EE
where $|\a|\!=\!a_1\!+\!\ldots\!+\!a_k$.\\

\noindent
If $X_{n;\a}\!\subset\!\P^{n-1}$ is a complete intersection preserved by~$\tau_n$, 
the sequence
\BE{CIprp_e4}0\lra (TX_{n;\a},\tnd\tau_{n;\a})\lra  
(T\P^{n-1},\tnd\tau_{n})\big|_{X_{n;\a}}\lra  
\bigoplus_{i=1}^k\big(\cO_{\P^{n-1}}(a_i),\wt\tau_{n;1}^{(a_i)}\big)\big|_{X_{n;\a}}
\lra 0\EE
is a short exact sequence of real bundle pairs over $(X_{n;\a},\tau_{n;\a})$.
If $X_{2m;\a}\!\subset\!\P^{2m-1}$ is a complete intersection preserved by~$\eta_{2m}$,
the odd degrees~$a_i$ come in pairs.
There is thus a short exact sequence 
\BE{CIprp_e4b}\begin{split}
0&\lra \big(TX_{2m;\a},\tnd\eta_{2m;\a}\big) 
\lra\big(T\P^{2m-1},\tnd\eta_{2m}\big)\big|_{X_{2m;\a}}\\
&\lra  \bigoplus_{a_i\in2\Z}\!\!\!
\big(\cO_{\P^{2m-1}}(a_i),\wt\eta_{2m;1}^{(a_i)}\big)\big|_{X_{2m;\a}}
\oplus \bigoplus_{a_i'\not\in2\Z}\!\!\!
\big(\cO_{\P^{2m-1}}(a_i'),\wt\eta_{2m;1,1}^{(a_i')}\big)\big|_{X_{2m;\a}} \lra 0
\end{split}\EE
of real bundle pairs over $(X_{2m;\a},\eta_{2m;\a})$,
where the second sum is taken over one $a_i'\!=\!a_i$ for each odd-degree pair.
The two exact sequences above determine isomorphisms
\BE{CIprp_e5}\begin{split}
\La_{\C}^{\top}\big(TX_{n;\a},\tnd\tau_{n;\a}\big) \otimes
\big(\cO_{\P^{n-1}}(|\a|),\wt\tau_{n;1}^{(|\a|)}\big)
 &\approx \La_{\C}^{\top}\big(T\P^{n-1},\tnd\tau_n\big)\big|_{X_{n;\a}}\,,\\
\La_{\C}^{\top}\big(TX_{2m;\a},\tnd\eta_{2m;\a}\big) \otimes
\big(\cO_{\P^{2m-1}}(|\a|),\wt\eta_{2m;1}^{(|\a|)}\big)
 &\approx \La_{\C}^{\top}\big(T\P^{2m-1},\tnd\eta_{2m}\big)\big|_{X_{2m;\a}}
\end{split}\EE
of real bundle pairs over  $(X_{n;\a},\tau_{n;\a})$ and $(X_{2m;\a},\eta_{2m;\a})$,
respectively.\\

\noindent
Let $\nu_n(\a)\!=\!n\!-\!|\a|$.
By the assumptions of Proposition~\ref{CIorient_prp}(1), $\nu_n(\a)\!\in\!2\Z$.
By the assumptions of Proposition~\ref{CIorient_prp}(2), $\nu_{2m}(\a)\!\in\!4\Z$.
Thus, the real bundle pairs
\BE{CIprp_e6}\begin{split}
\big(L_{\tau;n;\a},\wt\phi_{\tau;n;\a}\big) 
\equiv \big(\cO_{\P^{n-1}}(\nu_n(\a)/2),\wt\tau_{n;1}^{(\nu_n(\a)/2)}\big)
&\lra \big(\P^{n-1},\tau_n\big),\\
\big(L_{\eta;2m;\a},\wt\phi_{\eta;2m;\a}\big) 
\equiv \big(\cO_{\P^{2m-1}}(\nu_{2m}(\a)/2),\wt\eta_{2m;1}^{(\nu_{2m}(\a)/2)}\big)
&\lra \big(\P^{2m-1},\eta_{2m}\big)
\end{split}\EE
are well-defined.
By~\eref{Pnses_e2} and~\eref{CIprp_e5}, 
\BE{CIprp_e7}\begin{split}
\La_{\C}^{\top}\big(TX_{n;\a},\tnd\tau_{n;\a}\big) &\approx 
\Big(\!\big(L_{\tau;n;\a},\wt\phi_{\tau;n;\a}\big) \big|_{X_{n;\a}}
\Big)^{\otimes2},\\
\La_{\C}^{\top}\big(TX_{2m;\a},\tnd\eta_{2m;\a}\big) &\approx 
\Big(\!\big(L_{\eta;2m;\a},\wt\phi_{\eta;2m;\a}\big)
\big|_{X_{2m;\a}} \Big)^{\otimes2}\,.
\end{split}\EE
By~\eref{CIorient_e3},
\BE{CIprp_e9a}\begin{split}
w_2\big(X_{n;\a}^{\tau_{n;\a}}\big)
&=\bigg(\binom{n}{2}-n\sum_{i=1}^ka_i+\sum_{i=1}^ka_i^2+\sum_{i<j}a_ia_j\bigg)x^2\\
&=\bigg(\frac{n\!-\!|\a|}{2}\bigg)^2x^2
=w_1\big(L_{\tau;n;\a}^{\wt\phi_{\tau;n;\a}}\big)^2,
\end{split}\EE
where $x$ is the restriction of the generator of $H^1(\R\P^{n-1};\Z_2)$ to 
$X_{n;\a}^{\tau_{n;\a}}$.
Since $X_{2m;\a}^{\eta_{2m;\a}}\!=\!\eset$,
\BE{CIprp_e9b}w_2\big(X_{2m;\a}^{\eta_{2m;\a}}\big)=0=
w_1\big(L_{\eta;2m;\a}^{\wt\phi_{\eta;2m;\a}}\big)^2.\EE
Therefore,  $(X_{n;\a},\om_{n;\a},\tau_{n;\a})$ and $(X_{2m;\a},\om_{2m;\a},\eta_{2m;\a})$
are real-orientable symplectic manifolds
under the assumptions in~(1) and~(2), respectively,
of Proposition~\ref{CIorient_prp}.
\end{proof}

\subsection{Real orientations on complete intersections}
\label{RealOrientCI_subs}

\noindent
We next describe a real orientation on each complete intersection $(X_{n;\a},\tau_{n;\a})$
and $(X_{2m;\a},\eta_{2m;\a})$ of  Proposition~\ref{CIorient_prp} with
$$(L,\wt\phi)=\big(L_{\tau;n;\a},\wt\phi_{\tau;n;\a}\big) 
\quad\hbox{and}\quad
(L,\wt\phi)=\big(L_{\eta;2m;\a},\wt\phi_{\eta;2m;\a}\big),$$
respectively, as in~\eref{CIprp_e6}.
These real orientations are used to orient the normal bundles to torus fixed loci
in Section~\ref{EquivLoc_sec}.\\

\noindent
A homotopy class of isomorphisms as in~\eref{realorient_e4}  is determined by 
the corresponding isomorphism~\eref{CIprp_e7}.
By~\eref{CIprp_e9a} and~\eref{CIprp_e9b}, $(L,\wt\phi)$ satisfies the first condition 
in~\eref{realorient_e4}.
Since $X_{2m;\a}^{\eta_{2m;\a}}\!=\!\eset$, this determines a real orientation on
the complete intersections  $(X_{2m;\a},\eta_{2m;\a})$ of  Proposition~\ref{CIorient_prp}(2).\\

\noindent
In the case of the complete intersections $(X_{n;\a},\tau_{n;\a})$ of 
Proposition~\ref{CIorient_prp}(1), it remains to specify a spin structure 
as in~\ref{spin_it2}.

\begin{lmm}\label{CIorient_lmm}
Let $k\!\in\!\Z^{\ge0}$ and $\a\!\equiv\!(a_1,\ldots,a_k)\!\in\!(\Z^+)^k$.
If $\a$ satisfies the second condition in~\eref{CIorient_e1}, then
\BE{CIorientLmm_e}\big|\big\{i\!=\!1,\ldots,k\!:\,a_i\!\not\in\!2\Z\big\}\big|
\equiv \sum_{i=1}^ka_i\qquad\mod4.\EE
\end{lmm}

\begin{proof}
Since 
$$a_i^2\equiv \begin{cases}
0,&\hbox{if}~a_i\!\in\!2\Z;\\
1,&\hbox{if}~a_i\!\not\in\!2\Z; 
\end{cases} \mod4,$$
the left-hand side of the second equation in~\eref{CIorient_e1} equals
to the left-hand side of~\eref{CIorientLmm_e} modulo~4.
This establishes the claim.
\end{proof}

\noindent
For each $a\!\in\!\Z$, let
$$\cO_{\R\P^{n-1}}(a)=\cO_{\P^{n-1}}(a)^{\wt\tau_{n;1}^{(a)}}\lra \R\P^{n-1}\,.$$
The canonical orientation on
$$ \cO_{\R\P^{n-1}}(2)=\cO_{\R\P^{n-1}}(1)^{\otimes2}\lra \R\P^{n-1}$$
determines a homotopy class of isomorphisms
\BE{RealOrientCI_e3} \cO_{\R\P^{n-1}}(2a)\approx \R\P^1\!\times\!\R \qquad\hbox{and}\qquad
\cO_{\R\P^{n-1}}(2a\!+\!1)\approx\cO_{\R\P^{n-1}}(1)\EE
of real line bundles over $\R\P^{n-1}$;
the second isomorphism treats the last factor of $\cO_{\R\P^{n-1}}(1)$ as the remainder.\\

\noindent
Furthermore, the real vector bundle 
$$4\cO_{\R\P^{n-1}}(1)\lra \R\P^{n-1}$$
has a canonical orientation and spin structure. 
They are obtained by taking any trivialization of the first $2\cO_{\R\P^{n-1}}(1)$ 
over a loop and taking the same trivialization over the last $2\cO_{\R\P^{n-1}}(1)$.
These orientation and spin structure are invariant under the interchange of the first pair of 
the line bundles $\cO_{\R\P^{n-1}}(1)$ with the last.
The interchange of the components $\cO_{\R\P^{n-1}}(1)$ within 
the same pair reverses the canonical orientation.
By \cite[Lemma~3.11]{RealGWsII}, it also flips  the spin structure
(i.e.~the homotopy classes of the canonical trivializations and their compositions with
this interchange on the left do not differ 
by the compositions with an orientation-reversing diffeomorphism
of~$\R^n$ on the right).\\

\noindent
We now return to the setting of Proposition~\ref{CIorient_prp}(1). 
Let 
$$\ell_0(\a)=\big|\big\{i\!=\!1,\ldots,k\!:\,a_i\!\in\!2\Z\big\}\big| 
\qquad\hbox{and}\qquad
\ell_1(\a)=\big|\big\{i\!=\!1,\ldots,k\!:\,a_i\!\not\in\!2\Z\big\}\big|$$
be the number of even entries in~$\a$ and the number of odd entries, respectively.
The short exact sequences~\eref{Pnses_e} and~\eref{CIprp_e4} of real bundle pairs 
determine a homotopy class of isomorphisms
\BE{RealOrientCI_e5}\begin{split}
& \big(X_{n;\a}^{\tau_{n;\a}}\!\times\!\R\big)\oplus
\big(TX_{n;\a}^{\tau_{n;\a}}\!\oplus\!2(L^*)^{\wt\phi^*}\big|_{X_{n;\a}^{\tau_{n;\a}}}\big)
\oplus 
\bigoplus_{i=1}^k\cO_{\R\P^{n-1}}(a_i)\big|_{X_{n;\a}^{\tau_{n;\a}}}\\
&\hspace{2.5in}
\approx n\cO_{\R\P^{n-1}}(1)\big|_{X_{n;\a}^{\tau_{n;\a}}}
\oplus 2(L^*)^{\wt\phi^*}\big|_{X_{n;\a}^{\tau_{n;\a}}}
\end{split}\EE
of real vector bundles over $X_{n;\a}^{\tau_{n;\a}}$.
If $n\!-\!|\a|\!\in\!4\Z$, \eref{RealOrientCI_e3} and~\eref{RealOrientCI_e5} 
determine a homotopy class of isomorphisms
\BE{RealOrientCI_e9a}\begin{split}
& \big(X_{n;\a}^{\tau_{n;\a}}\!\times\!\R\big)\oplus
\big(TX_{n;\a}^{\tau_{n;\a}}\!\oplus\!2(L^*)^{\wt\phi^*}\big|_{X_{n;\a}^{\tau_{n;\a}}}\big)
\oplus \big(X_{n;\a}^{\tau_{n;\a}}\!\times\!\R^{\ell_0(\a)}\big)
\oplus \ell_1(\a)\cO_{\R\P^{n-1}}(1)\big|_{X_{n;\a}^{\tau_{n;\a}}}\\
&\hspace{1in}
\approx \big(n\!-\!\ell_1(\a)\big)\cO_{\R\P^{n-1}}(1)\big|_{X_{n;\a}^{\tau_{n;\a}}}
\oplus \ell_1(\a)\cO_{\R\P^{n-1}}(1)\big|_{X_{n;\a}^{\tau_{n;\a}}}
\oplus \big(X_{n;\a}^{\tau_{n;\a}}\!\times\!\R^2\big).
\end{split}\EE
If $n\!-\!|\a|\!\not\in\!4\Z$, \eref{RealOrientCI_e3} and~\eref{RealOrientCI_e5} 
determine a homotopy class of isomorphisms
\BE{RealOrientCI_e9b}\begin{split}
& \big(X_{n;\a}^{\tau_{n;\a}}\!\times\!\R\big)\oplus
\big(TX_{n;\a}^{\tau_{n;\a}}\!\oplus\!2(L^*)^{\wt\phi^*}\big|_{X_{n;\a}^{\tau_{n;\a}}}\big)
\oplus \big(X_{n;\a}^{\tau_{n;\a}}\!\times\!\R^{\ell_0(\a)}\big)
\oplus \ell_1(\a)\cO_{\R\P^{n-1}}(1)\big|_{X_{n;\a}^{\tau_{n;\a}}}\\
&\hspace{1.8in}
\approx \big(n\!+\!2\!-\!\ell_1(\a)\big)\cO_{\R\P^{n-1}}(1)\big|_{X_{n;\a}^{\tau_{n;\a}}}
\oplus \ell_1(\a)\cO_{\R\P^{n-1}}(1)\big|_{X_{n;\a}^{\tau_{n;\a}}}.
\end{split}\EE
The first terms on the right-hand sides of~\eref{RealOrientCI_e9a} and~\eref{RealOrientCI_e9b}
correspond to the first $n\!-\!\ell_1(\a)$ and $n\!+\!2\!-\!\ell_1(\a)$ summands
of the first term on the right-hand sides of~\eref{RealOrientCI_e5}.\\

\noindent
By Lemma~\ref{CIorient_lmm}, the ranks of the first summands
on the right-hand sides of~\eref{RealOrientCI_e9a} and~\eref{RealOrientCI_e9b}
are divisible by~4.
Thus, \eref{RealOrientCI_e9a} and~\eref{RealOrientCI_e9b} determine 
a homotopy class of isomorphisms
\BE{RealOrientCI_e11}\begin{split}
& \big(X_{n;\a}^{\tau_{n;\a}}\!\times\!\R\big)\oplus
\big(TX_{n;\a}^{\tau_{n;\a}}\!\oplus\!2(L^*)^{\wt\phi^*}\big|_{X_{n;\a}^{\tau_{n;\a}}}\big)
\oplus \big(X_{n;\a}^{\tau_{n;\a}}\!\times\!\R^{\ell_0(\a)}\big)
\oplus \ell_1(\a)\cO_{\R\P^{n-1}}(1)\big|_{X_{n;\a}^{\tau_{n;\a}}}\\
&\hspace{2.5in}
\approx \big(X_{n;\a}^{\tau_{n;\a}}\!\times\!\R^{n+2-\ell_1(\a)}\big)
\oplus \ell_1(\a)\cO_{\R\P^{n-1}}(1)\big|_{X_{n;\a}^{\tau_{n;\a}}}
\end{split}\EE
over every loop in $X_{n;\a}^{\tau_{n;\a}}$.
Since the  real bundle pair $(L,\wt\phi)$ satisfies~\eref{realorient_e4},
the real orientable vector bundle \hbox{$TX_{n;\a}^{\tau_{n;\a}}\!\oplus\!2(L^*)^{\wt\phi^*}$} 
admits a spin structure.
By the first two statements of \cite[Lemma~3.11]{RealGWsII},
it is determined by~\eref{RealOrientCI_e11} and the orientation on~$TX^{\phi}$
specified by the first isomorphism in~\eref{CIprp_e7}.

\subsection{The canonical orientations of the moduli spaces}
\label{CROprop_subs}

\noindent
By \cite[Theorem~1.3]{RealGWsI}, the real orientations on $(X_{n;\a},\tau_{n;\a})$
and $(X_{2m;\a},\eta_{2m;\a})$ constructed in Section~\ref{RealOrientCI_subs}
induce orientations on the moduli spaces of real maps
into these real symplectic manifolds if $n\!-\!k\!\in\!2\Z$ (so that the complex dimensions 
of these Kahler manifolds are odd).
The construction of these real orientations involves some implicit choices
(which are listed explicitly in the statement of Theorem~\ref{GWsCI_thm}). 
Below we describe the effect of these choices on the real orientations,
the induced orientations on the moduli spaces of real maps,
and the resulting GW-invariants of $(X_{n;\a},\tau_{n;\a})$ and $(X_{2m;\a},\eta_{2m;\a})$. 
We use this description to establish Theorem~\ref{GWsCI_thm}.
Proposition~\ref{CROprop_prp} notes additional properties of the  real orientations
of Section~\ref{RealOrientCI_subs}.\\

\noindent
Throughout this section, we denote by $(X,\phi)$ either 
a fixed complete intersection $(X_{n;\a},\tau_{n;\a})$ in $(\P^{n-1},\tau_n)$
cut out by a real holomorphic bundle section~$s_{n;\a}$  
or a fixed complete intersection $(X_{2m;\a},\eta_{2m;\a})$ in $(\P^{2m-1},\eta_{2m})$
cut out by a real holomorphic bundle section~$s_{2m;\a}$.
Two such sections cut out the same complete intersections if and only if they differ by a
real holomorphic automorphism of the holomorphic vector bundle
\BE{cLdfn_e}\cL_{n;\a}\equiv\bigoplus_{i=1}^k\cO_{\P^{n-1}}(a_i)\lra\P^{n-1}\,.\EE
A holomorphic endomorphism~$\Phi$ of~\eref{cLdfn_e} corresponds to $k^2$~elements
\BE{vphij_e}\vph_{ij}\in H^0\big(\P^{n-1};\cO_{\P^{n-1}}(a_i\!-\!a_j)\big), \qquad
i,j=1,\ldots,k.\EE
In particular, $\vph_{ij}\!=\!0$ if $a_i\!<\!a_j$ and $\vph_{ij}\!\in\!\C$ if $a_i\!=\!a_j$.
This endomorphism  is thus invertible if and only if
the (constant) matrix $(\vph_{ij}')_{ij}$ given~by
$$\vph_{ij}'=\begin{cases}\vph_{ij},&\hbox{if}~a_i\!=\!a_j;\\
0,&\hbox{if}~a_i\!\neq\!a_j;
\end{cases}$$
is invertible.
If $\phi\!=\!\tau_{n;\a}$, $\Phi$ is real if each $\vph_{ij}$ in~\eref{vphij_e} is real.
The homotopy classes of real automorphisms of~\eref{cLdfn_e} in this case are thus generated
by negating the individual components of~\eref{cLdfn_e}.
If $\phi\!=\!\eta_{2m;\a}$, the odd-degree line bundles in~\eref{vphij_e} are paired~up.
Every real holomorphic automorphism of the associated rank~2 real bundle pair is homotopic to
the identity through  real holomorphic automorphisms.
The homotopy classes of real automorphisms of~\eref{cLdfn_e} in this case are thus generated~by
negating the individual even-degree components of~\eref{cLdfn_e}.

\begin{rmk}\label{etaAuts_rmk}
The real bundle pairs $(2\cO_{\P^{2m-1}}(a),\wt\eta_{2m;1,1}^{(a)})$ with $a\!\not\in\!2\Z$
admit continuous real automorphisms not homotopic to the identity.
Their homotopy classes are characterized by their restrictions to an equator 
$S^1\!\subset\!\C^*$ inside of 
a real linearly embedded $\P^1\!\subset\!\P^{n-1}$ being homotopic to 
the automorphism
$$(\al,\be)\lra \big(z^{-1}\be,z\al\big).$$
This homomorphism extends continuously, but not holomorphically, over~$\P^{n-1}$.\\
\end{rmk}

\noindent
Whenever the choices implicitly made in Section~\ref{RealOrientCI_subs} affect
the resulting real orientation on $(X_{n;\a},\phi)$
constructed in Section~\ref{RealOrientCI_subs},
their effects on the induced orientations of the  moduli spaces 
of real maps are straightforward to determine 
using  oriented symmetric half-surfaces (or \sf{sh-surfaces})
as in \cite{XCapsSetup} and~\cite[Section~3.2]{RealGWsII}.
An sh-surface~$(\Si^b,c)$ contains ordinary boundary components and \sf{crosscaps}
(boundary components with an antipodal involution) and doubles to a symmetric surface~$(\Si,\si)$.
The parity of the total number of the boundary components of~$\Si^b$
is the parity of $g\!+\!1$, where $g$ is the genus of~$\Si$.\\

\noindent
The first choice made in Section~\ref{RealOrientCI_subs} is
a real holomorphic parametrization of $(\P^{n-1},\tau_n)$ and $(\P^{2m-1},\eta_{2m})$;
this in particular includes the ordering of the homogeneous components for the purposes
of Lemma~\ref{Pnses_lmm}.
Since the group $\Aut(\P^1,\eta)$ is connected,  
the group $\Aut(\P^{2m-1},\eta_{2m})$ is also connected.
A path in $\Aut(\P^{2m-1},\eta_{2m})$ lifts to 
a path of holomorphic automorphisms of $m(2\cO_{\P^{2m-1}}(1),\wt\eta_{2m;1,1}^{(1)})$.
Since the group of holomorphic automorphisms of this real bundle pair over the identity is connected,
it follows that the real orientation of Section~\ref{RealOrientCI_subs}
is independent of the choice of 
real holomorphic parametrization of~$(\P^{2m-1},\eta_{2m})$.\\

\noindent
The group $\Aut(\P^1,\tau)$ has two connected components;
the negation of a homogenous coordinate is not homotopic to the identity.
Since such an automorphism reverses the orientation of $\R\P^{n-1}$ whenever $n\!\in\!2\Z^+$,
this implies that  $\Aut(\P^{n-1},\tau_n)$ has precisely two connected components for any 
$n\!\ge\!2$.
Negating a homogenous coordinate of~$\P^n$ changes 
\begin{enumerate}[label=(C\arabic*),leftmargin=*]

\setcounter{enumi}{-1}

\item\label{Ch_it} the homotopy class of the surjection~$h$ in~\eref{Pnses_e},

\item\label{Ceta_it}  the homotopy class of the isomorphism in~\eref{realorient_e4} induced by
the first isomorphism in~\eref{CIprp_e7} over every real loop in~$X_{n;\a}$,
and 

\item\label{Corient_it} the orientation on $TX_{n;\a}^{\tau_{n;\a}}\!\oplus\!2(L^*)^{\wt\phi^*}$
induced by the isomorphism~\eref{RealOrientCI_e11}. 

\setcounter{svcnt}{\value{enumi}}
\end{enumerate}
By the third statement of \cite[Lemma~3.11]{RealGWsII}, such an interchange also changes
\begin{enumerate}[label=(C\arabic*),leftmargin=*]

\setcounter{enumi}{\value{svcnt}}

\item\label{Cspin_it} the resulting spin of $TX_{n;\a}^{\tau_{n;\a}}\!\oplus\!2(L^*)^{\wt\phi^*}$
over every loop in $X_{n;\a}^{\tau_{n;\a}}$ not contractible in~$\R\P^{n-1}$.

\end{enumerate}
For a real map~$u$ from an sh-surface~$\Si^b$, the effects of these changes on the trivializations
over the boundary components of~$\Si^b$ are to flip 
\begin{enumerate}[label=(E\arabic*),leftmargin=*]

\item\label{Eeta_it} the homotopy type of trivializations of $u^*(TX_{n;\a},\tnd\tau_{n;\a})$
over each crosscap,

\item\label{Eorient_it} the orientation of trivializations of $u^*TX_{n;\a}^{\tau_{n;\a}}$
over each ordinary boundary component,

\item\label{Espin_it} the spin of trivializations of $u^*TX_{n;\a}^{\tau_{n;\a}}$
over each ordinary boundary component to  which the restriction of $u$ is homotopically
non-trivial as a map to~$\R\P^{n-1}$.

\end{enumerate}
By \cite[Propositions~4.1,4.2]{BHH},
the parity of the number of the components in~\ref{Espin_it}
is the parity of the degree~$d$ of~$u$.

\begin{proof}[{\bf\emph{Proof of Theorem~\ref{GWsCI_thm}}}]
Let $\wt\eta_{2m}$ be an anti-holomorphic conjugation on
the vector bundle~$\cL_{n;\a}$ in~\eref{cLdfn_e} lifting 
the involution $\eta_{2m}$ on~$\P^{2m-1}$ and
$$H^0\big(\P^{2m-1};\cL_{n;\a}\big)^{\wt\eta_{2m}}\subset
H^0\big(\P^{2m-1};\cL_{n;\a}\big)$$
be the subspace of real holomorphic sections.
Since the fixed locus of the involution~$\eta_{2m}$ on~$\P^{2m-1}$ is empty,
the subspace 
$$\big\{(s,p)\!\in\!H^0\big(\P^{2m-1};\cL_{n;\a}\big)^{\wt\eta_{2m}}\!\times\!\P^{n-1}\!:\,
s(p)\!=\!0,\,\rk_{\C}\na s|_p\!<\!k\big\}$$
has complex codimension $n$ in $H^0(\P^{2m-1};\cL_{n;\a})^{\wt\eta_{2m}}\!\times\!\P^{n-1}$.
Thus, its projection to the first component has complex codimension~one
and the space of regular real sections of~\eref{cLdfn_e} is path-connected.
Along with the conclusions above Remark~\ref{etaAuts_rmk} and 
in the paragraph concerning $\Aut(\P^{2m-1},\eta_{2m})$, 
this implies that the real orientation on $(X_{2m;\a},\eta_{2m;\a})$
constructed in Section~\ref{RealOrientCI_subs}
and the resulting real GW-invariants are independent 
of the real parametrization of $(\P^{2m-1},\eta_{2m})$, 
the ordering of the line bundles associated with the complete intersection,
and the section~$s_{2m;\a}$.\\

\noindent
The first vanishing claim of Theorem~\ref{GWsCI_thm}\ref{CIvan_it} in this case follows
from Lemma~\ref{etamaps_lmm} below
(its $g\!=\!0$ case is contained in \cite[Lemma~1.9]{Teh}).
If $a_i\!\in\!2\Z$ for some~$i$, replacing the component~$s_i$ of  $s_{n;\a}$ by~$-s_i$
results in the changes in~\ref{Ceta_it} and thus in~\ref{Eeta_it}.
If $g\!\in\!2\Z$, this replacement thus changes 
the orientation of the moduli space of real maps to $(X_{2m;\a},\eta_{2m;\a})$ and 
the sign of the corresponding real GW-invariants.
However, by the previous paragraph, these invariants are independent of the choice of
the section~$s_{n;\a}$.
This establishes the second vanishing claim of Theorem~\ref{GWsCI_thm}\ref{CIvan_it}
for $\phi\!=\!\eta_{2m;\a}$.\\

\noindent
We next consider the case $\phi\!=\!\tau_{n;\a}$.
A real holomorphic reparametrization  of $(\P^{n-1},\tau_n)$ induced 
by a linear automorphism~$\vph$ of~$\C^n$ determines 
a commutative diagram of the~form 
\BE{CROdiag_e}\begin{split}
\xymatrix{\ov\fM_{g,l}(X,d)^{\phi}  \ar[d]\ar[rr]^>>>>>>>>>>>{\ev}&&
X^l  \ar[d]\ar@{^{(}->}[rr]&& \big(\P^{n-1}\big)^l \ar[d]\\
\ov\fM_{g,l}(X',d)^{\phi'}\ar[rr]^>>>>>>>>>>>{\ev}&& 
X'^l  \ar@{^{(}->}[rr]&&  \big(\P^{n-1}\big)^l}
\end{split}\EE
with $(X',\phi')$ denoting the complete intersection cut out by the section $s_{n;\a}'$
obtained from~$s_{n;\a}$ by a suitable transform.
The middle and right vertical arrows are orientation-preserving 
with respect to the canonical complex orientations on their domains and targets.
The left vertical arrow is orientation-preserving with respect
to the orientations of Section~\ref{RealOrientCI_subs} 
differing by the reparametrization of the middle term in~\eref{Pnses_e}
by~$\vph$.
The GW-invariants of Theorem~\ref{GWsCI_thm} are the intersection numbers of
the first horizontal arrows with the cycles represented by the constraints.\\

\noindent
By the commutativity of the first square in~\eref{CROdiag_e},
the GW-invariants of Theorem~\ref{GWsCI_thm}  are thus invariant 
under real holomorphic reparametrizations of $(\P^{n-1},\tau_n)$.
This in particular establishes Theorem~\ref{GWsCI_thm}\ref{CIindep_it} for $k\!=\!0$.
Negating one of the $\cO_{\P^{n-1}}(1)$-factors  in~\eref{Pnses_e} or 
one of the odd-degree components~$s_i$ of~$s_{n;\a}$ affects \ref{Ceta_it}-\ref{Cspin_it}
and the orientation of $\ov\fM_{g,l}(X,d)^{\phi}$ through~\ref{Eeta_it}-\ref{Espin_it}.
Each of the last three changes by itself would reverse the orientation of $\ov\fM_{g,l}(X,d)^{\phi}$.
Since the parity of the number of changes~\ref{Eeta_it} and~\ref{Eorient_it} is that of $g\!+\!1$
and of~\ref{Espin_it} is that of~$d$,
negating one of the $\cO_{\P^{n-1}}(1)$-factors in~\eref{Pnses_e} or 
one of the odd-degree components~$s_i$ of~$s_{n;\a}$
preserves the orientation of $\ov\fM_{g,l}(X,d)^{\phi}$ 
if $d\!-\!g\!\not\in\!2\Z$ and reverses it otherwise.
This establishes Theorem~\ref{GWsCI_thm}\ref{CIindep_it} under the assumptions that
$a_i\!\not\in\!2\Z$ for all~$i$ and $d\!-\!g\!\not\in\!2\Z$.
Since
negating an even-degree component~$s_i$ of~$s_{n;\a}$ affects \ref{Ceta_it} and~\ref{Corient_it} only
and the orientation of $\ov\fM_{g,l}(X,d)^{\phi}$ through~\ref{Eeta_it} and~\ref{Eorient_it},
this operation
preserves the orientation of $\ov\fM_{g,l}(X,d)^{\phi}$ 
if $g\!\not\in\!2\Z$ and reverses it otherwise.
Combined with the previous observation, 
this establishes Theorem~\ref{GWsCI_thm}\ref{CIindep_it} under the assumption
$g\!\not\in\!2\Z$ for all~$i$ and $d\!\in\!2\Z$.\\

\noindent
We now turn to Theorem~\ref{GWsCI_thm}\ref{CIvan_it} with $\phi\!=\!\tau_{n;\a}$.
If $d\!-\!g\!\in\!2\Z$, negating one of the $\cO_{\P^{n-1}}(1)$-factors  in~\eref{Pnses_e} 
changes the orientation of $\ov\fM_{g,l}(X,d)^{\phi}$ and thus the sign of
the real GW-invariants of~$(X,\phi)$.
If $a_i\!\in\!2\Z$ for some~$i$ and $g\!\in\!2\Z$, replacing $s_i$ by $-s_i$
changes the sign of the real GW-invariants of~$(X,\phi)$.
If the real genus~$g$ degree~$d$ GW-invariants of~$(X,\phi)$ are invariant under these changes,
then they must vanish.\\
 
\noindent
It remains to consider the $g\!=\!0$ case of  Theorem~\ref{GWsCI_thm}\ref{CIindep_it}.
The real GW-invariants of $(X,\phi)$ are then given by cupping the constraints 
with the Euler class of the bundle~\eref{cVdfn_e} to $\ov\fM_{g,l}(\P^{n-1},d)^{\tau_n}$.
The proof of \cite[Theorem~3]{PSW} establishes this statement for $(n,\a)\!=\!(5,(5))$ and $d\!\not\in\!2\Z$,
but its principles apply in general (as long as $g\!=\!0$).
The real GW-invariants of $(X,\phi)$ can then be computed using the equivariant localization
theorem of~\cite{AB} as in Section~\ref{LocData_subs}.
If $d\!\in\!2\Z$ (i.e.~$d\!-\!g\!\in\!2\Z$) or $a_i\!\in\!2\Z$ for some~$i$,
then all torus fixed loci contribute zero to these invariants; see Lemma~\ref{VanV_lmm2}.
Therefore, the $g\!=\!0$ real GW-invariants vanish in these cases and are
in particular independent of all choices implicitly made in Section~\ref{RealOrientCI_subs}.
The same reasoning applies whenever the real genus~$g$ degree~$d$ GW-invariants
can be similarly related to $\ov\fM_{g,l}(\P^{n-1},d)^{\tau_n}$ and 
either $d\!-\!g\!\in\!2\Z$ or $a_i\!\in\!2\Z$ for some~$i$ and $g\!\in\!2\Z$. 
\end{proof}

\begin{rmk}\label{Cobordism_rmk}
The independence of the real GW-invariants of $(X,\phi)$ in the $\phi\!=\!\eta_{n;\a}$ case
is established by taking the homotopy between two moduli spaces of real maps 
induced by a generic path between two regular real sections of~\eref{cLdfn_e};
it consists of regular sections in this case.
Such a path need not exist in the $\phi\!=\!\tau_{n;\a}$ case, but 
a cobordism between the moduli space would satisfy.
It would pass through spaces of real maps into hypersurfaces with isolated
real nodal points.
Unfortunately, we do not see at this point a notion of  a moduli space which
would be suitable for constructing the desired cobordism.
\end{rmk}

\begin{lmm}\label{etamaps_lmm}
Suppose $m,d\!\in\!\Z^+$ and $g,l\!\in\!\Z^{\ge0}$.
If $d\!-\!g\!\in\!2\Z$, then
$$\ov\fM_{g,l}\big(\P^{2m-1},d\big)^{\eta_{2m}}=\eset\,.$$
\end{lmm}

\begin{proof}
It is sufficient to establish this statement for $l\!=\!0$ and $m\!=\!1$
(after composition with a projection to a generic real line).
By the Riemann-Hurwitz formula \cite[p219]{GH},
a genus~$g$ degree~$d$ cover of~$\P^1$ has $2(d\!+\!g\!-\!1)$ branched points in~$\P^1$,
counted with multiplicity.
Any such cover is determined by the branched points and some combinatorial data.
Suppose some combinatorial data is compatible with the involution~$\eta_2$ on~$\P^1$ 
and some (necessarily fixed-point-free) involution on the domain.
We take the limit of such real covers by bringing all of the branched points 
to a pair of conjugate points.
The restriction of the limiting map to each non-contracted component of the domain
then has no branched points;
any such component is thus a $\P^1$ and the degree of the limiting map is~1.
Denote by $g_0$ the sum of the geometric genera of the contracted components
and by $N$ the number of nodes of the domain of the limiting map.
Thus,
\BE{etamaps_e3} g=d-N+1+g_0.\EE
Since the fixed locus of~$(\P^1,\eta)$ is empty, $g_0,N\!\in\!2\Z$.
The claim thus follows from~\eref{etamaps_e3}.
\end{proof}

\noindent
Since the involutions $\tau_n$ and~$\tau_n'$ on~$\P^{n-1}$ are related 
by the automorphism~\eref{PnAutdn_e}, 
the construction of real orientations in Section~\ref{RealOrientCI_subs}
applies with only minor changes to complete intersections invariant under 
the involutions~$\tau_n'$.
For such complete intersections, some of the odd-degree sections~$s_i$ may be paired up
(the sections corresponding to the real bundle pairs $(2\cO_{\P^{n-1}}(a),\wt\tau_{n;1,1}'^{(a)})$).
However, the ordering of the sections within each such pair is not fixed
(because the automorphism~$\wt\phi$ in the proof of Lemma~\ref{Pnses_lmm}
is of order~2 in this case).
The same invariance considerations as in the $\tau_n$ case apply
with the involution~$\tau_n'$.\\ 

\noindent
We next note some properties of the real orientations constructed in Section~\ref{RealOrientCI_subs}.

\begin{prp}\label{CROprop_prp}
\begin{enumerate}[label=(\arabic*),leftmargin=*]

\item 
Let $m,n\!\in\!\Z^+$, $k\!\in\!\Z^{\ge0}$, and $\a\!\equiv\!(a_1,\ldots,a_k)\!\in\!(\Z^+)^k$.
If the real orientations on $(X_{n;\a},\tau_{n;\a})$ and $(X_{2m;\a},\eta_{2m;\a})$
constructed in Section~\ref{RealOrientCI_subs} are independent of the choices made,
then they are also invariant under the inclusions
\BE{CROprop_e2}\big(\P^{n-1},\tau_n\big) \lra \big(\P^n,\tau_{n+1}\big)
\quad\hbox{and}\quad
\big(\P^{2m-1},\eta_{2m}\big)\lra \big(\P^{2m+1},\eta_{2m+2}\big)\EE
as coordinate subspaces.
If the induced orientation on the moduli space of genus~$g$ degree~$d$ real maps
is independent of these choices,
then it is also invariant under the inclusions~\eref{CROprop_e2}.

\item Let $m\!\in\!\Z^+$. The signed count of real lines through a pair of conjugate 
points in  $(\P^{2m-1},\tau_{2m})$ with respect to the real orientation of 
Section~\ref{RealOrientCI_subs} is given by~\eref{d1tau_e0}.

\end{enumerate}
\end{prp}

\begin{proof} (1) The first inclusion in~\eref{CROprop_e2} replaces 
$X_{n;\a}$ with $X_{n+1;\a'}$, where $\a'$ is tuple obtained from~$\a$
by adding~1 as the last component.
The only effect of this change on \eref{RealOrientCI_e9a}-\eref{RealOrientCI_e11} 
is to increase the coefficients~$\ell_1(\a)$ in front of $\cO_{\R\P^{n-1}}(1)\big|_{X_{n;\a}^{\tau_{n;\a}}}$
by~$1$.
There is no effect on the homotopy class of the isomorphism in~\eref{realorient_e4} induced by
the first isomorphism in~\eref{CIprp_e7}
or on the spin structure on $TX_{n;\a}^{\tau_{n;\a}}\!\oplus\!2(L^*)^{\wt\phi^*}$.\\

\noindent
The second inclusion in~\eref{CROprop_e2} replaces 
$X_{2m;\a}$ with $X_{2m+2;\a'}$, where $\a'$ is tuple obtained from~$\a$
by adding~1 as the last two components.
There is no effect on the homotopy class of the isomorphism in~\eref{realorient_e4} induced by
the second isomorphism in~\eref{CIprp_e7}.\\

\noindent
The second inclusion in~\eref{CROprop_e2} replaces 
$X_{2m;\a}$ with $X_{2m+2;\a'}$, where $\a'$ is tuple obtained from~$\a$
by adding~1 as the last two components.
There is no effect on the homotopy class of the isomorphism in~\eref{realorient_e4} induced by
the second isomorphism in~\eref{CIprp_e7}.\\ 

\noindent
(2) Suppose first $m\!\in\!2\Z$. 
The real part of the first line bundle in~\eref{CIprp_e6} is then orientable.
By \cite[Theorem~1.5]{RealGWsII},  the orientation on 
\BE{d1tau_e2}\ov\fM_{0,1}(\P^{2m-1},1)^{\tau_{2m}}=\fM_{0,1}(\P^{2m-1},1)^{\tau_{2m}}\EE
induced by the real orientation of Section~\ref{RealOrientCI_subs} is 
the orientation induced by the associated spin structure on $T\R\P^{2m-1}$.
The latter is induced by the canonical spin structure on 
$2m\cO_{\R\P^{2m-1}}(1)$ via Euler's sequence as in \cite[Section~5.5]{Teh}.
The claim now follows from \cite[(1.21)]{Teh}.\\

\noindent
Suppose now $m\!\not\in\!2\Z$. 
The real part of the first line bundle in~\eref{CIprp_e6} is then non-orientable.
By \cite[Theorem~1.5]{RealGWsII},  the orientation on~\eref{d1tau_e2} 
induced by the real orientation of Section~\ref{RealOrientCI_subs}
agrees with the orientation induced by the associated relative spin structure 
if and only if $m\!+\!1\!\in\!4\Z$.
This is the relative spin structure in the sense of \cite[Theorem~8.1.1]{FOOO}
associated with the oriented rank~2 real vector bundle
$$\cO_{\P^{2m-1}}(-m)\lra\P^{2m-1}$$
and the canonical spin structure on
\begin{equation*}\begin{split}
2m\cO_{\R\P^{2m-1}}(1)\oplus \cO_{\P^{2m-1}}(-m)\big|_{\R\P^{2m-1}}
&\approx 2m\cO_{\R\P^{2m-1}}(1)\oplus 2\cO_{\R\P^{2m-1}}(1)\\
&\approx 2(m\!+\!1)\cO_{\R\P^{2m-1}}(1).
\end{split}\end{equation*}
The orientation on~\eref{d1tau_e2} induced by this relative spin structure agrees
with the orientation induced by the relative spin structure 
 associated with the oriented rank~2 real vector bundle~$\cO_{\P^{2m-1}}(1)$
and the canonical spin structure on
$$2m\cO_{\R\P^{2m-1}}(1)\oplus \cO_{\P^{2m-1}}(1)\big|_{\R\P^{2m-1}}
\approx 2(m\!+\!1)\cO_{\R\P^{2m-1}}(1)$$
if and only if $m\!+\!1\!\in\!4\Z$.
Thus, the first and the third orientations on~\eref{d1tau_e2} are the same.
The second relative spin structure is the relative spin structure of \cite[Remark~6.5]{Teh}.
The claim now follows from \cite[Remark~1.11]{Teh}.\\

\noindent
The equality of the first and third orientations on~\eref{d1tau_e2} above
can be seen in another way as well.
The former is the orientation obtained as in~\cite{Ge2} 
by adding $2\cO_{\R\P^{2m-1}}(-m)$  to $T\P^{2m-1}$
and using the canonical spin structure on
$$2m\cO_{\R\P^{2m-1}}(1)\oplus 2\cO_{\R\P^{2m-1}}(-m)
\approx 2(m\!+\!1)\cO_{\R\P^{2m-1}}(1).$$
By \cite[Remark~3.10]{RealGWsII}, the resulting orientation on~\eref{d1tau_e2}
depends only on the latter (in contrast to the relative spin orientations).
Thus, we can replace $\cO_{\R\P^{2m-1}}(-m)$
by~$\cO_{\R\P^{2m-1}}(1)$.
By \cite[Corollary~3.8(1)]{RealGWsII}, the resulting orientation on~\eref{d1tau_e2}
agrees with  the third orientation on~\eref{d1tau_e2}.
\end{proof}

\begin{rmk}\label{d1num_rmk}
The computation of the number~\eref{d1tau_e0} in~\cite{Teh} in both cases is confirmed
through a second argument; see Corollary~5.4 and the paragraph above Remark~6.9 in~\cite{Teh}.
It is also indirectly confirmed by the two proofs of \cite[Proposition~3.5]{RealGWsII}.
\end{rmk}

\section{Proof of Theorem~\ref{g1EG_thm}}
\label{g1EG_sec}

\noindent
We assume that $(X,\om,\phi)$ is a compact real-orientable 6-manifold,
$$B \in H_2(X;\Z)\!-\!\{0\},\qquad l,k\in\Z^{\ge0}, \quad\hbox{and}\quad
 \mu_1,\ldots,\mu_l\!\in\!H^*(X;\Q), $$
are such~that 
$$\sum_{i=1}^{l}(\deg\mu_i\!-\!2)+2k=\blr{c_1(X),B}$$
and either $k\!=\!0$ or 
$$\lr{c_1(X),B'}\in4\Z ~{}~\forall\,B'\!\in\!H_2(X;\Z)~\hbox{with}~\phi_*B'\!=\!-B'
\quad\hbox{and}\quad \mu_1,\ldots,\mu_l=\PD_X(\pt).$$
We fix $J\!\in\!\cJ_{\om}^{\phi}$ as in the statement of Theorem~\ref{g1EG_thm}.

\subsection{Real GW- and~enumerative invariants}
\label{RealGWs_subs2}

\noindent
For $g\!\in\!\Z^{\ge0}$, denote by $\ov\fM_{g,l;k}(B)$ the moduli space of 
real genus~$g$ degree~$B$ $X$-valued $J$-holomorphic maps with
 $l$~conjugate pairs of marked points and $k$ real points and~by
$$\fM_{g,l;k}(B)\subset\ov\fM_{g,l;k}(B)$$
the subspace of maps from smooth domains.
The construction of Prym structures of~\cite{Loo} over~$\R$ provides 
a real analogue of the perturbations~$\nu$ of~\cite{RT2}
over the Deligne-Mumford moduli $\ov\cM_{g,l;k}$  of real curves.
For a $\phi$-invariant perturbation~$\nu$, denote by 
 $\ov\fM_{g,l;k}(B;\nu)$ the moduli space of 
real genus~$g$ degree~$B$ $X$-valued $(J,\nu)$-holomorphic 
maps with $l$~conjugate pairs of marked points
and $k$ real points and~by
$$\fM_{g,l;k}(B;\nu)\subset\ov\fM_{g,l;k}(B;\nu)$$
the subspace of maps from smooth domains.
Let
\begin{gather*}
\ev\!:\fM_{g,l;k}(B;\nu)\lra X^l\!\times\!(X^{\phi})^k, \\
\big[u,(z_1^+,z_1^-),\ldots,(z_l^+,z_l^-),x_1,\ldots,x_k,\fJ\big]\lra
\big(u(z_1^+),\ldots,u(z_l^+),u(x_1),\ldots,u(x_k)\big),
\end{gather*}
be the total evaluation map.\\

\noindent
Choose a tuple  $\p\!\equiv\!(p_1,\ldots,p_k)\!\in\!(X^{\phi})^k$ of real points
and pseudocycle representatives
$$h_1\!:Y_1\lra X, \quad\ldots, \quad h_l\!:Y_l\lra X$$
for the Poincare duals of  $\mu_1,\ldots,\mu_l$;
see the paragraph above \cite[Theorem~1.1]{PseudoCycles}.
Define
\begin{gather*}
\h\!=\!(h_1,\ldots,h_l)\!:\,\Y\!\equiv\!\prod_{i=1}^lY_i\lra X^l,\\
\De_{\p}^l=
\big\{(q_1,\ldots,q_l,p_1,\ldots,p_k,q_1,\ldots,q_l)\!:\,q_1,\ldots,q_l\!\in\!X\big\}
\subset X^l\!\times\!(X^{\phi})^k\!\times\!X^l\,.
\end{gather*}
With~$\nu$ as above, let
\begin{gather*}
\ov\fM_{g,\h;\p}(B;\nu)=
\big\{\big([\u],\y\big)\!\in\!\ov\fM_{g,l;k}(B;\nu)\!\times\!\Y\!:\,
\big(\ev(\u),\h(\y)\big)\!\in\!\De_{\p}^l\big\},\\
\fM_{g,\h;\p}(B;\nu)=\ov\fM_{g,\h;\p}(B;\nu)\cap
\big(\fM_{g,l;k}(B;\nu)\!\times\!\Y\big),\\
\ov\fM_{g,\h;\p}(B)=\ov\fM_{g,\h;\p}(B;0),\quad
\fM_{g,\h;\p}(B)=\fM_{g,\h;\p}(B;0).\\
\end{gather*}

\noindent
For generic choices of the tuple~$\p$, the pseudocycle representatives~$\h$, 
and a perturbation~$\nu$, the~map 
$$\ev\!\times\!\h\!: \ov\fM_{g,l;k}(B;\nu)\!\times\!\Y\lra X^l\!\times\!(X^{\phi})^k\!\times\!X^l $$
is transverse to~$\De_{\p}^l$ on every stratum of the domain.
In such a case, 
$$\fM_{g,\h;\p}(B;\nu)=\ov\fM_{g,\h;\p}(B;\nu)$$
is a zero-dimensional orbifold oriented by a real orientation on $(X,\om,\phi)$.
For $g\!=\!1$, the corresponding real GW-invariant is the weighted cardinality 
of this orbifold, i.e.
\BE{Jnucount_e}\GW_{1,B}^{\phi}\big(\mu_1,\ldots,\mu_l;\pt^k\big)
\equiv\, ^{\pm}\!\big|\fM_{1,\h;\p}(B;\nu)\big|\,;\EE
see \cite[Theorem~1.4]{RealGWsI} and the proof of \cite[Theorem~1.5]{RealGWsI}.
This number is independent of the generic choices of $\nu$, $\p$, and $\h$,
as well as of~$J$.\\

\noindent
By the  genus~1 regularity assumption on~$J$  and the proofs of \cite[Propositions~1.7,1.8]{g1comp},
the~map 
$$\ev\!\times\!\h\!: \ov\fM_{g,l;k}(B)\!\times\!\Y\lra X^l\!\times\!(X^{\phi})^k\!\times\!X^l $$
is also transverse to~$\De_{\p}^l$ on every stratum of the domain for $g\!=\!0,1$ and
generic choices of the tuple~$\p$ and the pseudocycle representatives~$\h$ and 
\BE{JfMspaces_e}\fM_{g,\h;\p}(B)\subset\ov\fM_{g,\h;\p}(B)\EE
is a finite zero-dimensional manifold consisting of non-intersecting embeddings.
A real orientation on $(X,\om,\phi)$ and the proofs of \cite[Theorems~1.4,1.5]{RealGWsI}
endow this manifold with an orientation. 
For $g\!=\!1$, the corresponding real enumerative invariant is the signed cardinality 
of this manifold, i.e.
\BE{Jcount_e}\E_{1,B}^{\phi}\big(\mu_1,\ldots,\mu_l;\pt^k\big)
\equiv\, ^{\pm}\!\big|\fM_{1,\h;\p}(B)\big|\,.\EE
We will show~that 
\BE{GWvsEnumR_e} \GW_{1,B}^{\phi}\big(\mu_1,\ldots,\mu_l;\pt^k\big)
=\E_{1,B}^{\phi}\big(\mu_1,\ldots,\mu_l;\pt^k\big);\EE
the two sets in~\eref{JfMspaces_e} may still be different.
By~\eref{GWvsEnumR_e}, the number in~\eref{Jcount_e} 
is independent of the generic choices of~$\p$ and~$\h$.

\subsection{Proof of \eref{GWvsEnumR_e}}
\label{GWvsEnumPf_subs}

\noindent
The proof of \eref{GWvsEnumR_e} is similar to the proof of \cite[Theorem~1.1]{g1comp2}.
We denote by $\fX_{1,l;k}(B)$ the space of all stable real genus~1 degree~$d$ maps 
in the $L^p_1$-topology of \cite[Section~3]{LT} and by
$$\fX_{1,l;k}^{\{0\}}(B)\subset\fX_{1,l;k}(B)$$
the subspace of tuples~$[\u]$ so that the degree of the restriction of the associated map~$u$
to the principal component of the domain is nonzero.
Let
\begin{gather*}
\fX_{1,\h;\p}(B)=
\big\{\big([\u],\y\big)\!\in\!\fX_{1,l;k}(B)\!\times\!\Y\!:\,
\big(\ev(\u),\h(\y)\big)\!\in\!\De_{\p}^l\big\},\\
\fX_{1,\h;\p}^{\{0\}}(B)=
\fX_{1,\h;\p}(B)\cap\big(\fX_{1,l;k}^{\{0\}}(B)\!\times\!\Y\big).
\end{gather*}
By the  genus~1 regularity assumption on~$J$,
\BE{JfMspaces_e2} 
\ov\fM_{1,\h;\p}(B)\cap \fX_{1,\h;\p}^{\{0\}}(B)=\fM_{1,\h;\p}(B)\,.\EE\\

\noindent
For $\nu$ sufficiently small, $\ov\fM_{1,\h;\p}(B;\nu)$ is contained
in an arbitrarily small neighborhood~of 
$$\ov\fM_{1,\h;\p}(B)\subset \fX_{1,l;k}(B)\!\times\!\Y\,.$$
By the regularity of the subspace~\eref{JfMspaces_e}, 
there is a unique element of $\fM_{1,\h;\p}(B;\nu)$ near each element of~$\fM_{1,\h;\p}(B)$.
Thus, the difference between the numbers in~\eref{GWvsEnumR_e} is the number of elements 
of $\fM_{1,\h;\p}(B;\nu)$ that lie close~to
\BE{JfMspaces_e3} \ov\fM_{1,\h;\p}(B)-\fM_{1,\h;\p}(B)\subset 
 \fX_{1,\h;\p}(B)- \fX_{1,\h;\p}^{\{0\}}(B);\EE
the inclusion above holds by~\eref{JfMspaces_e2}.\\

\noindent
By the regularity assumption on~$J$, every map~$u$ in the subset in~\eref{JfMspaces_e3}
 is not constant on a single bubble component $\Si_u^*\!\approx\!\P^1$
of its domain~$\Si_u$.
The image of~$u$ is an embedded rational curve intersecting the pseudocycles $h_1,\ldots,h_l$
in distinct non-real points.
Since $u$ is a real map and $\Si_u^*$ shares only one node~$P_u$ with the remainder of~$\Si_u$,
$u|_{\Si_u^*}$ is a real map and $P_u$ is a non-isolated real node.
The complement of~$\Si_u^*$ in~$\Si_u$ contains at most one real marked point
(in addition to the node shared with~$\Si_u^*$) and 
no conjugate pairs of marked points.  
If $\Si_u\!-\!\Si_u^*$ contains one marked point, then
the restriction of~$u$ to $\Si_u^*$ determines an element of 
$$\fM_{0,\h;\p}(B)=\ov\fM_{0,\h;\p}(B)$$
with one of the real marked points distinguished
and the remaining components of~$\Si_u$ determine an element of~$\ov\cM_{1,0;2}$.
If $\Si_u\!-\!\Si_u^*$ contains no marked point, then
the restriction of~$u$ to $\Si_u^*$ determines an element of 
the preimage 
$$\fM_{0,\h;\p}'(B)\subset \fM_{0,l;k+1}(B)\!\times\!\Y$$
of $\fM_{0,\h;\p}(B)$ under the forgetful morphism
$$\ff\!\times\!\id_{\Y}\!: \fM_{0,l;k+1}(B)\!\times\!\Y \lra \fM_{0,l;k}(B)\!\times\!\Y$$
dropping the additional marked point.
The remaining components of~$\Si_u$ determine an element of~$\ov\cM_{1,0;1}$ in this case.\\

\noindent
Since the images~$\cC_u$ of the elements of the $g\!=\!0$ spaces in~\eref{JfMspaces_e}
are disjoint real embedded  curves in~$(X,\phi)$, 
the topological components of $\ov\fM_{0,\h;\p}'(B)$ correspond 
to the elements~$[\u]$ of the $g\!=\!0$ spaces in~\eref{JfMspaces_e} .
The evaluation
$$\ev_0\!: \ov\fM_{0,\h;\p}'(B)\lra X^{\phi}$$
at the distinguished real marked point associated with the node restricts 
to a diffeomorphism from each topological component of the domain 
to the real locus of the curve~$\cC_u$ corresponding to the given component.\\

\noindent
The left-hand side of~\eref{JfMspaces_e3} is stratified by the subspaces 
$\cU_{\cT;\h;\p}(B)$ of maps of a fixed combinatorial type~$\cT$.
By the above, every non-empty stratum $\cU_{\cT;\h;\p}(B)$ is of
the~form
\BE{cUcTdecomp_e1}\cU_{\cT;\h;\p}(B) \approx \cU_{\cT}\!\times\!\fM_{0,\h;\p}'(B),\EE
where $\cU_{\cT}$ is a stratum of $\ov\cM_{1,0;1}$, or 
of the~form
\BE{cUcTdecomp_e2}\cU_{\cT;\h;\p}(B) \approx \cU_{\cT}\!\times\!\fM_{0,\h;\p}(B),\EE
where $\cU_{\cT}$ is a stratum of $\ov\cM_{1,0;2}$.
In particular, the \sf{main} boundary strata are of the form
\BE{mncMcT_e}\cM_{1,0;1}\!\times\!\fM_{0,\h;\p}'(B) \qquad\hbox{and}\qquad
\cM_{1,0;2}\!\times\!\fM_{0,\h;\p}(B).\EE
Proposition~\ref{bdcontr_prp} below describes the signed number of elements 
of  $\fM_{1,\h;\p}(B;\nu)$ near each stratum~$\cU_{\cT;\h;\p}(B)$.\\

\noindent 
We denote by 
$$\bE^{\R},L_1^{\R}\lra \ov\cM_{1,0;1} \qquad\hbox{and}\qquad
\bE^{\R},L_1^{\R}\lra\ov\cM_{1,0;2}$$
the real parts of the Hodge line bundle of holomorphic differentials and
of the universal tangent line bundle at the marked point 
and their pullbacks by the forgetful morphism.
The real part of the tautological line bundle over~$\ov\cM_{1,0;2}$ for the first marked point
and~$L_1^{\R}$ are canonically identified over~$\cM_{1,0;2}$.
Let 
$$L_0^{\R}\lra\ov\fM_{0,0;\{\0\}}(B)
\qquad\hbox{and}\qquad
L_0^{\R}\lra \fM_{0,\h;\p}(B),\ov\fM_{0,\h;\p}'(B)$$
denote the real part of the universal tangent line bundle at the marked point
and its pullbacks by the projection map to the moduli space component
and the forgetful morphisms
(keeping only the marked point associated with the node).\\

\noindent
The homomorphism
$$s_1\!:L_1^{\R}\lra \big(\bE^{\R}\big)^*, \qquad
\big\{s_1(v)\big\}(\psi)=\psi(v),$$
of real line bundles over $\ov\cM_{1,0;1}$ and $\ov\cM_{1,0;2}$ is an isomorphism.
We define a bundle homomorphism over $\ov\fM_{0,0;\{\0\}}(B)$, 
$\fM_{0,\h;\p}(B)$, and $\ov\fM_{0,\h;\p}'(B)$~by
\BE{cD0Rdfn_e}\cD_0^{\R}\!: L_0^{\R} \lra \ev_0^*TX^{\phi}\,,
\qquad \cD_0^{\R}\big([\u,v]\big)=\tnd_{x_0(\u)}u(v),\EE
where $x_0(\u)$ is the distinguished marked point and $u$ is the map component of the tuple~$\u$.
The restriction of~\eref{cD0Rdfn_e} over the component of $\ov\fM_{0,\h;\p}'(B)$ 
corresponding to an embedded real curve~$\cC$ in~$(X,\phi)$ identifies~$L_0^{\R}$
with~$\ev_0^*T\cC^{\phi}$.
For each combinatorial type~$\cT$ as above, let $\cD_{\cT}$ be the bundle homomorphism
over~$\cU_{\cT;\h;\p}(B)$ given~by 
$$\cD_{\cT}\!=\!\pi_1^*s_1\!\otimes_{\R}\!\pi_2^*\cD_0^{\R}\!:\,
\pi_1^*L_1^{\R}\!\otimes_{\R}\!\pi_2^*L_0^{\R} \lra 
\pi_1^*(\bE^{\R})^*\!\otimes_{\R}\!\pi_2^*\ev_0^*TX^{\phi}$$
with respect to the decomposition~\eref{cUcTdecomp_e1} or~\eref{cUcTdecomp_e2}.\\

\noindent
For a generic choice of a bundle section~$\ov\nu_{\cT}$ of 
\BE{ObscTdfn_e}\pi_1^*(\bE^{\R})^*\!\otimes_{\R}\!\pi_2^*\ev_0^*TX^{\phi}\lra  
\ov\cU_{\cT;\h;\p}(B),\EE
the affine bundle map
\begin{gather*}
\al_{\cT,\ov\nu_{\cT}}\!: \pi_1^*L_1^{\R}\!\otimes_{\R}\!\pi_2^*L_0^{\R} 
\lra \pi_1^*(\bE^{\R})^*\!\otimes_{\R}\!\pi_2^*\ev_0^*TX^{\phi}, \\
\al_{\cT,\ov\nu_{\cT}}(\ups)
=\cD_{\cT}(\ups)+\ov\nu_{\cT}(\u)
\quad\forall~\ups\!\in\!\pi_1^*L_1^{\R}\!\otimes_{\R}\!\pi_2^*L_0^{\R}\big|_{\u},
~\u\!\in\!\cU_{\cT;\h;\p}(B),
\end{gather*}
over $\ov\cU_{\cT;\h;\p}(B)$ is transverse to the zero set.
This section thus has no zeros unless $\cT$ corresponds to one
of the main strata~\eref{mncMcT_e}.
Since $\cD_{\cT}$ is injective, $\al_{\cT,\ov\nu_{\cT}}$ has a finite number of transverse 
zeros over each main stratum.\\

\noindent
The affine bundle map $\al_{\cT,\ov\nu_{\cT}}$ can be viewed 
as a section $\al_{\cT,\ov\nu_{\cT}}'$ of the bundle
$$\pi_{\cT}^*\big(  \pi_1^*(\bE^{\R})^*\!\otimes_{\R}\!\pi_2^*\ev_0^*TX^{\phi} \big)
\lra \pi_1^*L_1^{\R}\!\otimes_{\R}\!\pi_2^*L_0^{\R}, $$
where
$$\pi_{\cT}\!: \pi_1^*L_1^{\R}\!\otimes_{\R}\!\pi_2^*L_0^{\R} \lra \ov\cU_{\cT;\h;\p}(B)$$
is the bundle projection map. 
By our assumptions, 
the choice of a real orientation on $(X,\om,\phi)$ used to define the number~\eref{Jnucount_e}
orients the uncompactified moduli space $\fM_{1,\h;\p}(B)$. 
Since the domain and target of~$\al_{\cT,\ov\nu_{\cT}}$ form a deformation-obstruction complex 
for the oriented moduli space $\ov\fM_{1,\h;\p}(B)$ over the main strata $\cU_{\cT;\h;\p}(B)$
and all zeros of $\al_{\cT,\ov\nu_{\cT}}'$ are contained in
$$\pi_{\cT}^*\big(  \pi_1^*(\bE^{\R})^*\!\otimes_{\R}\!\pi_2^*\ev_0^*TX^{\phi} \big)
\big|_{\pi_1^*L_1^{\R}\otimes_{\R}\pi_2^*L_0^{\R}|_{\cU_{\cT;\h;\p}(B)}}
-\pi_1^*L_1^{\R}\!\otimes_{\R}\!\pi_2^*L_0^{\R}\,,$$
the orientation on $\fM_{1,\h;\p}(B)$ determines a sign for each zero of~$\al_{\cT,\ov\nu_{\cT}}'$
and $\al_{\cT,\ov\nu_{\cT}}$
(these are the same).
We denote the resulting weighted cardinality of $\al_{\cT,\ov\nu_{\cT}}^{-1}(0)$
by~$N(\cD_{\cT},\ov\nu_{\cT})$.\\

\noindent
As described above \cite[(3.14)]{g1comp2}, a perturbation~$\nu$ as above determines
a section~$\ov\nu_{\cT}$ of the complex analogue of the bundle~\eref{ObscTdfn_e}.
If $\nu$ is $\phi$-invariant, then  $\ov\nu_{\cT}$ is a section of~\eref{ObscTdfn_e}.
If $\nu$ is generic, then  $\al_{\cT,\ov\nu_{\cT}}$ is transverse to
the zero set for every~$\cT$.

\begin{prp}\label{bdcontr_prp}
There exists an non-empty subspace of $\phi$-invariant perturbations~$\nu$
satisfying the following property.
For each boundary stratum $\cU_{\cT;\h;\p}(B)$ of $\ov\fM_{1,\h;\p}(B)$, there exist 
$\cC_{\cT}(\dbar_J)\!\in\!\Q$ and a compact subset $K_{\nu}$ of $\cU_{\cT;\h;\p}(B)$ with the following property.
For every compact subset $K$ of $\cU_{\cT;\h;\p}(B)$ and 
open subset~$U$ of $\fX_{1,\h;\p}(B)$,
there exist an open neighborhood $U_{\nu}(K)$ of $K$ in $\fX_{1,\h;\p}(B)$
and $\eps_{\nu}(U)\!\in\!\R^+$, respectively, such that
$$^{\pm}\big|\fM_{1,\h;\p}(B;t\nu)\!\cap\!U\big|
=\cC_{\cT}(\dbar_J)
\quad\hbox{if}~~
t\!\in\!(0,\eps_{\nu}(U)),~K_{\nu}\!\subset\! K\!\subset\! U\!\subset\! U_{\nu}(K).$$
Furthermore, $\cC_{\cT}(\dbar_J)\!=\!N(\cD_{\cT},\ov\nu_{\cT})$
if $\cU_{\cT;\h;\p}(B)$ is a main stratum and $\cC_{\cT}(\dbar_J)\!=\!0$ otherwise.
\end{prp}

\noindent
This is the real case of \cite[Proposition~3.1]{g1comp2}.
It is obtained by restricting the proof in~\cite{g1comp2} to the space of real parameters.\\

\noindent
It remains to compute the number $N(\cD_{\cT},\ov\nu_{\cT})$ for each of
the main strata~\eref{mncMcT_e}.
Since the section~$\cD_{\cT}$ is injective, the zeros of $\al_{\cT,\ov\nu_{\cT}}$ 
correspond to the zeros of the transverse bundle section~$\ov\nu_{\cT}'$~of
\BE{Qbndl_e}
\pi_1^*(\bE^{\R})^*\!\otimes_{\R}\!\pi_2^*\ev_0^*TX^{\phi} \big/\Im\cD_{\cT}
\approx \pi_1^*L_1^{\R}\!\otimes_{\R}\!\pi_2^*
\big(\ev_0^*TX^{\phi}/\Im\cD_0\big)\EE
obtained by composting~$\ov\nu_{\cT}'$ with the projection~map;
see \cite[Section~3.3]{g2n2and3}.\\

\noindent
Each topological component of the closure $\ov\cU_{\cT;\h;\p}(B)$  of the first space 
in~\eref{mncMcT_e} is  $S^1\!\times\!S^1$,
with the circles coming from each factor.
By the proof of \cite[Theorem~1.5]{RealGWsI}, the orientation on $\fM_{1,\h;\p}(B)$
extends to an orientation over $\ov\cU_{\cT;\h;\p}(B)$.   
This orientation and an orientation on $\ov\cU_{\cT;\h;\p}(B)$ induce an orientation
on the target vector bundle in~\eref{Qbndl_e}.
The number $N(\cD_{\cT},\ov\nu_{\cT})$ is the signed number of zeros of 
the section~$\ov\nu_{\cT}'$ with respect to the chosen orientation on  $\ov\cU_{\cT;\h;\p}(B)$
and the induced orientation on this vector bundle, i.e.
\BE{Qbndl_e1}N(\cD_{\cT},\ov\nu_{\cT})=
\blr{e\big(\pi_1^*L_1^{\R}\!\otimes_{\R}\!\pi_2^*(\ev_0^*TX^{\phi}/\Im\cD_0)\big),
\big[\ov\cU_{\cT;\h;\p}(B)\big]}.\EE
Since the restriction of $\ev_0^*TX^{\phi}/\Im\cD_0$ to the component of
 $\ov\fM_{0,\h;\p}'(B)$ corresponding to a curve $\cC_u\!\subset\!X$
is the normal bundle of $\cC_u^{\phi}$ in~$X^{\phi}$ and is thus orientable,
the number~\eref{Qbndl_e1}  is zero in this case.\\

\noindent
Each topological component of the closure of the second space in~\eref{mncMcT_e} is~$\ov\cM_{1,0;2}$
and
\BE{Qbndl_e2}\pi_1^*L_1^{\R}\!\otimes_{\R}\!\pi_2^*
\big(\ev_0^*TX^{\phi}/\Im\cD_0\big)\approx 
L_1^{\R}\!\oplus\!L_1^{\R}\lra \ov\cM_{1,0;2}\EE
in this case.
Since $L_1^{\R}$ is the pullback of a bundle over $\ov\cM_{1,0;1}$ in our setup, 
there is an open subset of sections~$\ov\nu_{\cT}'$ of~\eref{Qbndl_e2}
that have no zeros.
Thus, $\nu$ can be chosen so that the section $\al_{\cT,\ov\nu_{\cT}}$ has no 
zeros and so $N(\cD_{\cT},\ov\nu_{\cT})\!=\!0$  in this case as~well.\\

\noindent
In summary, the difference between the two numbers in~\eref{GWvsEnumR_e} 
is the sum of the numbers $N(\cD_{\cT},\ov\nu_{\cT})$ corresponding
to the main boundary strata of $\ov\fM_{1,\h;\p}(B)$.
By the last two paragraphs, each of these numbers   $N(\cD_{\cT},\ov\nu_{\cT})$ is zero.
This establishes~\eref{GWvsEnumR_e} and Theorem~\ref{g1EG_thm}.

\section{Equivariant localization}
\label{EquivLoc_sec}

\noindent
Throughout this section, $n\!\in\!\Z^+$, $m$ is the integer part of~$n/2$, and
$$[n]=\{1,\ldots,n\}.$$
We denote by~$\phi$ either the involution~$\tau_n'$  on~$\P^{n-1}$
given by~\eref{tauprdfn_e} or the involution~$\eta_n$  on~$\P^{n-1}$ with $n\!=\!2m$
given by~\eref{tauetadfn_e}.
The restriction of the standard $\bT^n$-action on~$\P^{n-1}$
to a subtorus $\bT^m\!\subset\!\bT^n$ commutes with the involution~$\phi$;
see Section~\ref{EquivSetup_subs}.
This restriction thus induces a $\bT^m$-action on the moduli space $\ov\fM_{g,l}(\P^{n-1},d)^{\phi}$ 
of genus~$g$ degree~$d$ real $J_0$-holomorphic maps into~$(\P^{n-1},\phi)$ with 
$l$~pairs of conjugate marked points.
We describe the fixed loci of this action and their normal bundles, 
as  real vector bundles, in Sections~\ref{FixedLoci_subs} and~\ref{FLCsign_subs}.
By Proposition~\ref{CIorient_prp}, $(\P^{n-1},\phi)$ admits a real orientation 
if $(n,\phi)\!=\!(2m,\tau_n')$ or $(n,\phi)\!=\!(4m,\eta_n)$.
By \cite[Theorem~1.3]{RealGWsI}, $\ov\fM_{g,l}(\P^{n-1},d)^{\phi}$  is orientable in these cases
and is oriented by a real orientation on~$(\P^{n-1},\phi)$.
We determine  the orientations on the normal bundles to the fixed loci with respect to 
the real orientations of Section~\ref{RealOrientCI_subs} in Section~\ref{NBsign_subs}.
The equivariant localization data provided by Theorem~\ref{EquivLocal_thm}
determines the real GW-invariants 
of $(\P^{2m-1},\om_{2n},\tau_{2m})$ and $(\P^{4m-1},\om_{4m},\eta_{4m})$ and exhibits 
the two types of vanishing phenomena described in \cite[Sections~3.2,3.3]{Wal}.
We illustrate their use in Section~\ref{EquivLocApp_subs}.\\

\noindent
We also describe equivariant localization data related to the real GW-invariants
of the complete intersections $X_{n;\a}\!\subset\!\P^{n-1}$ preserved by~$\phi$.
Such a complete intersection is the zero set of a transverse holomorphic section~$s_{n;\a}$ 
of the vector bundle~\eref{cLdfn_e}
satisfying $s_{n;\a}\!\circ\!\phi\!=\!\wt\phi_{n;\a}\!\circ\!s_{n;\a}$
for a conjugation~$\wt\phi_{n;\a}$ lifting~$\phi$.
The latter is necessarily a direct sum of the conjugations~$\wt\phi_{n;1}^{(a)}$ 
and~$\wt\phi_{n;1,1}^{(a)}$ described in Section~\ref{RealOrientMfld_subs}. 
Let $\phi_{n;\a}\!=\!\phi|_{X_{n;\a}}$.
If $(n,\a)$ and~$\phi$ satisfy the assumptions of Proposition~\ref{CIorient_prp}
(with $\tau_n$ replaced by~$\tau_n'$ in the first case), 
then $(X_{n;\a},\phi_{n;\a})$ admits a real orientation.
By \cite[Theorem~1.3]{RealGWsI}, $\ov\fM_{g,l}(X_{n;\a},d)^{\phi_{n;\a}}$  is then orientable
if $n\!-\!k\!\in\!2\Z$ (so that the complex dimension of~$X_{n;\a}$ is odd)
and is oriented by a real orientation on~$(X_{n;\a},\phi_{n;\a})$.\\

\noindent
Whenever defined, the real  GW-invariants of $(X_{n;\a},\om_{2n;\a},\phi_{n;\a})$ 
are expected to be related
to the ``fibration''
\BE{cVdfn_e}
\pi_{n;\a}^{\wt\phi_{n;\a}}\!:
\cV_{n;\a}^{\wt\phi_{n;\a}}\equiv
\ov\fM_{g,l}\big(\cL_{n;\a},d\big)^{\wt\phi_{n;\a}}\lra  \ov\fM_{g,l}(\P^{n-1},d)^{\phi}\,.\EE
The $\bT^m$-action on the base of this ``fibration'' lifts canonically to the total space.
The restriction of~\eref{cVdfn_e} to a certain open subspace
\BE{fM0dfn_e}\fM_{g,l}^{\star}(\P^{n-1},d)^{\phi} \subset \ov\fM_{g,l}(\P^{n-1},d)^{\phi}\EE
is a vector orbi-bundle of the expected rank, but the dimensions of the fibers of~\eref{cVdfn_e}
are higher over the complement of~\eref{fM0dfn_e}.
Under the assumptions of Proposition~\ref{CIorient_prp}, 
the orientation systems of the restrictions of~$\cV_{n;\a}^{\wt\phi_{n;\a}}$ and
$T\ov\fM_{g,l}(\P^{n-1},d)^{\phi}$ to $\fM_{g,l}^{\star}(\P^{n-1},d)^{\phi}$ are the same.
The section~$s_{n;\a}$ and the real orientation on~$(X_{n;\a},\phi_{n;\a})$
described in Section~\ref{RealOrientCI_subs} determine a homotopy class of isomorphisms between
these two systems and a~class
\BE{PDeVdfn_e}
\e\big(\cV_{n;\a}^{\wt\phi_{n;\a}}\big)\in 
 H_{\bT^m}^*\big(\fM_{g,l}^{\star}(\P^{n-1},d)^{\phi};\fO\big)\EE
in the $\bT^m$-equivariant cohomology with coefficients in the orientation system 
of $\ov\fM_{g,l}(\P^{n-1},d)^{\phi}$.\\

\noindent
For each $i\!=\!1,\ldots,l$, let 
$$\psi_i\in H^2\big(\ov\fM_{g,l}(\P^{n-1},d)^{\phi};\Q\big)$$
denote the first Chern class of the universal cotangent line bundle associated with
the first marked point in the $i$-th conjugate pair.
In Section~\ref{LocData_subs}, we describe the equivariant localization contributions~to
\BE{PDeVdfn_e2}
\int_{[\ov\fM_{g,l}(\P^{n-1},d)^{\phi}]^{\vrt}}
\prod_{i=1}^l\!\!\big(\psi_i^{b_i}\ev_i^*\x^{p_i}\big)\, \e\big(\cV_{n;\a}^{\wt\phi_{n;\a}}\big)
\in H_{\bT^m}^*\EE
for any extension of~\eref{PDeVdfn_e} to a class 
$$\e\big(\cV_{n;\a}^{\wt\phi_{n;\a}}\big)\in 
 H_{\bT^m}^*\big(\ov\fM_{g,l}(\P^{n-1},d)^{\phi};\fO\big)$$
from the $\bT^m$-fixed loci contained in $\fM_{g,l}^{\star}(\P^{n-1},d)^{\phi}$;
all such loci consist of real maps with no contracted positive-genus components
(unless $k\!=\!0$, i.e.~$X_{n;\a}\!=\!\P^{n-1}$).
For $(n,\a)\!=\!(5,(5))$ and $\phi\!=\!\tau_n'$, our conclusion specializes to~\cite[(3.22)]{Wal}.
The equivariant contributions to~\eref{PDeVdfn_e2} described by Theorem~\ref{EquivLocal_thm} vanish
whenever the orientation on $\ov\fM_{g,l}(X_{n;\a},d)^{\phi_{n;\a}}$ 
induced by the canonical real orientation  depends 
on the implicit choices made in Section~\ref{RealOrientCI_subs};
see Section~\ref{CROprop_subs}.\\

\noindent
If $k\!=\!0$ or $g\!=\!0$, the subspace in~\eref{fM0dfn_e} is the entire moduli space. 
In the first case,  $\e(\cV_{n;\a})\!=\!1$.
In the second case, the class~\eref{PDeVdfn_e} relates
the $g\!=\!0$ real GW-invariants of $(X_{n;\a},\om_{n;\a},\phi_{n;\a})$ to $(\P^{n-1},\om_n,\phi)$,
as shown in the proof of \cite[Theorem~3]{PSW} for $(n,\a)\!=\!(5,(5))$ and $\phi\!=\!\tau_n'$.
If $g\!\ge\!1$ and $k\!\ge\!1$, the subspace in~\eref{fM0dfn_e} is not even dense in the entire moduli space.
By \cite[Theorem~1.1]{g1cone} and \cite[Theorem~1.1]{LZ},
the natural extension of the complex analogue of the $g\!=\!1$ case of~\eref{PDeVdfn_e} is 
a class on the closure 
$$ \ov\fM_{1,l}^0(\P^{n-1},d) \subset \ov\fM_{1,l}(\P^{n-1},d)$$
of the subspace corresponding to~\eref{fM0dfn_e} and relates
the reduced genus~1 (complex) GW-invariants of~$X_{n;\a}$ defined in~\cite{g1comp2} to $(\P^{n-1},\om_n)$.
By \cite[Theorem~1A]{g1diff}, the difference between the standard and reduced
GW-invariants of~$X_{n;\a}$ is a combination of the genus~0 GW-invariants.
The same should be the case in the real setting and in higher genera.
The proof of \cite[Theorem~1.1]{g1comp2} in fact suggests that the real analogues of 
the reduced genus~1 GW-invariants are equal to the genus~1 real GW-invariants of $(X,\om,\phi)$
defined in~\cite{RealGWsI} if $\dim_{\R}X\!=\!6$.
This is consistent with the prediction of \cite[Section~3.4]{Wal} that 
the genus~1 real GW-invariants of $(X_{n;\a},\om_{n;\a},\tau_{n;\a}')$ are obtained
from the equivariant contributions of the $\bT^m$-fixed loci in $\fM_{g,l}^{\star}(\P^{n-1},d)^{\phi}$ 
described in Section~\ref{LocData_subs}.

\subsection{Equivariant setting}
\label{EquivSetup_subs}

\noindent
Let $\bT^n$ denote the $n$-torus $(S^1)^n$ or its complex analogue~$(\C^*)^n$.
The $\bT^n$-quotient of its classifying space~$E\bT^n$ is $B\bT^n\!=\!(\P^{\i})^n$.
Thus, \sf{the group cohomology of~$\bT^n$} is
$$H_{\bT^n}^*\equiv H^*(B\bT^n;\Q)=\Q\big[\al_1,\ldots,\al_n\big],$$
where $\al_i\!\equiv\!\pi_i^*c_1(\ga_{\i}^*)$,
$\ga_{\i}\!\lra\!\P^{\i}$ is the tautological line bundle, and 
$$\pi_i\!: (\P^{\i})^n \lra \P^{\i}$$ 
is the projection to the $i$-th component.
We will call $\al_1,\ldots,\al_n$
\sf{the weights of the standard representation of~$\bT^n$ on~$\C^n$}. 
Let
$$\cH^*_{\bT^n}\approx \Q\big(\al_1,\ldots,\al_n\big)$$
denote the field of fractions of $H^*_{\bT^n}$.\\

\noindent
We denote the \sf{$\bT^n$-equivariant $\Q$-cohomology} of a topological space~$M$
with a $\bT^n$-action, i.e.~the cohomology~of
$$B_{\bT^n}M\equiv E\bT^n\!\times_{\bT^n}\!M\,,$$
by~$H_{\bT^n}^*(M)$.
The projection $B_{\bT^n}M\!\lra\!B\bT^n$ induces an action of~$H_{\bT^n}^*$
on $H_{\bT^n}^*(M)$.
Define
$$\cH^*_{\bT^n}(M)=\cH_{\bT^n}^*\!\otimes_{H_{\bT^n}^*}\!H_{\bT^n}^*(M)\,.$$
If $\bT^n$ acts trivially on~$M$, then 
$$B_{\bT^n}M= B\bT^n\!\times\!M, \qquad 
\cH^*_{\bT^n}(M)=\cH^*_{\bT^n}\!\otimes\!H^*(M;\Q);$$
the last identification is on the level of $\cH_{\bT^n}^*$-algebras.
For an oriented vector bundle $V\!\lra\!M$ with a $\bT^n$-action lifting
the action on~$M$, let
$$ \e(V)\equiv e\big(B_{\bT^n}V\big)\in H_{\bT^n}^*(M)$$ 
denote the \sf{equivariant Euler class of}~$V$.
For a complex bundle $V\!\lra\!M$ with a $\bT^n$-action lifting
the action on~$M$, let
$$ \c(V)\equiv c\big(B_{\bT^n}V\big)\in H_{\bT^n}^*(M)$$ 
denote the \sf{equivariant Chern class of}~$V$.
A continuous $\bT^n$-equivariant map $f\!:M'\!\lra\!M$ between two topological spaces
induces a homomorphism
$$f^*\!: H_{\bT}^*(M) \lra H_{\bT}^*(M').$$
The equivariant Euler and Chern classes are natural with respect to such maps.\\

\noindent
The standard action of $\bT^n$ on $\C^n$,
$$\big(u_1,\ldots,u_n\big)\cdot \big(z_1,\ldots,z_n\big)
= \big(u_1z_1,\ldots,u_nz_n\big)$$
descends to a $\bT^n$-action on $\P^{n-1}$.
The latter has $n$ fixed points,
\BE{fixedpt_e} 
P_1\equiv[1,0,\ldots,0], \qquad\ldots,\qquad P_n\equiv[0,\ldots,0,1].\EE
The curves preserved by this action are the lines through the fixed points,
\BE{P1ijdfn_e}\P_{ij}^1\equiv
\big\{[Z_1,\ldots,Z_n]\!\in\!\P^{n-1}\!:\,
Z_k\!=\!0~~\forall\,k\!\neq\!i,j\big\}, \qquad i,j\!\in\![n],~i\!\neq\!j.\EE
The product of the standard $\bT^n$-actions on~$\P^{n-1}$ and~$\C^n$ restricts 
to a $\bT^n$-action on the tautological line bundle
$$\ga_{n-1}\subset \P^{n-1}\!\times\!\C^n$$
as in~\eref{gandfn_e} and induces $\bT^n$-actions on the holomorphic line bundles
$$\cO_{\P^{n-1}}(a)\equiv \big(\ga_{n-1}^*\big)^{\otimes a}, \qquad \forall~a\!\in\!\Z.$$
Let
$$\x\equiv\e(\ga_{n-1}^*)\in H_{\bT^n}^*(\P^{n-1})$$
denote the equivariant hyperplane class.
The equivariant restrictions of $\x$ and $\c(T\P^{n-1})$ to 
the fixed points~\eref{fixedpt_e} are described~by
\BE{EquivRestr_e} \x|_{P_i}=\al_i\in H_{\bT^n}^*=H_{\bT^n}^*(P_i), \quad
\c(T\P^{n-1})\big|_{P_i}=\prod_{k\neq i}\!\big(1+\al_i\!-\!\al_k\big) 
\qquad\forall~i\!\in\![n].\EE
The first identity above follows from the definition of~$\al_i$.
The second identity follows from the homomorphisms~$f$ and~$g$ in 
the short exact sequence~\eref{Pnses_e} being $\bT^n$-equivariant
with respect to the $\bT^n$-action on the middle term obtained by tensoring the standard action
on~$\ga_{n-1}^*$ with the standard action on~$\C^n$.\\

\noindent
With notation as in~\eref{baridfn_e},
\BE{phiPact_e}  \phi(P_i)=P_{\phi(i)}\,.\EE
Denote by $\la_1,\ldots,\la_m$ the weights of the standard representation
of~$\bT^m$ on~$\C^m$.
The embedding 
\BE{TmTn_e}\io\!:\bT^m\lra\bT^n, \qquad 
\io\big(u_1,u_2,\ldots,u_m\big)= 
\begin{cases} (u_1,\bar{u}_1,\ldots,u_m,\bar{u}_m),&\hbox{if}~n\!=\!2m;\\
(u_1,\bar{u}_1,\ldots,u_m,\bar{u}_m,1),&\hbox{if}~n\!=\!2m\!+\!1;
\end{cases}\EE
induces a $\bT^m$-action on $\P^{n-1}$ that commutes with the involution~$\phi$. 
Under this embedding,
\BE{sp_weights_e}
(\al_1,\ldots,\al_n)\big|_{\bT^m}=
\begin{cases}
(\la_1,-\la_1,\ldots,\la_m,-\la_m), &\hbox{if}~n\!=\!2m;\\
(\la_1,-\la_1,\ldots,\la_m,-\la_m,0),&\hbox{if}~n\!=\!2m\!+\!1.
\end{cases}\EE\\

\noindent
A subspace $Y\!\subset\!\P^{n-1}$ is called \sf{$\phi$-invariant}
(\sf{$\bT^m$-invariant}) if $\phi(Y)\!=\!Y$ ($u(Y)\!=\!Y$ for all $u\!\in\!\bT^m$);
$Y$~is called \sf{$\phi$-fixed} (\sf{$\bT^m$-fixed}) 
if $\phi(y)\!=\!y$ for all $y\!\in\!Y$ ($u(y)\!=\!y$ for all $y\!\in\!Y$ and $u\!\in\!\bT^m$).
The next lemma describes the $\bT^m$- and $\phi$-fixed and invariant
zero- and one-dimensional subspaces of~$\P^{n-1}$. 
If $n\!=\!2m\!+\!1$, $1\!\le\!i\!\le\!m$, and $a,b\!\in\!\C^*$, then
\BE{cCidfn_e}\cC_i(a,b)\equiv\big\{[z_1,\ldots,z_n]\in\P^n\!:~
z_k\!=\!0~~\forall\,k\!\neq\!2i\!-\!1,2i,n,~
az_{2i-1}z_{2i}\!-\!bz_n^2\!=\!0\big\}\EE
is a smooth conic contained in the plane $\P^2_{2i-1,2i,n}$ spanned by $P_{2i-1},P_{2i},P_n$
and passing through~$P_{2i-1}$ and~$P_{2i}$.

\begin{lmm}[{\cite[Lemma~3.1]{bcov0_ci}}]\label{FixedCurves_lmm}
Suppose $n\!\in\!\Z^+$ and $m\!=\!\flr{n/2}$.
If $n\!=\!2m$, let $\phi$ be either the involution~$\tau_n'$ or~$\eta_{2m}$ on~$\P^{n-1}$;
if $n\!=\!2m\!+\!1$, let $\phi\!=\!\tau_n'$.
\begin{enumerate}[label=(\arabic*),leftmargin=*]

\item The $\bT^m$-fixed points in $\P^{n-1}$ are the  points $P_i$ in~\eref{fixedpt_e}
with $i\!\in\![n]$.

\item\label{Tcurves_it} If $n\!=\!2m$, the $\bT^m$-invariant irreducible curves in $\P^{n-1}$ 
are the lines $\P_{ij}^1$ as in~\eref{P1ijdfn_e}.
If $n\!=\!2m\!+\!1$, the $\bT^m$-invariant irreducible curves in $\P^{n-1}$ 
are the lines $\P_{ij}^1$ with $1\!\le\!i\!\neq\!j\!\le\!n$
and the conics $\cC_i(a,b)$ as in~\eref{cCidfn_e} with $a,b\!\in\!\C^*$.

\item\label{TRcurves_it} 
If $n\!=\!2m$, the $\phi$-invariant $\bT^m$-invariant irreducible curves in $\P^{n-1}$
are the lines $\P_{2i-1,2i}^1$ with $1\!\le\!i\!\le\!m$.
If $n\!=\!2m\!+\!1$, the $\phi$-invariant $\bT^m$-invariant irreducible curves 
in $\P^{n-1}$ are the lines $\P_{2i-1,2i}^1$ with $1\!\le\!i\!\le\!m$
and the conics $\cC_i(a,b)$ with $1\!\le\!i\!\le\!m$ and $a,b\!\in\!\C^*$ such that 
$a\bar{b}\!\in\!\R$.

\end{enumerate}
\end{lmm}

\begin{rmk}\label{conics_rmk}
If $n\!=\!2m\!+\!1$ and $a,b\!\in\!\C^*$ are such that $a\bar{b}\!\in\!\R$, then
$$\cC_i(a,b)^{\phi}=\begin{cases}
\{[z_1,\ldots,z_n]\!\in\!\P^2_{2i-1,2i,n}\!:\,|z_{2i-1}|\!=\!|z_{2i}|\big\},&
\hbox{if}~a\bar{b}\!\in\!\R^+;\\
\eset,&\hbox{if}~a\bar{b}\!\in\!\R^-.
\end{cases}$$
\end{rmk}

\subsection{Equivariant localization data}
\label{LocData_subs}

\noindent
Since the $\bT^m$-action on~$\P^{n-1}$ commutes with the involution~$\phi$,
it induces an action on $\ov\fM_{g,l}(\P^{n-1},d)^{\phi}$.
This action lifts to an action on the total space of the ``fibration"~\eref{cVdfn_e}.
Theorem~\ref{EquivLocal_thm} describes the class~\eref{PDeVdfn_e} 
as a sum of contributions from
the fixed loci of the $\bT^m$-action on  $\ov\fM_{g,l}(\P^{n-1},d)^{\phi}$.
Each such contribution is a rational fraction in the weights $\al_1,\ldots,\al_n$
of the standard representation of~$\bT^n$ on~$\C^n$ restricted to 
the subtorus $\bT^m\!\subset\!\bT^n$ defined by the embedding~\eref{TmTn_e};
thus, it is a rational fraction in the weights $\la_1,\ldots,\la_m$
of the standard representation of~$\bT^m$ on~$\C^m$.
As in the complex case described in detail in \cite[Chapter~27]{MirSym}, 
all such contributions (that do not cancel with other contributions) come from
fixed loci corresponding to decorated graphs.
In contrast to the complex case, there are also more complicated fixed loci involving
covers of the conics~\eref{cCidfn_e} if $n\!=\!2m\!+\!1$.
These are dealt with in Section~\ref{EquivLocalPf_sec}, 
where Theorem~\ref{EquivLocal_thm} is justified.\\

\noindent
A \sf{graph} $(\Ver,\Edg)$ is a pair consisting of a finite set~$\Ver$ of \sf{vertices}
and an element 
$$\Edg\in \Sym^k\big(\{\Ver'\!\subset\!\Ver\!:\,|\Ver'|\!=\!2\}\big)$$ 
for some $k\!\in\!\Z^{\ge0}$.
We will view~$\Edg$ as a collection of two-element subsets of~$\Ver$, called \sf{edges},
but it may contain several copies of the same two-element subset.
However, we do not allow an edge from a vertex back to itself.
For $e\!\in\!\Edg$ and $v\!\in\!e$, let $e/v\!\in\!\Ver$ denote the vertex in~$e$
other than~$v$.
A graph $(\Ver,\Edg)$ is \sf{connected} if for all $v,v'\!\in\!\Ver$ with $v\!\neq\!v'$
there~exist
$$k\!\in\!\Z^+, \quad v_1,\ldots,v_{k-1}\in\Ver,\quad 
e_1,\ldots,e_k\in\Edg
\qquad\hbox{s.t.}\quad v_{i-1},v_i\!\in\!e_i~~\forall~i\!=\!1,\ldots,k,$$ 
with $v_0\!\equiv\!v$ and $v_k\!\equiv\!v'$.
For any graph $(\Ver,\Edg)$, let
$$g\big(\Ver,\Edg\big)\equiv |\Edg|-|\Ver|+1$$
be its \sf{genus}.
An \sf{automorphism} of a graph $(\Ver,\Edg)$ is a bijection
$$h\!:\Ver\sqcup\Edg\lra \Ver\sqcup\Edg$$
such that $h(\Ver)\!=\!\Ver$ and $h(v)\!\in\!h(e)$ whenever 
$v\!\in\!\Ver$, $e\!\in\!\Edg$, and $v\!\in\!e$.
A \sf{subgraph} of $(\Ver,\Edg)$ is a graph $(\Ver',\Edg')$ such~that $\Ver'\!\subset\!\Ver$
and $\Edg'\!\subset\!\Edg$.\\

\noindent
For a finite set~$S$, an \sf{$S$-marked $[n]$-labeled decorated graph} 
(or simply \sf{decorated graph}) is a tuple
\BE{Dual_e} \Ga\equiv \big(\Ver,\Edg,\g,\vt,\fd,\fm\big) \EE
consisting of a graph $(\Ver,\Edg)$ and~maps
$$\g\!:\,\Ver\lra\Z^{\ge 0}, \quad \vt\!:\,\Ver\lra[n],\quad
\fd\!:\Edg\lra\Z^+, \quad \fm\!:\,S\lra\Ver,$$
such that $\vt(v)\!\neq\!\vt(e/v)$ for every $v\!\in\!e$ and $e\!\in\!\Edg$.
We define \sf{the genus}~$g(\Ga)$ and \sf{the degree}~$d(\Ga)$  of such a graph~by
$$g(\Ga)= g\big(\Ver,\Edg\big)+\sum_{v\in\Ver}\!\!\g(v) \qquad\hbox{and}\qquad
d(\Ga)=\sum_{e\in\Edg}\!\!\fd(e)\,,$$
respectively.
For each $v\!\in\!\Ver$, let
$$\E_v(\Ga)=\big\{e\!\in\!\Edg\!:\,v\!\in\!e\big\} \quad\hbox{and}\quad
\val_v(\Ga)=2\g(v)+\big|\E_v(\Ga)\big|+\big|\fm^{-1}(v)\big|$$
denote the set of edges leaving the vertex~$v$ and its \sf{valence}.\\

\noindent
An \sf{automorphism} of a decorated graph~$\Ga$ as in~\eref{Dual_e} 
is an automorphism~$h$ of the graph $(\Ver,\Edg)$ such~that 
$$\g=\g\circ h|_{\Ver}, \quad \vt=\vt\circ h|_{\Ver}, \quad 
\fd=\fd\circ h|_{\Edg}, \quad \fm=h\circ\fm.$$
A \sf{decorated subgraph} of a decorated graph~$\Ga$ as in~\eref{Dual_e} is a tuple
\BE{subDual_e}\Ga'\equiv \big(\Ver',\Edg',\g',\vt',\fd',\fm'\big)\EE
such that $(\Ver',\Edg')$ is a subgraph of $(\Ver,\Edg)$,
$$(\g',\vt')=(\g,\vt)\big|_{\Ver'}, \quad \fd'=\fd|_{\Edg'}, \quad 
\fm'=\fm|_{\fm^{-1}(\Ver')} \,.$$\\

\noindent
An \sf{involution} $\si$ on a decorated graph~$\Ga$ as in~\eref{Dual_e} is an automorphism of 
the graph~$(\Ver,\Edg)$ and the set~$S$ such~that
\BE{Invol_e}
\si\!\circ\!\si=\id,\quad \g=\g\!\circ\!\si|_{\Ver},\quad 
\phi\!\circ\!\vt=\vt\!\circ\!\si|_{\Ver},
\quad \fd=\fd\!\circ\!\si|_{\Edg},
\quad \si\!\circ\!\fm=\fm\!\circ\!\si|_S\,,\EE
with $\phi$ as in~\eref{baridfn_e}.
In such a case, let 
$$\nV_{\R}^{\si}(\Ga)\subset \Ver  \qquad\hbox{and}\qquad 
\E_\R^{\si}(\Ga)\subset\Edg$$ 
be  the subsets consisting of the fixed points of~$\si$ and define
$$\nV_\C^{\si}(\Ga)\equiv\Ver\!-\!\nV_\R^{\si}(\Ga),\qquad  
\E_\C^{\si}(\Ga)\equiv\Edg\!-\!\E_{\R}^{\si}(\Ga).$$
If $\nV_{\R}^{\si}(\Ga)\!\neq\!\eset$, then $n\!=\!2m\!+\!1$  and 
$\vt(v)\!=\!n$ for all $v\!\in\!\nV_{\R}^{\si}(\Ga)$.
If $e\!\in\!\E_{\R}^{\si}(\Ga)$, $v_1,v_2\!\in\!e$, and $v_1\!\neq\!v_2$,
then $\vt(v_1)\!=\!\phi(\vt(v_2))$.
If $n\!=\!2m\!+\!1$, then $\vt(v)\!\neq\!n$ for all $v\!\in\!e$ with $e\!\in\!\E_{\R}^{\si}(\Ga)$.
An \sf{automorphism} of a pair $(\Ga,\si)$ consisting of a decorated graph with an involution 
is an automorphism~$h$ of~$\Ga$ such that $h\!\circ\!\si\!=\!\si\!\circ\!h$.
We denote the group of all automorphisms of~$(\Ga,\si)$ by~$\Aut(\Ga,\si)$.\\

\noindent
The elements  of the set $\E_{\C}^{\si}(\Ga)$ above are called 
\sf{Klein edges} in \cite[Section~3]{Wal};
the elements of $\E_{\R}^{\si}(\Ga)$ are doubled \sf{half-edges} or \sf{disk edges}
in the terminology of~\cite{Wal}.
A graph with an involution can be depicted as in Figure~\ref{discord_fig}.
The label next to each edge~$e$ indicates the value of~$\fd$ on~$e$.
Each vertex~$v$ should similarly be labeled by the pair $(\g(v),\vt(v))$;
we drop the first label if it is~zero.
The involution~$\si$ is indicated by the two-sided arrows in Figure~\ref{discord_fig}.
For example, it exchanges the two vertices in the first two diagrams.
It flips each of the edges back to itself in the first diagram, but exchanges two of
them in the second.
The first diagram contains no Klein edges, while 
the second contains a pair of such edges.
In the first two diagrams, $S\!=\!\eset$.
In the last two diagrams, the disks indicate any possibility for 
the graph~$\Ga_+$.\\

\begin{figure}
\begin{pspicture}(-1,-2.1)(10,1.8)
\psset{unit=.3cm}
\pscircle*(5,3){.2}\pscircle*(5,-3){.2}
\pscircle[linewidth=.07](5,0){3}\psline[linewidth=.07](5,3)(5,-3)
\rput(5,4){$i$}\rput(5,-4){$\phi(i)$}
\rput(2.6,0){$1$}\rput(7.5,0){$1$}\rput(5.6,0){$1$}
\psline[linewidth=.05]{<->}(1.5,2)(1.5,-2)\psline[linewidth=.05]{<->}(4.5,2)(4.5,-2)
\psline[linewidth=.05]{<->}(8.5,2)(8.5,-2)
\pscircle*(13,3){.2}\pscircle*(13,-3){.2}
\pscircle[linewidth=.07](13,0){3}\psline[linewidth=.07](13,3)(13,-3)
\rput(13,4){$i$}\rput(13,-4){$\phi(i)$}
\rput(10.6,0){$1$}\rput(15.5,0){$1$}\rput(13.6,0){$1$}
\psline[linewidth=.05]{<->}(12.5,1.5)(11.5,-1.5)\psline[linewidth=.05]{<->}(12.5,-1.5)(11.5,1.5)
\psline[linewidth=.05]{<->}(16.5,2)(16.5,-2)
\pscircle*(22,2.5){.2}\pscircle*(27,2.5){.2}\pscircle*(32,2.5){.2}
\pscircle*(22,-2.5){.2}\pscircle*(27,-2.5){.2}\pscircle*(32,-2.5){.2}
\psline[linewidth=.07](22,2.5)(32,2.5)\psline[linewidth=.07](22,-2.5)(32,-2.5)
\psline[linewidth=.07](22,2.5)(22,-2.5)\psline[linewidth=.07](32,2.5)(32,-2.5)
\psline[linewidth=.05]{<->}(21.5,2)(21.5,-2)\psline[linewidth=.05]{<->}(32.5,2)(32.5,-2)
\pscircle[linewidth=.2,linestyle=dotted](27,4.6){2}
\pscircle[linewidth=.2,linestyle=dotted](27,-4.6){2}
\rput(21.5,3){\smsize{$1$}}\rput(32.5,3){\smsize{$1$}}
\rput(21.5,-3){\smsize{$2$}}\rput(32.5,-3){\smsize{$2$}}
\rput(27,1.6){\smsize{$3$}}\rput(27,-1.6){\smsize{$4$}}
\rput(27,4.5){$\Ga_+$}\rput(27,-4.5){\smsize{$\si(\Ga_+)$}}
\pscircle*(36,2.5){.2}\pscircle*(41,2.5){.2}\pscircle*(46,2.5){.2}
\pscircle*(36,-2.5){.2}\pscircle*(41,-2.5){.2}\pscircle*(46,-2.5){.2}
\psline[linewidth=.07](36,2.5)(46,2.5)\psline[linewidth=.07](36,-2.5)(46,-2.5)
\psline[linewidth=.07](36,2.5)(36,-2.5)\psline[linewidth=.07](46,2.5)(46,-2.5)
\psline[linewidth=.05]{<->}(38,1.5)(44,-1.5)\psline[linewidth=.05]{<->}(38,-1.5)(44,1.5)
\pscircle[linewidth=.2,linestyle=dotted](41,4.6){2}
\pscircle[linewidth=.2,linestyle=dotted](41,-4.6){2}
\rput(35.5,3){\smsize{$1$}}\rput(46.5,3){\smsize{$1$}}
\rput(35.5,-3){\smsize{$2$}}\rput(46.5,-3){\smsize{$2$}}
\rput(41,1.6){\smsize{$3$}}\rput(41,-1.6){\smsize{$4$}}
\rput(41,4.5){$\Ga_+$}\rput(41,-4.5){\smsize{$\si(\Ga_+)$}}
\end{pspicture}
\caption{Examples of decorated graphs with involutions.}
\label{discord_fig}
\end{figure}
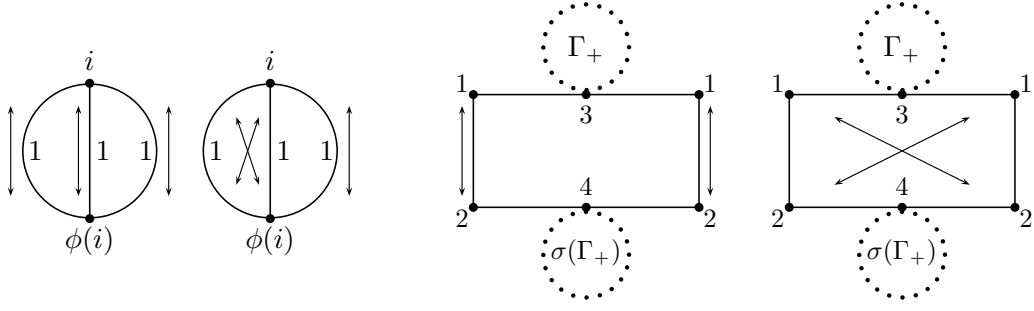

\noindent
We will call a pair $(\Ga,\si)$ consisting of a connected decorated graph with an involution 
\sf{admissible} if $\Ga$ is a connected graph, 
\BE{AdmissPairCond_e}
\fd(e)\not\in2\Z ~{}~\forall\,v\!\in\!\E_{\R}^{\si}(\Ga) 
\qquad\hbox{and}\qquad
\vt(v)\neq 2m\!+\!1 ~{}~\forall\, v\!\in\!\Edg.\EE
The second condition in~\eref{AdmissPairCond_e}, which is relevant only if $n\!\not\in\!2\Z$, implies that 
$\nV_{\R}^{\si}(\Ga)\!=\!\eset$.
For $d,g,l\!\in\!\Z^{\ge0}$, let $\cA_{g,l}(n,d)$ denote 
the set of admissible pairs $(\Ga,\si)$ such that $\Ga$ is an $S_l$-marked $[n]$-labeled
decorated graph with $g(\Ga)\!=\!g$, $d(\Ga)\!=\!d$, 
\BE{Sldfn_e}S_l\equiv \{1^+,1^-,\ldots,l^+,l^-\big\}, \qquad 
\si\big(i^{\pm}\big)=i^{\mp}\quad\forall~i\!=\!1,\ldots,l.\EE
We show in Section~\ref{EquivLocalPf_sec} that all equivariant contributions to~\eref{PDeVdfn_e}
arise from the fixed loci corresponding to the elements of $\cA_{g,l}(n,d)$.\\

\noindent
In the remainder of this section, we describe the equivariant contributions arising 
from the elements of~$\cA_{g,l}(n,d)$.
Fix tuples
$$\b\equiv(b_1,\ldots,b_l)\in\big(\Z^{\ge0}\big)^l  \qquad\hbox{and}\qquad 
\p\equiv(p_1,\ldots,p_l)\in\big(\Z^{\ge0}\big)^l.$$
We will describe the equivariant localization contributions to~\eref{PDeVdfn_e2}
under the assumptions that $n\!-\!k\!\in\!2\Z$ and either 
\begin{enumerate}[label=$\bu$,leftmargin=*]

\item $\phi\!=\!\tau_n'$ and $(n,\a)$ satisfies the conditions of Proposition~\ref{CIorient_prp}(1) or

\item  $\phi\!=\!\eta_{2m}$ and $(n,\a)$ satisfies 
the conditions of Proposition~\ref{CIorient_prp}(2) with $n\!=\!2m$.

\end{enumerate}
If $k\!=\!0$, i.e.~$X_{n;\a}\!=\!\P^{n-1}$, 
the equivariant Euler class in~\eref{PDeVdfn_e2} is~1.
If $g\!=\!0$, this class is well-defined over the entire moduli space 
as an element in the cohomology of $\ov\fM_{0,l}(\P^{n-1},d)^{\phi}$
twisted by the orientation system of this moduli space.
The integral in~\eref{PDeVdfn_e2} then computes the genus~0 real GW-invariants
of $(X_{n;\a},\om_{n;\a},\phi_{n;\a})$.
In other cases, this Euler class is well-defined over the subspace~\eref{fM0dfn_e}.
For the admissible pairs $(\Ga,\si)$ with $\g(v)\!=\!0$ for all $v\!\in\!\Ver$,
Theorem~\ref{EquivLocal_thm} below then describes the contribution of
the $\bT^m$-fixed locus in $\ov\fM_{g,l}(\P^{n-1},d)^{\phi}$ corresponding to~$(\Ga,\si)$ 
to the integral in~\eref{PDeVdfn_e2} of any extension of  $\e(\cV_{n;\a}^{\wt\phi_{n;\a}})$
to an equivariant cohomology class over the entire space.
In the case $g\!=\!1$ and $\dim_{\C}\!X_{n;\a}\!=\!3$, it is expected that this class integral is independent
of the extension and computes the  genus~1 real GW-invariants
of $(X_{n;\a},\om_{n;\a},\phi_{n;\a})$.\\

\noindent
For $\a\!\in\!(\Z^+)^k$ as before, let
$$ 
\lr\a=a_1\cdot\ldots\cdot a_k.$$
For $g\!\in\!\Z^{\ge0}$ and a finite set~$S$ with $2g\!+\!|S|\!\ge\!3$, 
denote by $\ov\cM_{g,S}$  the usual Deligne-Mumford moduli space of stable genus~$g$
$S$-marked curves and~by
$$\bE\lra \ov\cM_{g,S}$$
the Hodge vector bundle of holomorphic differentials.
For each $i\!\in\!S$, let
$$\psi_i\in H^2\big(\ov\cM_{g,S};\Q\big)$$
be the first Chern class of the universal cotangent line bundle associated with
the $i$-th marked point.\\

\noindent
Suppose $(\Ga,\si)$ is an element of $\cA_{g,l}(n,d)$ with $\Ga$ as in~\eref{Dual_e}. 
For each $v\!\in\!\Ver$, let
\begin{gather}
\label{Svdfn_e}
S_v=\E_v(\Ga)\sqcup\fm^{-1}(v), \quad
S_v^-=\big\{i\!=\!1,\ldots\!,l\!:\,i^-\!\in\!S_v\big\}, \quad
|\b|_v=\sum_{\begin{subarray}{c}1\le i\le l\\ i^{\pm}\in S_v\end{subarray}}\!\!\!b_i, 
\quad
|\p|_v=\sum_{\begin{subarray}{c}1\le i\le l\\ i^{\pm}\in S_v\end{subarray}}\!\!\!p_i,\\ 
\notag
\fs_v=\g(v)\!-\!1\!+\!|\E_v(\Ga)|+\!\!\sum_{i\in S_v^-}\!\!\big(1\!+\!b_i\!+\!p_i\big),
\quad
\psi_{e;v}=\frac{\al_{\vt(e/v)}\!-\!\al_{\vt(v)}}{\fd(e)}~~\forall\,e\!\in\!\E_v(\Ga).
\end{gather}
If $\val_v(\Ga)\!\ge\!3$, let 
\BE{Contr3V_e}\Cntr_{(\Ga,\si);v}(\b,\p)=
\bigg(\frac{\e(T_{P_{\vt(v)}}\P^{n-1})}{\lr\a\al_{\vt(v)}^k}\bigg)^{\!|\E_v(\Ga)|-1}
\al_{\vt(v)}^{|\p|_v} \!\!\!\!\!\!\!\!\!
\int\limits_{\ov\cM_{\g(v),S_v}}\!\!
\frac{\e(\bE^*\!\otimes\!T_{P_{\vt(v)}}\P^{n-1})}
{\prod\limits_{e\in\E_v(\Ga)}\!\!\!\!\!\!\psi_{e;v}\!\left(\psi_{e;v}\!+\!\psi_e\right)}
\!\!\prod_{\begin{subarray}{c}1\le i\le l\\ i^{\pm}\in S_v\end{subarray}}\!\!\!\psi_i^{b_i}\,.\EE
If $\val_v(\Ga)\!=\!1,2$, let
\BE{Contr2V_e}\begin{split}
\Cntr_{(\Ga,\si);v}(\b,\p)&=-(-1)^{\val_v(\Ga)}\!
\bigg(\frac{\e(T_{P_{\vt(v)}}\P^{n-1})}{\lr\a\al_{\vt(v)}^{k}}\bigg)^{\!|\E_v(\Ga)|-1}
\!\al_{\vt(v)}^{|\p|_v}\\
&\hspace{1in}
\times\bigg(\prod\limits_{e\in\E_v(\Ga)}\!\!\!\!\!\!\psi_{e;v}\bigg)^{-1}
\bigg(\!\sum\limits_{e\in\E_v(\Ga)}\!\!\!\!\!\psi_{e;v}\!
\bigg)^{\!3-\val_v(\Ga)-|\E_v(\Ga)|+|\b|_v}\,.
\end{split}\EE 
In light of~\eref{EquivRestr_e},
the equivariant Euler classes of $T_{P_{\vt(v)}}\P^{n-1}$ and $\bE^*\!\otimes\!T_{P_{\vt(v)}}\P^{n-1}$
are readily expressible in terms of the torus weights $\al_1,\ldots,\al_n$
and the Hodge classes $c_i(\bE)$ on~$\ov\cM_{\g(v),S_v}$.

\begin{rmk}\label{VerCntr_rmk}
Our \sf{vertex contributions}, i.e.~\eref{Contr3V_e} and~\eref{Contr2V_e},  
{\it include} the movements of the nodes associated with the edges $e\!\in\!\E_v(\Ga)$,
in contrast to \cite[(27.8)]{MirSym} and \cite[(3.22)]{Wal}.
The right-hand sides of these equations are the standard 
vertex contributions in the complex setting.
The inclusion of the node movements 
has the effect of dividing the product of all factors on the first two lines 
in \cite[(27.8)]{MirSym} and the last two lines in \cite[(3.22)]{Wal}
associated with~$v$ by the product of $-\psi_{e;v}$ with $e\!\in\!\E_v(\Ga)$.
In the real setting, each vertex contribution comes with the sign $(-1)^{\fs_v}$;
see~\eref{EquivLocal_e}.
The contribution from $S_v^-$ comes from orienting the moduli space by
the positive marked points.
The contributions of $\g(v)\!-\!1$ and $|\E_v(\Ga)|$ appear for more delicate reasons. 
The first arises from the comparison between the orientation on the moduli space~\eref{fMbu_e0}  
of real maps from the nodal doublets~\eref{SymSurfDbl_e}
induced by a real orientation
and the standard complex orientation on the moduli space of maps from one
of the components; see the first equation in~\eref{Csgn_e} and \cite[Theorem~1.4]{RealGWsII}.
The second sign contribution arises because the complex line bundle of smoothings 
of the node associated with each flag~$(e,v)$ should be taken with the anti-complex orientation;
see \eref{cVcZclasses_e3} and \cite[Theorem~1.2]{RealGWsII}.
The product of these extra signs over all relevant vertices
 constitutes the leading sign in \cite[(3.15)]{Wal}.\\
\end{rmk}

\noindent
With $(\Ga,\si)$ as in the previous paragraph, let $e\!\in\!\Edg$
and $v_1,v_2\!\in\!e$ be the two vertices of~$e$.
Suppose first that $e\!\in\!\E_{\C}^{\si}(\Ga)$.
If \BE{ConicsCovCond_e}
 n=2m\!+\!1, \qquad \fd(e)\in2\Z, \quad\hbox{and}\quad \vt(v_1)=\phi\big(\vt(v_2)\big)\,,\EE 
then $k\!\ge\!1$.
In this case, we set $\Cntr_{(\Ga,\si);e}\!=\!0$ if $k\!\ge\!2$ and 
\BE{ContrECb_e}\begin{split} 
&\Cntr_{(\Ga,\si);e}=
\frac{(-1)^{(a_1+1)\fd(e)/2}}{\fd(e)}\,
\frac{a_1((a_1\fd(e)/2)!)^2}{((\fd(e)/2)!)^2(\fd(e)!)^2}
\frac{\left(\frac{2\al_{\vt(v)}}{\fd(e)}\right)^{\!(a_1-3)\fd(e)+2}}
{\prod\limits_{\begin{subarray}{c}1\le j<n\\ j\neq\vt(v_1),\vt(v_2) \end{subarray}}
\!\!\!\!\left(\!\al_j\!\!\! \prod\limits_{r=1}^{\fd(e)/2}\!\!
\left(\!r^2\left(\frac{2\al_{\vt(v)}}{\fd(e)}\right)^{\!2}\!-\!\al_j^2\!\right)\!\!\right)}
\end{split}\EE
if $k\!=\!1$ and $v\!\in\!e$ is either vertex.
If one of the conditions in~\eref{ConicsCovCond_e} is not satisfied, then let
\BE{ContrEC_e}\begin{split}
\Cntr_{(\Ga,\si);e}&=\frac{(-1)^{\fd(e)}}{\fd(e)\,(\fd(e)!)^2}
\frac{\prod\limits_{i=1}^k\prod\limits_{r=0}^{a_i\fd(e)}\!
\frac{(a_i\fd(e)-r)\al_{\vt(v_1)}+r\al_{\vt(v_2)}}{\fd(e)}}
{\left(\frac{\al_{\vt(v_1)}-\al_{\vt(v_2)}}{\fd(e)}\right)^{\!\!2\fd(e)-2} \!\!\!\!\!\!\!\!\!\!\!
\prod\limits_{j\neq\vt(v_1),\vt(v_2)}\prod\limits_{r=0}^{\fd(e)}\!\!
\left(\!\frac{(\fd(e)-r)\al_{\vt(v_1)}+r\al_{\vt(v_2)}}{\fd(e)}\!-\!\al_j\!\right)}\,.
\end{split}\EE
If $e\!\in\!\E_{\R}^{\si}(\Ga)$, then $\fd(e)\!\not\in\!2\Z$ by the first assumption
in~\eref{AdmissPairCond_e}.
If in addition $a_i\!\not\in\!2\Z$ for all $i\!=\!1,\ldots,k$,  let
\BE{ContrER_e}
\Cntr_{(\Ga,\si);e}=
\frac{(-1)^{|\phi|+\frac{\fd(e)-1}{2}}}{\fd(e)}\,
\frac{\prod\limits_{i=1}^k\!\!\left(a_i\fd(e)\right)!!}{2^{\fd(e)-1}\fd(e)!}
\frac{\left(\frac{\al_{\vt(v_1)}}{\fd(e)}\right)^{\!\frac{(|\a|-2)\fd(e)+k}{2}+1}}
{\prod\limits_{j\neq\vt(v_1),\vt(v_2)}
\!\!\!\!\!\!\!\prod\limits_{r=0}^{(\fd(e)-1)/2}\!\!
\left(\frac{(\fd(e)-2r)\al_{\vt(v_1)}}{\fd(e)}\!-\!\al_j\right)}.\EE
If $a_i\!\in\!2\Z$ for some $i\!=\!1,\ldots,k$, set $\Cntr_{(\Ga,\si);e}\!=\!0$.

\begin{rmk}\label{EdgCntr_rmk}
Our \sf{edge contributions}, i.e.~\eref{ContrECb_e}, \eref{ContrEC_e}, and~\eref{ContrER_e},
do {\it not} include the movements of 
the nodes associated with the vertices  $v_1,v_2\!\in\!e$,
in contrast to \cite[(27.8)]{MirSym} and \cite[(3.22)]{Wal}.
They are included into the vertex contributions.
The right-hand side in~\eref{ContrEC_e} is the {\it negative} of the standard 
edge contribution in the complex setting with the automorphism group taken into account.
This has the effect of multiplying the product of all factors on the last line 
in \cite[(27.8)]{MirSym} and the first line in \cite[(3.22)]{Wal}
associated with~$e$ by  $-\psi_{e;v_1}\psi_{e;v_2}/\fd(e)$.
The extra negative sign for each Klein edge arises due to \cite[Theorem~1.4]{RealGWsII};
see Remark~\ref{VerCntr_rmk} and the second equation in~\eref{Csgn_e}. 
The product of these extra signs over all relevant Klein edges
 constitutes the leading sign in \cite[(3.22)]{Wal}.
For $(n,\a)\!=\!(5,(5))$ and $(n,k)\!=\!(2m,0)$
with $m\!\in\!2\Z$, \eref{ContrER_e} becomes the disk factor
on the first line in \cite[(3.22)]{Wal} multiplied by $-\psi_{e;v_1}/\fd(e)$
and \cite[(6.21)]{Teh} multiplied by $-(-1)^{|\phi|}\psi_{e;v_1}/\fd(e)$, respectively.
For $(n,k)\!=\!(2m,0)$ with $m\!\not\in\!2\Z$ and $\phi\!=\!\tau_n'$,
\eref{ContrER_e} becomes
\cite[(6.21)]{Teh} multiplied by 
$$-(-1)^{(\fd(e)-1)/2}\psi_{e;v_1}/\fd(e).$$
Along with the node sign correction of \cite[Remark~6.9]{Teh},
the extra factor of $(-1)^{(\fd(e)-1)/2}$ accounts for the difference 
between our canonical orientation on the moduli space and that 
induced by the relative spin structure of  \cite[Remark~6.5]{Teh};
see \cite[Corollary~3.6]{RealGWsII} and the second case in Section~\ref{EdgeCntrPrp_subs}.
The map automorphisms arising from the edge degrees are absorbed into the automorphism
factor \cite[(3.15)]{Wal};
these automorphisms and 
the factor of $(-1)^{|\phi|}$ arising from the $\eta$ involution on the domain
is taken into account in \cite[(6.7)]{Teh}.
\end{rmk}

\begin{rmk}\label{EdgCntrVan_rmk}
The right-hand side in~\eref{ContrECb_e} is obtained from the right-hand side in~\eref{ContrEC_e}
by first setting $\al_n\!=\!0$ as required by~\eref{sp_weights_e}, 
simplifying the resulting expression, and then setting \hbox{$\al_{\vt(v_1)}\!=\!-\al_{\vt(v_2)}$}.
This corresponds to resolving the $\frac{0}{0}$ ambiguity described below \cite[(3.23)]{Wal} 
``from" the left diagram in \cite[Figure~5]{Wal}.
The restrictions of the denominators in~\eref{ContrEC_e} may vanish because 
some fixed loci of the $\bT^m$-action on $\ov\fM_{g,l}(\P^{n-1},d)^{\phi}$
consist of covers of the  conics~$\cC_i(a,b)$
of Lemma~\ref{FixedCurves_lmm}\ref{TRcurves_it};
see Section~\ref{FixedLoci_subs}.
The equivariant contributions from these fixed loci are computed in Lemma~\ref{AuxAct_lmm}
using an auxiliary $S^1$-action.
The right diagram in \cite[Figure~5]{Wal} involves graphs $\Ga$ with vertices $v\!\in\!\Ver$
mapped by~$\vt$ to~$2m\!+\!1$;
the second condition in~\eref{AdmissPairCond_e} excludes such graphs in our perspective.\\
\end{rmk}

\noindent
Since $\nV_{\R}^{\si}(\Ga)\!=\!\eset$, we can choose a subset 
$\nV_+^{\si}(\Ga)$ of~$\Ver$ such~that 
\BE{nVplus_e}\Ver=\nV_+^{\si}(\Ga) \sqcup \nV_-^{\si}(\Ga)
\qquad\hbox{with}\quad \nV_-^{\si}(\Ga) \equiv \si\big(\nV_+^{\si}(\Ga)\big).\EE
We can also choose a subset $\E_+^{\si}(\Ga)$ of~$\E_{\C}^{\si}(\Ga)$ such~that 
\BE{nEplus_e}\Edg=\E_{\R}^{\si}(\Ga) \sqcup \E_+^{\si}(\Ga) \sqcup \E_-^{\si}(\Ga)
\qquad\hbox{with}\quad \E_-^{\si}(\Ga) \equiv \si\big(\E_+^{\si}(\Ga)\big).\EE

\begin{thm}\label{EquivLocal_thm}
Suppose $(\Ga,\si)\!\in\!\cA_{g,l}(n,d)$ with $\Ga$ as in~\eref{Dual_e}.
If $\g(v)\!=\!0$ for all $v\!\in\!\Ver$ or $k\!=\!0$,
then the contribution of the $\bT^m$-fixed locus in $\ov\fM_{g,l}(\P^{n-1},d)^{\phi}$ 
corresponding to~$(\Ga,\si)$ to~\eref{PDeVdfn_e2} is the restriction~of
\BE{EquivLocal_e}\Cntr_{(\Ga,\si)}(\b,\p)=\frac{1}{|\Aut(\Ga,\si)|}
\prod_{v\in\nV_+^{\si}(\Ga)}\!\!\!\!\!(-1)^{\fs_v}\Cntr_{(\Ga,\si);v}(\b,\p)
\prod_{e\in\E_{\R}^{\si}(\Ga)\sqcup\E_+^{\si}(\Ga)}
\hspace{-.33in}\Cntr_{(\Ga,\si);e}\EE
to the subtorus $\bT^m\!\subset\!\bT^n$ as in~\eref{sp_weights_e}.
The integral in~\eref{PDeVdfn_e2}  is the sum of these contributions
over~$\cA_{g,l}(n,d)$ if $g\!=\!0,1$ or $k\!=\!0$.
\end{thm}

\begin{rmk}\label{EquivLocal_rmk}
As indicated in Section~\ref{NBsign_subs}, the transfer of the movement of the node corresponding
to a flag~$(e,v)$ from the contribution for~$e$ to the contribution for~$v$ 
in effect marks all nodes with $v\!\in\!\nV_+^{\si}(\Ga)$ as positive.
This is used to compare the orientation associated with~$(\Ga,\si)$
to  the orientations associated with its edges and vertices.
This transfer allows us to choose the subsets $\nV_+^{\si}(\Ga)$ and $\E_+^{\si}(\Ga)$
at random, as long as the conditions~\eref{nVplus_e} and~\eref{nEplus_e} are satisfied.
We show directly in Section~\ref{EquivLocApp_subs} that~\eref{EquivLocal_e}
is independent of all choices made, provided 
the conditions of Proposition~\ref{CIorient_prp} and Theorem~\ref{EquivLocal_thm}
are satisfied and $n\!-\!k\!\in\!2\Z$
(otherwise, the moduli space of real maps into $(X_{n;\a},\phi_{n;\a})$
may not be orientable).
\end{rmk}

\subsection{Examples and applications}
\label{EquivLocApp_subs}

\noindent
We now make a number of observations regarding the contributions of Theorem~\ref{EquivLocal_thm},
apply it in some specific cases, 
give an alternative proof of  
Theorem~\ref{GWsCI_thm}\ref{CIvan_it} in the case of projective spaces, 
and establish Proposition~\ref{GWsPn_prp}.\\

\noindent
If the restrictions of Proposition~\ref{CIorient_prp} and Theorem~\ref{EquivLocal_thm}
are satisfied and $n\!-\!k\!\in\!2\Z$,
the contribution~\eref{EquivLocal_e} from an admissible graph $(\Ga,\si)$ 
is independent of all choices made: 
$$v_1\in e \quad\forall~e\!\in\!\Edg, \qquad
\nV_+^{\si}(\Ga)\subset \nV_{\C}^{\si}(\Ga), \quad\hbox{and}\quad
\E_+^{\si}(\Ga)\subset \E_{\C}^{\si}(\Ga).$$
The right-hand sides of~\eref{ContrECb_e} and~\eref{ContrEC_e} are symmetric in~$v_1$ and~$v_2$,
even before restricting to the subtorus $\bT^m\!\subset\!\bT^n$.
Suppose $e\!\in\!\E_{\R}^{\si}(\Ga)$ and $v_1,v_2\!\in\!e$.
Since $\fd(e)\!\not\in\!2\Z$ and $\al_{\vt(v_1)}\!=\!-\al_{\vt(v_2)}$ on~$\bT^m$
in this case,
replacing~$v_1$ by~$v_2$ in~\eref{ContrER_e} changes the restriction of its right-hand side 
to~$\bT^m$ by the factor of~$-1$ to the power~of
$$\frac{(|\a|\!-\!2)\fd(e)\!+\!k}{2}+1+\frac{\fd(e)\!+\!1}{2}(n\!-\!2) \equiv
\frac12\big( (n\!+\!|\a|)\fd(e)+(n\!+\!k)\big) \qquad\mod2.$$
If $a_i\!\not\in\!2\Z$ for all $i\!=\!1,\ldots,k$, then Lemma~\ref{CIorient_lmm} implies that 
\BE{EvenSum_e1}\frac12\big( (n\!+\!|\a|)\fd(e)+(n\!+\!k)\big) \equiv 
\frac12(n\!+\!k)\big(\fd(e)\!+\!1)\big) \qquad\mod2.\EE
If $n\!-\!k\!\in\!2\Z$ and $a_i\!\not\in\!2\Z$ for all $i\!=\!1,\ldots,k$,
interchanging~$v_1$ and~$v_2$ thus has no effect on the restriction of the right-hand side 
of~\eref{ContrER_e} to~$\bT^m$.
If $a_i\!\in\!2\Z$ for some $i\!=\!1,\ldots,k$, then $\Cntr_{(\Ga,\si);e}$ is~0.
In summary, the edge contributions in~\eref{EquivLocal_e} are independent of the ordering of 
the vertices of each edge~$e$ in $\E_{\R}^{\si}(\Ga)\!\sqcup\!\E_+^{\si}(\Ga)$.\\

\noindent
Replacing $e\!\in\!\E_{\C}^{\si}(\Ga)$ by $\si(e)$ in~\eref{ContrEC_e} changes the restriction of 
its right-hand side  to~$\bT^m$ by the factor of~$-1$ to the power~of
$$\sum_{i=1}^k\big(a_i\fd(e)\!+\!1\big)+\big(2\fd(e)\!-\!2\big)+
\big(\fd(e)\!+\!1\big)(n\!-\!2)
\equiv \big(n\!+\!|\a|\big)\fd(e)+\big(n\!+\!k\big) \mod2.$$
If $n\!-\!k\!\in\!2\Z$ and the assumptions of Proposition~\ref{CIorient_prp} are satisfied,
both numbers on the right-hand side above are even.
Thus, replacing an element $e\!\in\!\E_+^{\si}(\Ga)$ with $\si(e)$ has no effect on
the restrictions of~\eref{ContrEC_e} and~\eref{EquivLocal_e} to~$\bT^m$.
By the first sentence in Remark~\ref{EdgCntrVan_rmk}, the same is the case of~\eref{ContrECb_e}. 
In fact, the last conclusion follows from the independence of~\eref{ContrECb_e}
of the choice of~$v\!\in\!e$,
since replacing an element $e\!\in\!\E_+^{\si}(\Ga)$ with $\si(e)$ does not change
the subset $\{\vt(v_1),\vt(v_2)\}$ of~$[n]$ in this~case.\\

\noindent
Replacing $v\!\in\!\Ver$ by $\si(v)$ changes $(-1)^{\fs_v}$ by the factor of~$-1$ to the power~of
\BE{Vch1_e}
\big|S_v^-\big|+\big|S_{\si(v)}^-\big|+
\sum_{\begin{subarray}{c}1\le i\le l\\ i^{\pm}\in S_v\end{subarray}}\!\!\!b_i
+\sum_{\begin{subarray}{c}1\le i\le l\\ i^{\pm}\in S_v\end{subarray}}\!\!\!p_i
=\big|\fm^{-1}(v)\big|\!+\!|\b|_v\!+\!|\p|_v\,.\EE
This replacement changes the restriction of the right-hand side of~\eref{Contr2V_e} 
to~$\bT^m$ by the factor of~$-1$ to the power~of
\BE{Vch22_e}\begin{split}
&(n\!-\!1\!-\!k)\big(|\E_v(\Ga)|\!-\!1\big)+|\b|_v\!+\!|\p|_v+
\big(3\!-\!\val_v(\Ga)\!-\!2|\E_v(\Ga)|\big)\\
& \hspace{1.5in}\equiv \big|\fm^{-1}(v)\big|\!+\!|\b|_v\!+\!|\p|_v
+(n\!-\!k)\big(|\E_v(\Ga)|\!-\!1\big)\mod2.
\end{split}\EE
It changes the restriction of the right-hand side of~\eref{Contr3V_e} 
to~$\bT^m$ by the factor of~$-1$ to the power~of
\BE{Vch2_3}\begin{split}
&(n\!-\!1\!-\!k)\big(|\E_v(\Ga)|\!-\!1\big)+|\b|_v\!+\!|\p|_v+(n\!-\!1)\g(v)-2|\E_v(\Ga)|
-\big(3\g(v)\!-\!3\!+\!|S_v|\big)\\
& \hspace{1.5in}\equiv \big|\fm^{-1}(v)\big|\!+\!|\b|_v\!+\!|\p|_v
+(n\!-\!k)\big(|\E_v(\Ga)|\!-\!1\big)+n\g(v) \mod2.
\end{split}\EE
If $n\!-\!k\!\in\!2\Z$ and the conditions of Theorem~\ref{EquivLocal_thm} are satisfied,
the right-hand sides of~\eref{Vch22_e} and~\eref{Vch2_3} reduce to
the right-hand side of~\eref{Vch1_e}.
Thus, replacing an element $v\!\in\!\nV_+^{\si}(\Ga)$ with $\si(v)$ has no effect on
the restriction of~\eref{EquivLocal_e} to~$\bT^m$.\\

\noindent 
We next observe that the contributions~\eref{EquivLocal_e} for certain admissible pairs 
$(\Ga,\si_1)$ and~$(\Ga,\si_2)$ are opposites of each other if all automorphisms are ignored.
Two examples of such pairs appear in Figure~\ref{discord_fig}.
The cardinalities of $\Aut(\Ga,\si_1)$ and $\Aut(\Ga,\si_2)$
are different  for the first pair (6 and 2, respectively) and
the same for the second pair ($2|\Aut(\Ga_+)|$).

\begin{crl}\label{EquivLocal_crl}
Suppose $n\!\in\!\Z^+$, $g,d,k,l\!\in\!\Z^{\ge0}$ with $n\!-\!k\!\in\!2\Z$ , 
$\a\!\in\!(\Z^+\!-\!2\Z)^k$, and $\phi\!=\!\tau_n'$ or $n\!\in\!2\Z$ and $\phi\!=\!\eta_n$.
If $\phi\!=\!\tau_n'$, assume that $(n,\a)$ satisfies the assumptions
of Proposition~\ref{CIorient_prp}(1);
if $\phi\!=\!\eta_n$, assume that $(m\!\equiv\!n/2,\a)$ satisfies the assumptions
of Proposition~\ref{CIorient_prp}(2).
If $(\Ga,\si_1)$ and~$(\Ga,\si_2)$ are elements of~$\cA_{g,l}(n,d)$ such that 
$k\g(v)\!=\!0$ for all vertices~$v$ of~$\Ga$, 
$$\E_{\R}^{\si_1}(\Ga)\supset \E_{\R}^{\si_2}(\Ga),
\qquad\hbox{and}\qquad 
\big|\E_{\R}^{\si_1}(\Ga)\!-\!\E_{\R}^{\si_2}(\Ga)\big|=2,$$
then 
\BE{EquivLocalCrl}
\fd(e^*)\big|\Aut(\Ga,\si_1)\big|\Cntr_{(\Ga,\si_1)}(\b,\p)
=-\big|\Aut(\Ga,\si_2)\big|\Cntr_{(\Ga,\si_2)}(\b,\p),\EE
where $e^*$ is one of the elements of $\E_{\R}^{\si_1}(\Ga)\!-\!\E_{\R}^{\si_2}(\Ga)$.
\end{crl}

\begin{lmm}\label{EquivLocal_lmm}
Let $S$ be a finite set with involutions~$\si_1$ and~$\si_2$ and
$$S_{\R}^{\si_i}\equiv\big\{s\!\in\!S\!:\,\si_i(s)\!=\!s\big\}.$$
If $S_{\R}^{\si_2}\!=\!\eset$, then
there exists a subset $S_+\!\subset\!S$ such that 
\BE{EquivLocalLmm_e}\big|S_+\!\cap\!S_{\R}^{\si_1}\big|=\frac12\big|S_{\R}^{\si_1}\big|, \quad
S=S_{\R}^{\si_1}\sqcup\big(S_+\!-\!S_{\R}^{\si_1}\big)\sqcup \si_1\big(S_+\!-\!S_{\R}^{\si_1}\big) 
=S_+\sqcup\si_2(S_+).\EE
\end{lmm}

\begin{proof}
Since $\si_2$ acts without fixed points, the cardinalities of~$S$ and $S_{\R}^{\si_1}$ are even.
Let $\si_1'$ be any involution on~$S$ such~that 
$$S_{\R}^{\si_1'}=\eset \qquad\hbox{and}\qquad
\si_1'\big|_{S-S_{\R}^{\si_1}}=\si_1\big|_{S-S_{\R}^{\si_1}}.$$
In particular, $\si_1'$ restricts to an involution on $S_{\R}^{\si_1}$ without fixed points.
Therefore, a subset $S_+\!\subset\!S$ that satisfies~\eref{EquivLocalLmm_e} with~$\si_1$
replaced by~$\si_1'$ also satisfies~\eref{EquivLocalLmm_e} itself.
Thus, it is sufficient to establish the claim under the assumption that $S_{\R}^{\si_1}\!=\!\eset$.\\

\noindent
Suppose we have constructed a subset $S_+'\!\subset\!S$ such that 
$$S_+'\!\cap\!\si_1(S_+'),S_+'\!\cap\!\si_2(S_+')=\eset, \qquad
\big|\si_2(S_+')-S_+'\!\cup\!\si_1(S_+')\big)\big|\le 1.$$
If $S\!=\!S_+'\!\cup\!\si_1(S_+')$, then we can take $S_+\!=\!S_+'$.
If 
$$S\supsetneq S_+'\!\cup\!\si_1(S_+') \qquad\hbox{and}\qquad \si_2(S_+')\subset S_+'\!\cup\!\si_1(S_+'),$$
enlarge $S_+'$ by adding any element from the complement of $\!S_+'\!\cup\!\si_1(S_+')$ in~$S$.
If $\si_2(S_+')$ is not contained in $S_+'\!\cup\!\si_1(S_+')$
and $s$ is the unique element in the complement, 
enlarge $S_+'$ by adding~$\si_1(s)$ to it; 
this element is not in $\si_2(S_+')$ by the uniqueness of~$s$.
After repeating this procedure finitely many times, we obtain a subset $S_+\!\subset\!S$
satisfying~\eref{EquivLocalLmm_e}.
\end{proof}

\begin{proof}[{\bf\emph{Proof of Corollary~\ref{EquivLocal_crl}}}]
Let $\Ga$ be as in~\eref{Dual_e} and
$e_1^*$ and $e_2^*$ be the two edges in the complement of
$\E_{\R}^{\si_2}(\Ga)$ in $\E_{\R}^{\si_1}(\Ga)$. 
We denote the vertices of~$e_1^*$ by $v_{11}^*$ and~$v_{12}^*$ and 
the vertices of~$e_2^*$ by $v_{21}^*$ and~$v_{22}^*$.
Since $(\Ga,\si_1)$ is an admissible graph, $\fd(e_1^*),\fd(e_2^*)\!\not\in\!2\Z$.
The involution~$\si_1$ pairs up the decorated one-edge subgraphs of~$\Ga$ determined  
by the edges in~$\E_{\C}^{\si_1}(\Ga)$ and $\si_2$ pairs up 
the decorated one-edge subgraphs determined by the edges~in
\BE{EquivLocalCrl_e2}
\E_{\C}^{\si_2}(\Ga)=\E_{\C}^{\si_1}(\Ga)\sqcup\big\{e_1^*,e_2^*\big\}.\EE
Thus,
\BE{EquivLocalCrl_e3}
\fd(e_1^*)=\fd(e_2^*),\qquad
\big\{\vt(v_{11}^*),\vt(v_{12}^*)\big\}
=\big\{\phi(\vt(v_{21}^*)),\phi(\vt(v_{22}^*))\big\}
=\big\{\vt(v_{21}^*),\vt(v_{22}^*)\big\};\EE
the last equality holds because $v_{21}^*$ and~$v_{22}^*$ are interchanged by~$\si_1$.
By~\eref{ContrER_e}, \eref{ContrEC_e}, \eref{EquivLocalCrl_e3}, and~\eref{EvenSum_e1},
\BE{EquivLocalCrl_e5}
\Cntr_{(\Ga,\si_1);e_1^*}\Cntr_{(\Ga,\si_1);e_2^*}
=-\frac{\Cntr_{(\Ga,\si_2);e_1^*}}{\fd(e_1^*)}
=-\frac{\Cntr_{(\Ga,\si_2);e_2^*}}{\fd(e_2^*)}\EE
when restricted to the subtorus $\bT^m\!\subset\!\bT^n$.\\

\noindent
Since $(\Ga,\si_1)$ and $(\Ga,\si_2)$ are admissible pairs,  $\si_1$ and~$\si_2$
act on~$\Ver$ without fixed points.
By Lemma~\ref{EquivLocal_lmm}, there thus exists a subset $\nV_+^{\si_1\si_2}(\Ga)$ 
of~$\Ver$ that satisfies~\eref{nVplus_e} for $\si\!=\!\si_1$ and $\si\!=\!\si_2$
at the same time.
The involutions $\si_1$ and $\si_2$ restrict to involutions on~\eref{EquivLocalCrl_e2}
such that $\si_2$ acts without fixed points and the fixed points of~$\si_1$ are $\{e_1^*,e_2^*\}$.
By Lemma~\ref{EquivLocal_lmm}, there thus exists a subset $\E_+^{\si_2}(\Ga)$ 
of~$\E_{\C}^{\si_2}(\Ga)$ satisfying~\eref{nEplus_e} for $\si\!=\!\si_2$  
such~that 
$$\E_+^{\si_1}(\Ga)\equiv\E_+^{\si_2}(\Ga)-\big\{e_1^*,e_2^*\big\} $$ 
satisfies~\eref{nEplus_e} for $\si\!=\!\si_1$.
Let $e_i^*$ be the unique element of $\E_+^{\si_2}(\Ga)\!\cap\!\{e_1^*,e_2^*\}$.
The vertex and edge contributions on the right-hand side of~\eref{EquivLocal_e} are then
the same for $\si\!=\!\si_1$ and $\si\!=\!\si_2$, except the contribution of
$e_i^*\!\in\!\E_+^{\si_2}(\Ga)$ is replaced by the product of the contributions
of $e_1^*,e_2^*\!\in\!\E_{\R}^{\si_1}(\Ga)$.
The claim now follows from~\eref{EquivLocalCrl_e5}.
\end{proof}

\begin{proof}[{\bf\emph{Proof of Theorem~\ref{GWsCI_thm}\ref{CIvan_it} for~$k\!=\!0$}}]
Let $(\Ga,\si)$ be an element of $\cA_{g,l}(n,d)$. 
Since $\nV_{\R}^{\si}(\Ga)\!=\!\eset$,
\BE{PnEGthm_e1}
g=g(\Ga)=1+\big|\E_{\R}^{\si}(\Ga)\big|+2\big|\E_+^{\si}(\Ga)\big|
 +2\!\!\!\sum_{v\in\nV_+^{\si}(\Ga)}\!\!\!\!\!\!\big(\g(v)\!-\!1\big)
\equiv 1+\big|\E_{\R}^{\si}(\Ga)\big| \mod2.\EE
Since $\fd(e)\!\not\in\!2\Z$ for all $v\!\in\!\E_{\R}^{\si}(\Ga)$,
\BE{PnEGthm_e3}
d=d(\Ga)=\sum_{e\in\E_{\R}^{\si}(\Ga)}\!\!\!\!\!\fd(e)
+2\!\!\!\!\!\sum_{e\in\E_+^{\si}(\Ga)}\!\!\!\!\!\!\fd(e)
\equiv \big|\E_{\R}^{\si}(\Ga)\big| \mod2.\EE
By~\eref{PnEGthm_e1} and~\eref{PnEGthm_e3}, $\cA_{g,l}(n,d)\!=\!\eset$ if $d\!-\!g\!\in\!2\Z$.
Since all equivariant contributions to the genus~$g$ degree~$d$ real GW-invariants of 
$(\P^{2m-1},\om_{2m},\tau_{2m})$ and 
$(\P^{4m-1},\om_{4m},\eta_{4m})$ with only conjugate pairs of insertions 
arise from the elements of $\cA_{g,l}(n,d)$, with $n\!=\!2m$ and $n\!=\!4m$, respectively,
this establishes the claim.
\end{proof}

\begin{proof}[{\bf\emph{Proof of Proposition~\ref{GWsPn_prp}}}]
By Theorem~\ref{EquivLocal_thm}, the genus~$g$ real GW-invariants of 
$(\P^{4n-1},\om_{4n},\tau_{4n})$ and $(\P^{4n-1},\om_{4n},\eta_{4n})$ 
with only conjugate pairs of insertions are obtained by summing the contributions
from the same set of admissible pairs~$(\Ga,\si)$.
The contributions~\eref{EquivLocal_e} of~$(\Ga,\si)$ to the two invariants
are products of the factors~\eref{Contr3V_e}-\eref{ContrER_e}.
The factors~\eref{ContrER_e} corresponding to the $\si$-fixed edges of~$\Ga$
have opposite signs in the two cases;
all other factors are the same.
By~\eref{PnEGthm_e1}, the parity of~$|\E_{\R}^{\si}(\Ga)|$ is the same as the parity
of $g\!-\!1$.
Thus, the contributions~\eref{EquivLocal_e}  to the two invariants
from every admissible pair~$(\Ga,\si)$ differ by the factor of~$(-1)^{g-1}$;
this establishes the claim.
\end{proof}

\begin{eg}[$d\!=\!2$]\label{d2_eg}
We now apply Theorem~\ref{EquivLocal_thm} to compute the genus~$g$ degree~2
real GW-invariants of~$(\P^3,\tau_4)$ with 2 conjugate pairs of point constraints.
They are given~by
\BE{d2eg_e}\GW_{g,2}^{\P^3,\tau_4}\big(H^3,H^3\big)
=\int_{[\ov\fM_{g,2}(\P^3,2)^{\tau_4}]^{\vrt}}
\ev_1^{\,*}\!\prod_{j\neq 1}\!(\x\!-\!\al_j)
~\ev_2^{\,*}\!\prod_{j\neq 3}\!(\x\!-\!\al_j),\EE
where $H\!\in\!H^2(\P^3;\Q)$ is the usual hyperplane class and 
$\x\!\in\!H_{\bT^4}^2(\P^3;\Q)$ is the equivariant hyperplane class.
If $(\Ga,\si)\!\in\!\cA_{g,2}(4,2)$ and $\Ga$ is as in~\eref{Dual_e}, then
\BE{d2eg_e3}\ev_i^{\,*}\!\!\!\!\!\prod_{j\neq 2i-1}\!\!\!\!(\x\!-\!\al_j)~\big|_{\cZ_{\Ga,\si}}
=\prod_{j\neq2i-1}\!\!\!\!\big(\al_{\vt(\fm(i^+))}\!-\!\al_j\big)
\qquad \forall~i\!=\!1,2,\EE
where $\cZ_{\Ga,\si}$ is the $\bT^2$-fixed locus corresponding to $(\Ga,\si)$;
this restriction is formally encoded into the vertex contribution,
\eref{Contr3V_e} or~\eref{Contr2V_e}, of $v\!=\!\fm(i^+)$.
Thus, the restriction of the integrand in~\eref{d2eg_e} to~$\cZ_{\Ga,\si}$   vanishes unless
$$\vt(\fm(1^+))=1, \quad \vt(\fm(2^+))=3, \quad 
\vt(\fm(1^-))=2, \quad \vt(\fm(2^-))=4.$$
Since there are no degree~2 connected graphs with at least~4 vertices,
the restriction  of the integrand in~\eref{d2eg_e} to all $\bT^2$-fixed loci
vanishes and~so
$$\GW_{g,2}^{\P^3,\tau_4}(H^3,H^3)=0 \qquad\forall~g\!\in\!\Z.$$
\end{eg}

\begin{eg}[$g\!=\!1,d\!=\!4$]\label{g1d4_eg}
We next compute the genus~1 degree~4
real GW-invariant of~$(\P^3,\tau_4)$ with 4 conjugate pairs of point constraints~as
\BE{g1d4eg_e}\GW_{1,4}^{\P^3,\tau_4}\big(H^3,H^3,H^3,H^3\big)=
\int_{[\ov\fM_{1,4}(\P^3,4)^{\tau_4}]^{\vrt}}
\prod_{i=1}^4\!\bigg(\!\ev_i^{\,*}\!\prod_{j\neq i}\!(\x\!-\!\al_j)\!\bigg).\EE
Similarly to Example~\ref{d2_eg}, 
the restriction of the integrand in~\eref{g1d4eg_e} to a fixed locus~$\cZ_{\Ga,\si}$   vanishes unless
$\vt(\fm(i^+))\!=\!i$ for all $i\!\in\![4]$.
There are 11~pairs in $\cA_{1,4}(4,4)$ satisfying this condition:
the first diagram in Figure~\ref{g1d4_fig},  with the four possible ways of labeling its vertices
and  the two possible involutions on the loop, and the three other diagrams.
For each of these diagrams,
$$ \ev_i^{\,*}\!\prod_{j\neq i}\!(\x\!-\!\al_j)~\big|_{\cZ_{\Ga,\si}}
=\prod_{j\neq i}\!\big(\al_{\vt(\fm(i^+))}\!-\!\al_j\big)
=\e\big(T_{P_i}\P^3\big);$$
see~\eref{EquivRestr_e}.
All eight versions of the first diagram in Figure~\ref{g1d4_fig} have the same 
automorphism group, i.e.~$\Z_2$.
Since the degrees of the vertical edges are~1, Corollary~\ref{EquivLocal_crl}
thus implies that these graphs cancel in pairs.
For the remaining three graphs, we can choose the same distinguished subset 
$\nV_+^{\si}(\Ga)$ of~$\Ver$ consisting of the top vertices.
Their contributions in the three cases are given by~\eref{Contr3V_e} with
$\al_{\vt(v)}^{|\p|_v}$ replaced by $\e(T_{P_{\vt(v)}}\P^3)^2$
and $\ov\cM_{\g(v),S_v}\!=\!\ov\cM_{0,4}$:
\begin{alignat*}{3}
&\e\big(T_{P_1}\P^3\big)^3\frac{3\la_1\!-\!\la_2}{8\la_1^3(\la_1\!-\!\la_2)^3}\,,
&\qquad 
&\e\big(T_{P_1}\P^3\big)^3\frac{3\la_1\!+\!\la_2}{8\la_1^3(\la_1\!+\!\la_2)^3}\,,
&\qquad 
\e\big(T_{P_1}\P^3\big)^3\frac{2\la_1}{(\la_1^2\!-\!\la_2^2)^3}\,,\\
&\e\big(T_{P_3}\P^3\big)^3\frac{3\la_2\!-\!\la_1}{8\la_2^3(\la_2\!-\!\la_1)^3}\,,
&\qquad
&-\e\big(T_{P_4}\P^3\big)^3\frac{3\la_2\!+\!\la_1}{8\la_2^3(\la_2\!+\!\la_1)^3}\,,
&\qquad
\e\big(T_{P_3}\P^3\big)^3\frac{2\la_2}{(\la_2^2\!-\!\la_1^2)^3}\,;
\end{alignat*}
see also~\eref{sp_weights_e}.
In the case of the two middle diagrams, the set $\E_{\R}^{\si}(\Ga)$ consists of the two vertical edges.
The product of their contributions, as given by~\eref{ContrER_e}, is $-(\la_1^2\!-\!\la_2^2)^{-2}$
in both cases.
For the set $\E_+^{\si}(\Ga)$, we can choose the top edge in these cases.
Its contributions,  as given by~\eref{ContrEC_e}, are
\BE{g1d4eg_e7}-\frac{1}{4\la_1\la_2(\la_1\!+\!\la_2)^2} \qquad\hbox{and}\qquad
\frac{1}{4\la_1\la_2(\la_1\!-\!\la_2)^2}\,,\EE
respectively.
In the case of the last diagram,  we can take $\E_+^{\si}(\Ga)$ to consist of the top and
left edges; their contributions are given by~\eref{g1d4eg_e7}.
In all three cases, the vertex signs $(-1)^{\fs_v}$ are the same for the two vertices.
Putting the  three contributions together, we obtain
\begin{equation*}\begin{split}
\GW_{1,4}^{\P^3,\tau_4}\big(H^3,H^3,H^3,H^3\big)=
\frac{(\la_1\!+\!\la_2)^2(3\la_1\!-\!\la_2)(3\la_2\!-\!\la_1)}{4\la_1\la_2(\la_1\!-\!\la_2)^2}
&+\frac{(\la_1\!-\!\la_2)^2(3\la_1\!+\!\la_2)(3\la_2\!+\!\la_1)}{4\la_1\la_2(\la_1\!+\!\la_2)^2}\\
&-\frac{16\la_1^2\la_2^2}{(\la_1^2\!-\!\la_2^2)^2}
=-1.
\end{split}\end{equation*}
\end{eg}

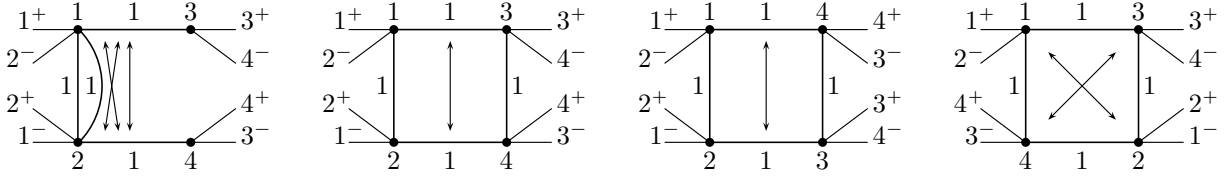
\begin{figure}
\begin{pspicture}(-1.5,-1.5)(10,1.8)
\psset{unit=.3cm}
\pscircle*(-1,2.5){.2}\pscircle*(4,2.5){.2}
\pscircle*(-1,-2.5){.2}\pscircle*(4,-2.5){.2}
\psline[linewidth=.07](-1,2.5)(4,2.5)\psline[linewidth=.07](-1,-2.5)(4,-2.5)
\psline[linewidth=.07](-1,2.5)(-1,-2.5)\psarc[linewidth=.07](-3.5,0){3.57}{-45}{45}
\psline[linewidth=.05]{<->}(.2,2)(.8,-2)\psline[linewidth=.05]{<->}(.8,2)(.2,-2)
\psline[linewidth=.05]{<->}(1.3,2)(1.3,-2)
\rput(-1,3.3){\smsize{$1$}}\rput(-1,-3.3){\smsize{$2$}}
\rput(4,3.3){\smsize{$3$}}\rput(4,-3.3){\smsize{$4$}}
\rput(1.5,3.3){\smsize{$1$}}\rput(1.5,-3.3){\smsize{$1$}}
\rput(-1.5,0){\smsize{$1$}}\rput(-.4,0){\smsize{$1$}}
\psline[linewidth=.05](-1,2.5)(-3,2.5)\psline[linewidth=.05](-1,-2.5)(-3,-2.5)
\psline[linewidth=.05](4,2.5)(6,2.5)\psline[linewidth=.05](4,-2.5)(6,-2.5)
\psline[linewidth=.05](-1,2.5)(-3,1)\psline[linewidth=.05](-1,-2.5)(-3,-1)
\psline[linewidth=.05](4,2.5)(6,1)\psline[linewidth=.05](4,-2.5)(6,-1)
\rput(-3,3){\smsize{$1^+$}}\rput(-3,-2.1){\smsize{$1^-$}}
\rput(-3.5,1.3){\smsize{$2^-$}}\rput(-3.5,-.7){\smsize{$2^+$}}
\rput(6.9,3){\smsize{$3^+$}}\rput(6.9,-2.1){\smsize{$3^-$}}
\rput(6.9,1.3){\smsize{$4^-$}}\rput(6.9,-.7){\smsize{$4^+$}}
\pscircle*(13,2.5){.2}\pscircle*(18,2.5){.2}
\pscircle*(13,-2.5){.2}\pscircle*(18,-2.5){.2}
\psline[linewidth=.07](13,2.5)(18,2.5)\psline[linewidth=.07](13,-2.5)(18,-2.5)
\psline[linewidth=.07](13,2.5)(13,-2.5)\psline[linewidth=.07](18,2.5)(18,-2.5)
\psline[linewidth=.05]{<->}(15.5,2)(15.5,-2)
\rput(13,3.3){\smsize{$1$}}\rput(13,-3.3){\smsize{$2$}}
\rput(18,3.3){\smsize{$3$}}\rput(18,-3.3){\smsize{$4$}}
\rput(15.5,3.3){\smsize{$1$}}\rput(15.5,-3.3){\smsize{$1$}}
\rput(12.5,0){\smsize{$1$}}\rput(18.5,0){\smsize{$1$}}
\psline[linewidth=.05](13,2.5)(11,2.5)\psline[linewidth=.05](13,-2.5)(11,-2.5)
\psline[linewidth=.05](18,2.5)(20,2.5)\psline[linewidth=.05](18,-2.5)(20,-2.5)
\psline[linewidth=.05](13,2.5)(11,1)\psline[linewidth=.05](13,-2.5)(11,-1)
\psline[linewidth=.05](18,2.5)(20,1)\psline[linewidth=.05](18,-2.5)(20,-1)
\rput(11,3){\smsize{$1^+$}}\rput(11,-2.1){\smsize{$1^-$}}
\rput(10.5,1.3){\smsize{$2^-$}}\rput(10.5,-.7){\smsize{$2^+$}}
\rput(20.9,3){\smsize{$3^+$}}\rput(20.9,-2.1){\smsize{$3^-$}}
\rput(20.9,1.3){\smsize{$4^-$}}\rput(20.9,-.7){\smsize{$4^+$}}
\pscircle*(27,2.5){.2}\pscircle*(32,2.5){.2}
\pscircle*(27,-2.5){.2}\pscircle*(32,-2.5){.2}
\psline[linewidth=.07](27,2.5)(32,2.5)\psline[linewidth=.07](27,-2.5)(32,-2.5)
\psline[linewidth=.07](27,2.5)(27,-2.5)\psline[linewidth=.07](32,2.5)(32,-2.5)
\psline[linewidth=.05]{<->}(29.5,2)(29.5,-2)
\rput(27,3.3){\smsize{$1$}}\rput(27,-3.3){\smsize{$2$}}
\rput(32,3.3){\smsize{$4$}}\rput(32,-3.3){\smsize{$3$}}
\rput(29.5,3.3){\smsize{$1$}}\rput(29.5,-3.3){\smsize{$1$}}
\rput(26.5,0){\smsize{$1$}}\rput(32.5,0){\smsize{$1$}}
\psline[linewidth=.05](27,2.5)(25,2.5)\psline[linewidth=.05](27,-2.5)(25,-2.5)
\psline[linewidth=.05](32,2.5)(34,2.5)\psline[linewidth=.05](32,-2.5)(34,-2.5)
\psline[linewidth=.05](27,2.5)(25,1)\psline[linewidth=.05](27,-2.5)(25,-1)
\psline[linewidth=.05](32,2.5)(34,1)\psline[linewidth=.05](32,-2.5)(34,-1)
\rput(25,3){\smsize{$1^+$}}\rput(25,-2.1){\smsize{$1^-$}}
\rput(24.5,1.3){\smsize{$2^-$}}\rput(24.5,-.7){\smsize{$2^+$}}
\rput(34.9,3){\smsize{$4^+$}}\rput(34.9,-2.1){\smsize{$4^-$}}
\rput(34.9,1.3){\smsize{$3^-$}}\rput(34.9,-.7){\smsize{$3^+$}}
\pscircle*(41,2.5){.2}\pscircle*(46,2.5){.2}
\pscircle*(41,-2.5){.2}\pscircle*(46,-2.5){.2}
\psline[linewidth=.07](41,2.5)(46,2.5)\psline[linewidth=.07](41,-2.5)(46,-2.5)
\psline[linewidth=.07](41,2.5)(41,-2.5)\psline[linewidth=.07](46,2.5)(46,-2.5)
\psline[linewidth=.05]{<->}(42,1.5)(45,-1.5)\psline[linewidth=.05]{<->}(42,-1.5)(45,1.5)
\rput(41,3.3){\smsize{$1$}}\rput(41,-3.3){\smsize{$4$}}
\rput(46,3.3){\smsize{$3$}}\rput(46,-3.3){\smsize{$2$}}
\rput(43.5,3.3){\smsize{$1$}}\rput(43.5,-3.3){\smsize{$1$}}
\rput(40.5,0){\smsize{$1$}}\rput(46.5,0){\smsize{$1$}}
\psline[linewidth=.05](41,2.5)(39,2.5)\psline[linewidth=.05](41,-2.5)(39,-2.5)
\psline[linewidth=.05](46,2.5)(48,2.5)\psline[linewidth=.05](46,-2.5)(48,-2.5)
\psline[linewidth=.05](41,2.5)(39,1)\psline[linewidth=.05](41,-2.5)(39,-1)
\psline[linewidth=.05](46,2.5)(48,1)\psline[linewidth=.05](46,-2.5)(48,-1)
\rput(39,3){\smsize{$1^+$}}\rput(39,-2.1){\smsize{$3^-$}}
\rput(38.5,1.3){\smsize{$2^-$}}\rput(38.5,-.7){\smsize{$4^+$}}
\rput(48.9,3){\smsize{$3^+$}}\rput(48.9,-2.1){\smsize{$1^-$}}
\rput(48.9,1.3){\smsize{$4^-$}}\rput(48.9,-.7){\smsize{$2^+$}}
\end{pspicture}
\caption{The elements of~$\cA_{1,4}(4,4)$ potentially contributing to~\eref{g1d4eg_e}.}
\label{g1d4_fig}
\end{figure}

\section{Proof of Theorem~\ref{EquivLocal_thm}}
\label{EquivLocalPf_sec}

\noindent
It remains to establish Theorem~\ref{EquivLocal_thm}.
For the remainder of this paper, we assume 
$$n\!\in\!\Z^+, \quad g,d,k,l\!\in\!\Z^{\ge0}, \quad 
n\!-\!k\!\in\!2\Z, \quad \a\!\in\!(\Z^+)^k,$$
and $\phi\!=\!\tau_n'$ or $n\!\in\!2\Z$ and $\phi\!=\!\eta_n$.
We also assume  that 
$(n,\a)$ satisfies the assumptions of Proposition~\ref{CIorient_prp}(1)
if  $\phi\!=\!\tau_n'$
and of  Proposition~\ref{CIorient_prp}(2)  with $(m\!=\!n/2,\a)$ if $\phi\!=\!\eta_n$.
Let $m\!=\!\flr{n/2}$ as before and $S_l$ be as in~\eref{Sldfn_e}.

\subsection{The torus-fixed loci}
\label{FixedLoci_subs}

\noindent
We first identify the topological components of  the fixed locus of 
the $\bT^m$-action on $\ov\fM_{g,l}(\P^{n-1},d)^{\phi}$.\\

\noindent
Similarly to the situation in \cite[Section~27.3]{MirSym}, an element 
\BE{umapdfn_e} [\u]\equiv \big[\Si,(z_1^+,z_1^-),\ldots,(z_l^+,z_l^-),\si,u\big]
\in \ov\fM_{g,l}(\P^{n-1},d)^{\phi}\EE
is fixed by the $\bT^m$-action if and only~if
\begin{enumerate}[label=(F\arabic*),leftmargin=*]

\item\label{TmMap_it} the image of every irreducible component of $\Si$ is either a $\bT^m$-fixed 
point or a $\bT^m$-invariant irreducible curve in~$\P^{n-1}$,

\item the image of every nodal and marked point of~$\Si$ is a $\bT^m$-fixed 
point  in~$\P^{n-1}$,

\item\label{TmBP_it} the image of every branch point of the restriction $u$ to an irreducible component
of~$\Si$ is  a $\bT^m$-fixed  point  in~$\P^{n-1}$.\\

\end{enumerate}

\noindent
The $\bT^m$-fixed points and  $\bT^m$-invariant irreducible curves in~$\P^{n-1}$
are described by Lemma~\ref{FixedCurves_lmm}.
For a $\bT^m$-invariant stable map as in~\eref{umapdfn_e}, 
every non-constant restriction of the map~$u$  to an irreducible component of~$\Si$
is thus a cover of a line $\P^1_{ij}$ with $i\!\neq\!j$ branched only over 
the points~$P_i$ and~$P_j$ or 
of a conic $\cC_i(a,b)$ with $a,b\!\in\!\C^*$ branched only over the points~$P_{2i-1}$ 
and~$P_{2i}$;
the latter is a possibility only if $n\!=\!2m\!+\!1$.
Since $\P^1_{ij}$ and~$\cC_i(a,b)$ are smooth rational curves,
the Riemann-Hurwitz formula \cite[p219]{GH} implies that the domain of 
any irreducible cover of either $\P^1_{ij}$ or~$\cC_i(a,b)$ branched only over two points 
is also a~$\P^1$.\\

\noindent
The combinatorial structure of a $\bT^m$-invariant stable map as in~\eref{umapdfn_e}
can thus be described
by a connected decorated graph~$\Ga$ as in~\eref{Dual_e}.
The irreducible components~$\Si_e$ of~$\Si$ on which the map~$u$ is not constant
are rational and correspond to the edges $e\!\in\!\Edg$.
For $e\!=\!\{v_1,v_2\}$, $u|_{\Si_e}$ is either a degree~$\fd(e)$ cover of the line 
$\P^1_{\vt(v_1),\vt(v_2)}$
or a degree~$\fd(e)/2$ cover of a conic~$\cC_{\flr{(\vt(v_1)+1)/2}}(a,b)$;
the latter is a possibility only if~\eref{ConicsCovCond_e} holds. 
In both cases, the map $u|_{\Si_e}$ is ramified only over $P_{\vt(v_1)}$ and~$P_{\vt(v_2)}$.
We denote the moduli space of all possible $u|_{\Si_e}$ and its closure~by 
\BE{EdgeSpaces_e}\begin{aligned}
&\fM_{\Ga,\si;e}^{\phi;\bT;\circ}\subset \ov\fM_{\Ga,\si;e}^{\phi;\bT;\circ}\subset
\ov\fM_{0,0}\big(\P^{n-1},\fd(e)\big)^{\phi}
&\qquad &\hbox{if}~e\!\in\!\E_{\R}^{\si}(\Ga),\\
&\fM_{\Ga;e}^{\bT;\circ}\subset \ov\fM_{\Ga;e}^{\bT;\circ}\subset
\ov\fM_{0,0}\big(\P^{n-1},\fd(e)\big)
&\qquad &\hbox{if}~e\!\in\!\E_{\C}^{\si}(\Ga).
\end{aligned}\EE
We denote~by 
$$\cN_{\Ga,\si;e}^{\phi;\circ}\lra  \ov\fM_{\Ga,\si;e}^{\phi;\bT;\circ}
\qquad\hbox{and}\qquad
\cN_{\Ga;e}^{\circ}\lra  \ov\fM_{\Ga;e}^{\bT;\circ}$$
the corresponding normal bundles.\\ 

\noindent
The vertices $v\!\in\!\Ver$ with $\val_v(\Ga)\!\ge\!3$ correspond to 
the maximal connected unions~$\Si_v$ of irreducible components of~$\Si$ 
on which~$u$ is constant.
The arithmetic genus of such~$\Si_v$ is~$\g(v)$;
it is sent by~$u$ to~$P_{\vt(v)}$ and carries the marked points $\fm^{-1}(v)\!\subset\!S_l$.
If $v\!\in\!\nV_{\C}^{\si}(\Ga)$, we denote the moduli space of all possible~$\Si_v$ by
\BE{ovfMGavdfn_e}\ov\fM_{\Ga;v}^{\bT}\subset \ov\fM_{\g(v),S_v}\big(\P^{n-1},0\big)\EE
with~$S_v$ as in~\eref{Svdfn_e}; it is isomorphic to $\ov\cM_{\g(v),S_v}$. 
The remaining marked points of~$\Si$ are the branch points of~$u|_{\Si_e}$
corresponding to the vertices $v\!\in\!e$ with $\val_v(\Ga)\!=\!2$ and $|\E_v(\Ga)|\!=\!1$.
The remaining vertices of~$v$ with $\val_v(\Ga)\!=\!2$ correspond to the nodes of~$\Si$
shared by two irreducible components~$\Si_{e_1}$ and~$\Si_{e_2}$ with $e_1,e_2\!\in\!\Edg$.
The involution~$\si$ on~$\Si$ induces an involution~$\si$ on the graph~$\Ga$.
We will call the pair $(\Ga,\si)$ obtained in this way the \sf{combinatorial type of}
the $\bT^m$-fixed stable map~\eref{umapdfn_e}.

\begin{rmk}\label{ConicsSpaces_rmk}
If $e$ does not satisfy~\eref{ConicsCovCond_e}, then the spaces in~\eref{EdgeSpaces_e} 
consist of a single element with the automorphism group~$\Z_{\fd(e)}$. 
Let $i\!=\!\flr{(\vt(v_1)\!+\!1)/2}$.
If $e\!\in\!\E_{\C}^{\si}(\Ga)$ satisfies~\eref{ConicsCovCond_e}, then
$$\ov\fM_{\Ga;e}^{\bT;\circ}\approx \big\{[a,b]\!\in\!\P^1\!\big\}$$
as topological spaces;
see Lemma~\ref{FixedCurves_lmm}\ref{Tcurves_it}
This identification can be chosen so that the image of the map corresponding
to $[a,b]$ is the conic $\cC_i(a,b)$.
The points $[1,0]$ and $[0,1]$ then correspond to the covers of $\P^1_{01}\!\cup\!\P^1_{02}$
and~$\P^1_{12}$, respectively.
The automorphism groups of these points are $(\Z_{\fd(e)/2})^2$
and $\Z_{\fd(e)}$, respectively;
the automorphism groups of the remaining elements are~$\Z_{\fd(e)/2}$.
By Lemma~\ref{FixedCurves_lmm}\ref{TRcurves_it}, the space of
real $\bT^1$-invariant conics (not necessarily smooth) in $(\P^2,\tau_3)$~is
$$\big\{[a,b]\!\in\!\P^1\!:\,a\bar{b}\!\in\!\R\big\}\approx S^1\subset\P^1\,.$$
Suppose $e\!\in\!\E_{\R}^{\si}(\Ga)$ satisfies~\eref{ConicsCovCond_e}.
The  $\bT^1$-invariant degree~$\fd(e)$ cover of~$\P^1_{12}$ is then compatible with both involutions
on the domain;
the automorphism groups of both resulting real  covers are~$\Z_{\fd(e)}$.
The  $\bT^1$-invariant degree $\fd(e)/2$ cover of $\P^1_{01}\!\cup\!\P^1_{02}$ is 
compatible with one involution on the domain;
the automorphism group of the resulting real cover is $\Z_{\fd(e)/2}$.
If $\fd(e)\!\not\in\!4\Z$,  the same is the case for 
the  $\bT^1$-invariant degree $\fd(e)/2$ cover of each of 
the smooth real conics~$\cC_i(a,b)$.
If $\fd(e)\!\in\!4\Z$, the degree $\fd(e)/2$ covers of 
the conics~$\cC_i(a,b)$ with $a\bar{b}\!\in\!\R^-$
are not compatible with any involution
on the domain as these conics have no fixed locus;
see Remark~\ref{conics_rmk} and \cite[Lemma~1.9]{Teh}.
The  $\bT^1$-invariant degree $\fd(e)/2$ cover of each 
conic~$\cC_i(a,b)$ with $a\bar{b}\!\in\!\R^+$
is  compatible with both involutions on the domain;
the automorphism groups of both resulting real  covers are~$\Z_{\fd(e)/2}$.
In both cases, $\ov\fM_{\Ga,\si;e}^{\phi;\bT;\circ}$ can be viewed as the interval $[-1,1]$
with the trivial $\Z_{\fd(e)/2}$-action on the interior points and 
the trivial $\Z_{\fd(e)}$-action on the endpoints;
as an orbifold, it has no boundary.
It can alternatively be viewed as the quotient of $S^1\!\subset\!\C^*$ by 
the $\Z_{\fd(e)}$-action generated by the conjugation.\\
\end{rmk}

\noindent
Let $\Ga$ be an $S_l$-marked $[n]$-decorated connected graph with  an involution~$\si$
such that $g(\Ga)\!=\!g$ and $d(\Ga)\!=\!d$.
We denote~by
$$\cZ_{\Ga,\si}\subset \ov\fM_{g,l}(\P^{n-1},d)^{\phi}$$ 
the subspace consisting of all $\bT^m$-fixed elements of the combinatorial type 
of~$(\Ga,\si)$.
This subspace is closed unless some edge $e\!\in\!\Edg$ satisfies~\eref{ConicsCovCond_e}.
In such a case, the closure~$\ov\cZ_{\Ga,\si}$ of~$\cZ_{\Ga,\si}$ also includes the subspaces
$\cZ_{\Ga',\si'}$ corresponding to the pairs $(\Ga',\si')$ obtained from~$(\Ga,\si)$ by
the ``local replacement" of \cite[Figure~5]{Wal}: 
\begin{enumerate}[label=$\bu$,leftmargin=*]

\item adding a new vertex~$v_e$ to an edge~$e$ satisfying~\eref{ConicsCovCond_e},

\item  extending the functions~$\g$ and~$\vt$ to~$v_e$ by~0 and~$n$, respectively, and 

\item replacing the value of~$\fd$ on~$e$ by the values of~$\fd(e)/2$ on each of the two new edges;

\end{enumerate}
see Figure~\ref{GaDegen_fig}.
If $e\!\in\!\E_{\R}^{\si}(\Ga)$, the involution~$\si'$ is obtained from~$\si$ by sending~$v_e$ 
to itself and interchanging the two new edges.
If $e\!\in\!\E_{\C}^{\si}(\Ga)$, the above breaking procedure should simultaneously be performed
on the edge~$\si(e)$.
The involution~$\si'$ on~$\Ga'$ is then obtained from~$\si$ by interchanging~$v_e$
with~$v_{\si(e)}$ and the two pairs of new edges according to the action on their vertices.
This graph degeneration corresponds to the degeneration of the conics~$\cC_i(a,b)$ 
to the union of the lines~$\P^1_{2i-1,n}$ and~$\P^1_{2i,n}$.
Denote~by
$$\cN_{\Ga,\si}^{\phi}\lra \ov\cZ_{\Ga,\si}$$
the normal bundle of $\ov\cZ_{\Ga,\si}$ in $\ov\fM_{g,l}(\P^{n-1},d)^{\phi}$.\\ 


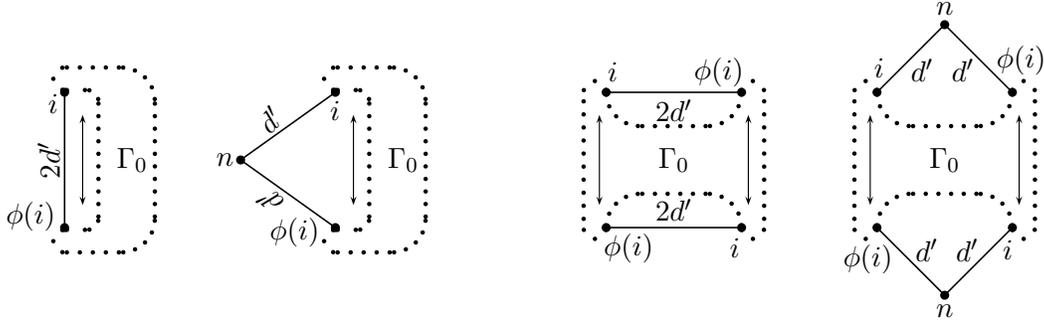
\begin{figure}
\begin{pspicture}(-1.5,-2.5)(10,1.8)
\psset{unit=.3cm}
\pscircle*(2,3){.2}\pscircle*(2,-3){.2}
\psline[linewidth=.07](2,3)(2,-3)
\psline[linewidth=.2,linestyle=dotted](2,3.1)(3,3.1)
\psline[linewidth=.2,linestyle=dotted](2,-3.1)(3,-3.1)
\psline[linewidth=.2,linestyle=dotted](2,4.1)(4.5,4.1)
\psline[linewidth=.2,linestyle=dotted](2,-4.1)(4.5,-4.1)
\psarc[linewidth=.2,linestyle=dotted](4.5,2.6){1.5}{0}{90}
\psarc[linewidth=.2,linestyle=dotted](4.5,-2.6){1.5}{-90}{0}
\psarc[linewidth=.2,linestyle=dotted](2,3.6){.5}{90}{270}
\psarc[linewidth=.2,linestyle=dotted](2,-3.6){.5}{90}{270}
\psarc[linewidth=.2,linestyle=dotted](3,2.6){.5}{0}{90}
\psarc[linewidth=.2,linestyle=dotted](3,-2.6){.5}{-90}{0}
\psline[linewidth=.2,linestyle=dotted](3.5,2.6)(3.5,-2.6)
\psline[linewidth=.2,linestyle=dotted](6,2.6)(6,-2.6)
\rput(1.5,2.5){$i$}\rput(.5,-2.5){$\phi(i)$}
\rput{90}(1.3,0){$2d'$}\rput(5,0){$\Ga_0$}
\psline[linewidth=.05]{<->}(2.8,2)(2.8,-2)
\pscircle*(14,3){.2}\pscircle*(14,-3){.2}\pscircle*(9.8,0){.2}
\psline[linewidth=.07](14,3)(9.8,0)
\psline[linewidth=.07](14,-3)(9.8,0)
\psline[linewidth=.2,linestyle=dotted](14,3.1)(15,3.1)
\psline[linewidth=.2,linestyle=dotted](14,-3.1)(15,-3.1)
\psline[linewidth=.2,linestyle=dotted](14,4.1)(16.5,4.1)
\psline[linewidth=.2,linestyle=dotted](14,-4.1)(16.5,-4.1)
\psarc[linewidth=.2,linestyle=dotted](16.5,2.6){1.5}{0}{90}
\psarc[linewidth=.2,linestyle=dotted](16.5,-2.6){1.5}{-90}{0}
\psarc[linewidth=.2,linestyle=dotted](14,3.6){.5}{90}{270}
\psarc[linewidth=.2,linestyle=dotted](14,-3.6){.5}{90}{270}
\psarc[linewidth=.2,linestyle=dotted](15,2.6){.5}{0}{90}
\psarc[linewidth=.2,linestyle=dotted](15,-2.6){.5}{-90}{0}
\psline[linewidth=.2,linestyle=dotted](15.5,2.6)(15.5,-2.6)
\psline[linewidth=.2,linestyle=dotted](18,2.6)(18,-2.6)
\rput(14,2.2){$i$}\rput(12.2,-3.2){$\phi(i)$}\rput(9.1,0){$n$}
\rput{45}(11.1,1.8){$d'$}\rput{135}(11.1,-1.8){$d'$}
\rput(17,0){$\Ga_0$}
\psline[linewidth=.05]{<->}(14.8,2)(14.8,-2)
\pscircle*(26,3){.2}\pscircle*(32,3){.2}
\pscircle*(26,-3){.2}\pscircle*(32,-3){.2}
\psline[linewidth=.07](26,3)(32,3)
\psline[linewidth=.07](26,-3)(32,-3)
\psarc[linewidth=.2,linestyle=dotted](27.5,3){1.5}{180}{270}
\psarc[linewidth=.2,linestyle=dotted](27.5,-3){1.5}{90}{180}
\psarc[linewidth=.2,linestyle=dotted](30.5,3){1.5}{270}{360}
\psarc[linewidth=.2,linestyle=dotted](30.5,-3){1.5}{0}{90}
\psline[linewidth=.2,linestyle=dotted](27.5,1.5)(30.5,1.5)
\psline[linewidth=.2,linestyle=dotted](27.5,-1.5)(30.5,-1.5)
\psarc[linewidth=.2,linestyle=dotted](25.5,3){.5}{0}{180}
\psarc[linewidth=.2,linestyle=dotted](25.5,-3){.5}{180}{360}
\psarc[linewidth=.2,linestyle=dotted](32.5,3){.5}{0}{180}
\psarc[linewidth=.2,linestyle=dotted](32.5,-3){.5}{180}{360}
\psline[linewidth=.2,linestyle=dotted](25,2.7)(25,-2.7)
\psline[linewidth=.2,linestyle=dotted](33,2.7)(33,-2.7)
\rput(29,0){$\Ga_0$}\rput(29,2.2){$2d'$}\rput(29,-2.2){$2d'$}
\rput(26.3,3.9){$i$}\rput(31,3.9){$\phi(i)$}
\rput(31.7,-3.9){$i$}\rput(27,-3.9){$\phi(i)$}
\psline[linewidth=.05]{<->}(25.7,2)(25.7,-2)
\psline[linewidth=.05]{<->}(32.3,2)(32.3,-2)
\pscircle*(38,3){.2}\pscircle*(44,3){.2}\pscircle*(41,6){.2}
\pscircle*(38,-3){.2}\pscircle*(44,-3){.2}\pscircle*(41,-6){.2}
\psline[linewidth=.07](38,3)(41,6)\psline[linewidth=.07](44,3)(41,6)
\psline[linewidth=.07](38,-3)(41,-6)\psline[linewidth=.07](44,-3)(41,-6)
\psarc[linewidth=.2,linestyle=dotted](39.5,3){1.5}{180}{270}
\psarc[linewidth=.2,linestyle=dotted](39.5,-3){1.5}{90}{180}
\psarc[linewidth=.2,linestyle=dotted](42.5,3){1.5}{270}{360}
\psarc[linewidth=.2,linestyle=dotted](42.5,-3){1.5}{0}{90}
\psline[linewidth=.2,linestyle=dotted](39.5,1.5)(42.5,1.5)
\psline[linewidth=.2,linestyle=dotted](39.5,-1.5)(42.5,-1.5)
\psarc[linewidth=.2,linestyle=dotted](37.5,3){.5}{0}{180}
\psarc[linewidth=.2,linestyle=dotted](37.5,-3){.5}{180}{360}
\psarc[linewidth=.2,linestyle=dotted](44.5,3){.5}{0}{180}
\psarc[linewidth=.2,linestyle=dotted](44.5,-3){.5}{180}{360}
\psline[linewidth=.2,linestyle=dotted](37,2.7)(37,-2.7)
\psline[linewidth=.2,linestyle=dotted](45,2.7)(45,-2.7)
\rput(41,0){$\Ga_0$}\rput(40,4){$d'$}\rput(41.8,4){$d'$}
\rput(40.2,-4){$d'$}\rput(42,-4){$d'$}
\rput(38.1,4.1){$i$}\rput(44.4,4.4){$\phi(i)$}\rput(41,6.7){$n$}
\rput(43.8,-4){$i$}\rput(37.6,-4.4){$\phi(i)$}\rput(41,-6.7){$n$}
\psline[linewidth=.05]{<->}(37.7,2)(37.7,-2)
\psline[linewidth=.05]{<->}(44.3,2)(44.3,-2)
\end{pspicture}
\caption{The second and fourth graphs are degenerations of the first and third,
respectively, induced by the degenerations of
$\bT^m$-fixed conics in $\P^{n-1}$ with $n$~odd;
$\Ga_0$ indicates any graph compatible with the indicated involution.}
\label{GaDegen_fig}
\end{figure}

\noindent
By the above,  
$$\ov\fM_{g,l}(\P^{n-1},d)^{\phi;\bT^m}=\bigsqcup_{(\Ga,\si)}\!\!\ov\cZ_{\Ga,\si}\,,$$
with the union taken over all pairs $(\Ga,\si)$ consisting of
an $S_l$-marked $[n]$-decorated connected graph~$\Ga$ with  an involution~$\si$
such that $g(\Ga)\!=\!g$, $d(\Ga)\!=\!d$, and $\Ga$ contains no vertex~$v$ such that 
$$\vt(v)=2m\!+\!1, \quad \val_v(\Ga)\!=\!\E_v(\Ga)\!=\!2,  \quad
\fd(e_1)=\fd(e_2), \quad \vt(e_1/v)=\phi\big(\vt(e_2/v)\big),$$
where $e_1,e_2\!\in\!\E_v(\Ga)$ are the two elements of~$\E_v(\Ga)$.
We denote the set of all such pairs $(\Ga,\si)$ by~$\ov\cA_{g,l}(n,d)$.\\

\noindent
The spaces~$\ov\cZ_{\Ga,\si}$ need not be connected or non-empty.
If $n\!\in\!2\Z$, $\phi\!=\!\tau_n'$, and $\fd(e)\!\in\!2\Z$ for some 
\hbox{$e\!\in\!\E_{\R}^{\si}(\Ga)$}, 
then $(\Ga,\si)$ is compatible with two distinct topological types of the real degree~$\fd(e)$
covers \hbox{$\Si_e\!\!\lra\!\P^1_{\vt(v_1)\vt(v_2)}$} branched only over two points.
By the proof of Lemma~\ref{VanV_lmm2}, the $\bT^m$-fixed loci associated with the two types of covers
contribute to~\eref{PDeVdfn_e2} with opposite signs.
If $k\!\in\!\Z^+$ and $\fd(e)\!\in\!2\Z$,  
then the restriction of~$\cV_{n;\a}^{\wt\phi_{n;\a}}$ to $\ov\cZ_{\Ga,\si}$
contains a subbundle of odd rank.
This is in particular the case 
if $n\!\not\in\!2\Z$ and thus $\phi\!=\!\tau_n'$.
Since the Euler class of such a subbundle vanishes, 
 the contribution of~$\ov\cZ_{\Ga,\si}$  to~\eref{PDeVdfn_e2}  is zero in this case as well.
If  $n\!\in\!2\Z$, $\phi\!=\!\eta_n$, and $\fd(e)\!\in\!2\Z$, 
then $(\Ga,\si)$ is not compatible with any real degree~$\fd(e)$ cover 
$\Si_e\!\!\lra\!\P^1_{\vt(v_1)\vt(v_2)}$; 
see \cite[Lemma~1.9]{Teh}. 
Thus, $\cZ_{\Ga,\si}$ is empty and does not contribute to~\eref{PDeVdfn_e2}
in this case either if $\fd(e)\!\in\!2\Z$ for some $e\!\in\!\E_{\R}^{\si}(\Ga)$.
This motivates the first restriction in~\eref{AdmissPairCond_e}.\\

\noindent
Suppose $\Ga$ is a decorated graph as in~\eref{Dual_e}, $\si$ is an involution
on~$\Ga$, and $[\u]$  is an element  of~$\ov\cZ_{\Ga,\si}$ as in~\eref{umapdfn_e}.
For each $e\!\in\!\E_v(\Ga)$, let $x_{e,v}\!\in\!\Si_e$ be the branch point of $u|_{\Si_e}$
sent to~$P_{\vt(v)}$.
Given $v\!\in\!\Ver$, let 
\begin{equation*}\begin{split}
\F_v(\Ga)&=\big\{(e,v)\!:\,e\!\in\!\E_v(\Ga)\big\}\cup
\big\{(\si(e),\si(v))\!:\,e\!\in\!\E_v(\Ga)\big\}\subset \Edg\!\times\!\Ver,\\
\nV_v'(\Ga)&=\big\{e/v\!:\,e\!\in\!\E_v(\Ga)\big\}\!\cup\!\big\{\si(e/v)\!:\,e\!\in\!\E_v(\Ga)\big\}\\
&\qquad-\big\{v,\si(v)\big\}-
\big\{e/v\!:\,e\!\in\!\E_v(\Ga)\big\}\!\cap\!\big\{\si(e/v)\!:\,e\!\in\!\E_v(\Ga)\big\},\\
\F_v'(\Ga)&=\big\{(e',v')\!:\,e'\!\in\!\E_v(\Ga)\!\cup\!\E_{\si(v)}(\Ga),
\,v'\!\in\!e'\!\cap\!\nV_v'(\Ga)\big\}.
\end{split}\end{equation*}
We define $\Si_v'\!\subset\!\Si$ by
\begin{gather*}
\Si_v'=\bigcup_{e\in\E_v(\Ga)}\!\!\!\!\!\big(\Si_e\!\cup\!\Si_{\si(e)}\big) \cup
\begin{cases}
\Si_v\!\cup\!\Si_{\si(v)},&\hbox{if}~\val_v(\Ga)\!\ge\!3;\\
\eset,&\hbox{if}~\val_v(\Ga)\!\le\!2.
\end{cases}\end{gather*}
Thus, $\Si_v'$ is a union of irreducible components of~$\Si_v$;  
the nodes shared by~$\Si_v'$  with other irreducible components of~$\Si$
are contained in the set $\{x_{e',v'}\}_{(e',v')\in\F_v'(\Ga)}$.
Let $\wt\Si_v'$ be the nodal surface obtained from~$\Si_v'$ by removing
the components in $\Si_v\!\cup\!\Si_{\si(v)}$ if $\val_v(\Ga)\!\ge\!3$ 
and separating $\Si_v'$ at the node or nodes corresponding to~$v$ and~$\si(v)$
if $\val_v(\Ga),|\E_v(\Ga)|\!=\!2$. 
Thus, every topological component of~$\wt\Si_v'$ is either $\P^1$ or a wedge of two copies of~$\P^1$.
The involution~$\si$ on~$\Si$ restricts to an involution on~$\Si_v'$ and induces
an involution on~$\wt\Si_v'$.\\

\noindent
Let $(\cL_{n;\a},\wt\phi_{n;\a})$ be as in~\eref{cLdfn_e}.
Thus,
$$\cV_{n;\a}^{\wt\phi_{n;\a}}\big|_{[\u]}=
H^0\big(\Si;u^*\cL_{n;\a}\big)^{\wt\phi_{n;\a}}\big/\Aut(\u)\,.$$
For each $v\!\in\!\Ver$, define
\begin{gather*}
H^0\big(\Si_v';u^*\cL_{n;\a}\big)_v^{\wt\phi_{n;\a}}
\equiv\big\{\xi\!\in\!H^0\big(\Si_v';u^*\cL_{n;\a}\big)^{\wt\phi_{n;\a}}\!:\,
\xi(x_{e',v'})\!=\!0~\forall\,(e',v')\!\in\!\F_v'(\Ga)\big\},\\
\begin{aligned}
\wt\phi_v\!: \bL_v\!\equiv\!\bigoplus_{(e',v')\in\F_v(\Ga)}
\!\!\!\!\!\!\!\!\!\cL_{n;\a}|_{P_{\vt(v')}}
&\lra \bL_v,
& \big(\wt\phi_v\big((w_{(e',v')})_{(e',v')\in\F_v(\Ga)}\big)\big)_{(e'',v'')}
&=\wt\phi_{n;\a}\big(w_{(\si(e''),\si(v''))}\big),\\
\wt\phi_v'\!: \bL_v'\!\equiv\!\bigoplus_{(e',v')\in\F_v'(\Ga)}
\!\!\!\!\!\!\!\!\!\cL_{n;\a}|_{P_{\vt(v')}}
&\lra \bL_v',
& \big(\wt\phi_v'\big((w_{(e',v')})_{(e',v')\in\F_v'(\Ga)}\big)
\big)_{(e'',v'')}&=\wt\phi_{n;\a}\big(w_{(\si(e''),\si(v''))}\big).
\end{aligned}\end{gather*}
The two involutions above interchange the components indexed by~$(e',v')$ and~$(\si(e'),\si(v'))$.
Furthermore,
\BE{Hovwt_e}\begin{split}
H^0\big(\Si_v';u^*\cL_{n;\a}\big)_v^{\wt\phi_{n;\a}}
=\big\{\xi\!\in\!H^0\big(\wt\Si_v';u^*\cL_{n;\a}\big)^{\wt\phi_{n;\a}}\!:
~&\xi(x_{e',v'})\!=\!0~\forall\,(e',v')\!\in\!\F_v'(\Ga),\\
&\xi(x_{e_1,v})\!=\!\xi(x_{e_2,v})~\forall\,e_1,e_2\!\in\!\E_v(\Ga)\big\}.
\end{split}\EE
Since the nodes shared by~$\Si_v'$  with the remainder of~$\Si$
are contained in the set $\{x_{e',v'}\}_{(e',v')\in F_v'(\Ga)}$,
the image of the restriction homomorphism
$$H^0\big(\Si;u^*\cL_{n;\a}\big)^{\wt\phi_{n;\a}}\lra H^0\big(\Si_v';u^*\cL_{n;\a}\big)^{\wt\phi_{n;\a}},
\qquad \xi\lra \xi|_{\Si_v'}\,,$$
contains the subspace $H^0(\Si_v';u^*\cL_{n;\a}\big)_v^{\wt\phi_{n;\a}}$.
Since each topological component of~$\wt\Si_v'$ is rational and contains at most two 
of the points $x_{e',v'}$ with $(e',v')$ in $\F_v(\Ga)\!\cup\!\F_v'(\Ga)$, 
the homomorphism
$$H^0\big(\wt\Si_v';u^*\cL_{n;\a}\big)^{\wt\phi_{n;\a}}\lra 
\bL_v^{\wt\phi_v}\!\oplus\!\bL_v'^{\wt\phi_v'}, \quad
\xi\lra \big(\big(\xi(x_{e',v'})\big)_{(e',v')\in\F_v(\Ga)},
\big(\xi(x_{e',v'})\big)_{(e',v')\in\F_v'(\Ga)}\big),$$
is surjective.\\

\noindent
If $v\!\in\!\Ver$ and $\val_v(\Ga)\!\ge\!3$, $\xi|_{\Si_v}$ is constant for all
$\xi\!\in\!H^0(\Si;u^*\cL_{n;\a})$.
Thus, the evaluation homomorphism
$$\wt\ev_v\!: H^0\big(\Si;u^*\cL_{n;\a}\big)\lra\cL_{n;\a}\big|_{P_{\vt(v)}}, \qquad
\wt\ev_v(\xi)=\xi(\Si_v),$$
is well defined.
If $\val_v(\Ga),|\E_v(\Ga)|\!=\!2$, we define it to be the evaluation at the node
of~$\Si$ corresponding to~$v$.
If  $\val_v(\Ga)\!\le\!2$ and $\E_v(\Ga)\!=\!\{e\}$, we take this homomorphism to be
the evaluation at the preimage $x_{e,v}\!\in\!\Si_e$ of~$P_{\vt(v)}$.
If $v\!\in\!\nV_{\R}^{\si}(\Ga)$, the previous paragraph implies that 
the homomorphism
\BE{HovevalR_e} H^0\big(\Si;u^*\cL_{n;\a}\big)^{\wt\phi_{n;\a}}\lra 
\big(\cL_{n;\a}|_{P_{\vt(v)}}\big)^{\wt\phi_{n;\a}}, \qquad \xi\lra \wt\ev_v(\xi),\EE
is surjective.
If $v\!\in\!\nV_{\C}^{\si}(\Ga)$, it implies that  the homomorphism
\BE{HovevalC_e} 
H^0\big(\Si;u^*\cL_{n;\a}\big)^{\wt\phi_{n;\a}}\lra 
\big(\cL_{n;\a}\big|_{P_{\vt(v)}}\!\oplus\!\cL_{n;\a}|_{P_{\phi(\vt(v)})}\big)^{\wt\phi_{n;\a}}, 
\quad  \xi\lra \big(\wt\ev_v(\xi),\wt\ev_{\si(v)}(\xi)\big),\EE
is surjective.
From this, we obtain the following observation.

\begin{lmm}\label{VanV_lmm}
Suppose  $n\!=\!2m\!+\!1$ and $\Ga\!\in\!\ov\cA_{g,l}(n,d)$.
If $\vt(v)\!=\!n$ for some $v\!\in\!\Ver$, then
$$\e(\cV_{n;\a}^{\wt\phi_{n;\a}})|_{\ov\cZ_{\Ga,\si}}=0.$$
\end{lmm}

\begin{proof}
Since $n\!\not\in\!2\Z$, $k\!>\!0$ and so the targets in~\eref{HovevalR_e}
and~\eref{HovevalC_e} are non-trivial.
If $v\!\in\!\nV_{\R}^{\si}(\Ga)$, the surjectivity of~\eref{HovevalR_e} implies that 
$\cV_{n;\a}^{\wt\phi_{n;\a}}|_{\cZ_{\Ga,\si}}$ contains a trivial real line bundle
with the trivial $\bT^m$-action.
If \hbox{$v\!\in\!\nV_{\C}^{\si}(\Ga)$}, the surjectivity of~\eref{HovevalC_e} implies that 
$\cV_{n;\a}^{\wt\phi_{n;\a}}|_{\cZ_{\Ga,\si}}$ contains a trivial complex line bundle
with the trivial $\bT^m$-action.
In either case, $\e(\cV_{n;\a}^{\wt\phi_{n;\a}})|_{\cZ_{\Ga,\si}}$ vanishes.
\end{proof}

\noindent
By Lemma~\ref{VanV_lmm}, 
$\ov\cZ_{\Ga,\si}$ does not contribute to~\eref{EquivLocal_e}
unless  $\vt(v)\!\neq\!2m\!+\!1$ for all vertices~$v\!\in\!\Ver$.
Thus,  it is sufficient to restrict attention to the subset
$$\ov\cA_{g,l}'(n,d)\subset\ov\cA_{g,l}(n,d)$$ 
of pairs $(\Ga,\si)$ satisfying the second condition in~\eref{AdmissPairCond_e}.

\subsection{The fixed-locus contribution}
\label{FLCsign_subs}

\noindent
For each pair $(\Ga,\si)$ in $\ov\cA_{g,l}'(n,d)$, we will next describe
the moduli spaces associated with the vertices and edges of~$\Ga$ 
and then determine the normal bundle to 
the $\bT^m$-fixed locus $\ov\cZ_{\Ga,\si}$,
after capping with  $\e(\cV_{n;\a}^{\wt\phi_{n;\a}})$ if $k\!\in\!\Z^+$. 
We fix an element of $\ov\cA_{g,l}'(n,d)$  with $\Ga$ as in~\eref{Dual_e}
throughout this section.\\

\noindent
Let $g_0\!\in\!\Z^{\ge0}$.
We will call a two-component symmetric surface $(\Si,\si)$ of the~form
\BE{SymSurfDbl_e}\Si\equiv \Si_1\!\sqcup\!\Si_2
\equiv \{1\}\!\times\!\Si_0 \sqcup \{2\}\!\times\!\ov\Si_0, \qquad
\si(i,z)=\big(3\!-\!i,z\big)~~\forall~(i,z)\!\in\!\Si,\EE
where $\Si_0$ is a connected oriented, possibly nodal, genus~$g_0$ surface 
and  $\ov\Si_0$ denotes $\Si_0$ with the opposite orientation,
a \sf{nodal $g_0$-doublet}.
The arithmetic genus of a $g_0$-doublet is $2g_0\!-\!1$.
If $d_0\!\in\!\Z$ and $S_1$ and $S_2$ are finite sets with a fixed bijection~$\si_S$ between them,
let $\ov\fM_{2g_0-1,S_1\sqcup S_2}^{\bu}(\P^{n-1},2d_0)^{\phi}$
denote the moduli space of stable real  degree~$2d_0$ $J_0$-holomorphic maps into~$\P^{n-1}$ from 
nodal $g_0$-doublets with the first component carrying the $S_1$-marked points 
and with the marked points interchanged by the involution~$\si_S$.
We denote by $\ov\fM_{g_0,S_1}(\P^{n-1},d_0)$
the usual moduli space of stable genus~$g_0$ $S_1$-marked degree~$d_0$ maps into~$\P^{n-1}$.\\

\noindent
For each vertex~$v$ of~$\Ga$, let $S_v$ be as in~\eref{Svdfn_e}.
Since $v\!\neq\!\si(v)$, $\si$ induces an involution on the set $S_v\!\sqcup\!S_{\si(v)}$.
If $\val_v(\Ga)\!\ge\!3$, let 
\BE{fMbu_e0}
\ov\fM_{\Ga;v}=\ov\fM_{\g(v),S_v}(\P^{n-1},0) 
\qquad\hbox{and}\qquad
\ov\fM_{\Ga;v}^{\phi}=\ov\fM_{2\g(v)-1,S_v\sqcup S_{\si(v)}}^{\bu}(\P^{n-1},0)^{\phi}\,.\EE
We denote by
\BE{evalv_e}
\ev_v\!:\ov\fM_{\Ga;v}\lra \P^{n-1} \qquad\hbox{and}\qquad
\ev_v^{\phi}\!:\ov\fM_{\Ga;v}^{\phi}\lra \P^{n-1}\EE
the morphism sending each constant stable map to its value
and the morphism  sending each degree~0 holomorphic map from a doublet to its value on
the first component, respectively.
For each~$e\!\in\!S_v$, let
\BE{Lvdfn_e}L_{v;e}\lra \ov\fM_{\Ga;v} \qquad\hbox{and}\qquad
L_{v;e}^{\phi}\lra \ov\fM_{\Ga;v}^{\phi}\EE
be the universal tangent line bundles for this point.\\

\noindent
The restriction of the map to the $S_v$-marked component induces a diffeomorphism
\BE{RtoCv_e} \Psi_{\Ga;v}\!:\ov\fM_{\Ga;v}^{\phi}\lra \ov\fM_{\Ga;v}\EE
between the two moduli spaces which commutes with the evaluation morphisms~\eref{evalv_e}
and naturally lifts to an isomorphism between the line bundles~\eref{Lvdfn_e}.
Let 
$$\ov\fM_{\Ga;v}^{\phi;\bT}=\Psi_{\Ga;v}^{-1}\big(\ov\fM_{\Ga;v}^{\bT}\big)
\subset \ov\fM_{\Ga;v}^{\phi}\,;$$
see~\eref{ovfMGavdfn_e}.
We denote by 
$$\cN_{\Ga;v}\lra \ov\fM_{\Ga;v}^{\bT} \qquad\hbox{and}\qquad 
\cN_{\Ga;v}^{\phi}\lra \ov\fM_{\Ga;v}^{\phi;\bT} $$
the normal bundle of $\ov\fM_{\Ga;v}^{\bT}$ in $\ov\fM_{\Ga;v}$
and of $\ov\fM_{\Ga;v}^{\phi;\bT}$ in $\ov\fM_{\Ga;v}^{\phi}$, respectively.
The diffeomorphism~\eref{RtoCv_e} descends to an isomorphism from the second bundle
to the~first.\\

\noindent
For each $e\!\in\!\Edg$ and $v\!\in\!e$, let
$$S_{e,v}=\begin{cases} \{v\},&\hbox{if}~\val_v(\Ga)\!+\!|\E_v(\Ga)|\!\ge\!4;\\
\fm^{-1}(v),&\hbox{if}~\val_v(\Ga)\!+\!|\E_v(\Ga)|\!=\!3;\\
\eset,&\hbox{if}~\val_v(\Ga)\!+\!|\E_v(\Ga)|\!=\!2.
\end{cases}$$
If $e\!=\!\{v_1,v_2\}$, let $S_e\!=\!S_{e,v_1}\!\sqcup\!S_{e,v_2}$.
If $e\!\in\!\E_{\R}^{\si}(\Ga)$, $S_e$ is either empty or consists of two elements interchanged 
by~$\si$.
In such a case, let 
$$ \ov\fM_{\Ga,\si;e}^{\phi}\equiv \ov\fM_{0,S_e}\big(\P^{n-1},\fd(e)\big)^{\phi}$$
denote the moduli space of stable real genus~0 $S_e$-marked degree~$\fd(e)$ maps into~$\P^{n-1}$.
Let 
$$\ov\fM_{\Ga,\si;e}^{\phi;\bT}\subset  \ov\fM_{\Ga,\si;e}^{\phi}$$
be the preimage~of  $\ov\fM_{\Ga,\si;e}^{\phi;\bT;\circ}$ under the forgetful morphism
$$\ff_{\Ga,\si;e}^{\phi}\!: \ov\fM_{\Ga,\si;e}^{\phi}
\lra \ov\fM_{\Ga,\si;e}^{\phi;\circ} \equiv \ov\fM_{0,0}\big(\P^{n-1},\fd(e)\big)^{\phi};$$
see~\eref{EdgeSpaces_e}.
We denote~by 
$$\cN_{\Ga,\si;e}^{\phi}\lra \ov\fM_{\Ga,\si;e}^{\phi;\bT} $$
the normal bundle of $\ov\fM_{\Ga,\si;e}^{\phi;\bT}$ in $\ov\fM_{\Ga,\si;e}^{\phi}$.\\

\noindent
If $e\!\in\!\E_{\C}^{\si}(\Ga)$, $S_e$ and $S_{\si(e)}$ consist of either one or two elements each;
the involution~$\si$ interchanges the two sets.
Let
$$\ov\fM_{\Ga;e}= \ov\fM_{0,S_e}\big(\P^{n-1},\fd(e)\big) 
\qquad\hbox{and}\qquad
\ov\fM_{\Ga,\si;e}^{\phi}=\ov\fM_{-1,S_e\sqcup S_{\si(e)}}^{\bu}(\P^{n-1},2\fd(e))^{\phi}\,.$$
Thus, $\ov\fM_{\Ga,\si;e}^{\phi}\!=\!\ov\fM_{\Ga,\si;\si(e)}^{\phi}$ and 
the restriction of the map to the $S_e$-marked component induces a diffeomorphism
\BE{RtoCe_e} \Psi_{\Ga,\si;e}\!: \ov\fM_{\Ga,\si;e}^{\phi}\lra \ov\fM_{\Ga;e} \EE
between the two moduli spaces.
Let 
$$\ov\fM_{\Ga;e}^{\bT}\subset \ov\fM_{\Ga;e}$$
be the preimage~of  $\ov\fM_{\Ga;e}^{\bT;\circ}$ under the forgetful morphism
\BE{ffGaE_e}\ff_{\Ga;e}\!: \ov\fM_{\Ga;e}\lra 
\ov\fM_{\Ga;e}^{\circ}\!\equiv\!\ov\fM_{0,0}\big(\P^{n-1},\fd(e)\big)^{\phi}\EE
and define
$$\ov\fM_{\Ga,\si;e}^{\phi;\bT}=\Psi_{\Ga;v}^{-1}\big(\ov\fM_{\Ga;e}^{\bT}\big)
\subset \ov\fM_{\Ga,\si;e}^{\phi}\,.$$
We denote by 
$$\cN_{\Ga;e}\lra \ov\fM_{\Ga;e}^{\bT} \qquad\hbox{and}\qquad 
\cN_{\Ga,\si;e}^{\phi}\lra \ov\fM_{\Ga,\si;e}^{\phi;\bT} $$
the normal bundle of $\ov\fM_{\Ga;e}^{\bT}$ in $\ov\fM_{\Ga;e}$
and of $\ov\fM_{\Ga,\si;e}^{\phi;\bT}$ in $\ov\fM_{\Ga,\si;e}^{\phi}$, respectively.
The diffeomorphism~\eref{RtoCe_e} descends to an isomorphism from the second bundle
to the~first.\\

\noindent
In either of the two cases above, for $v\!\in\!S_e\!\cup\!S_{\si(e)}$ let 
$$\ev_{e,v}^{\phi}\!:\ov\fM_{\Ga,\si;e}^{\phi}\lra \P^{n-1} \qquad\hbox{and}\qquad
L_{e;v}^{\phi}\lra \ov\fM_{\Ga,\si;e}^{\phi}$$
be the evaluation morphism and the universal tangent line bundle for the marked point 
indexed by~$v$.
If $e\!\in\!\E_{\C}^{\si}(\Ga)$ and $v\!\in\!S_e$, let
$$\ev_{e,v}\!:\ov\fM_{\Ga;e}\lra \P^{n-1} \qquad\hbox{and}\qquad
L_{e;v}\lra \ov\fM_{\Ga;e}$$
be the analogous objects for the target space of the diffeomorphism~\eref{RtoCe_e}.
We extend these definitions to $v\!\in\!S_{\si(e)}$ by setting
$$\ev_{e,v}=\phi\circ\ev_{e,\si(v)} \qquad\hbox{and}\qquad
L_{e;v}=\ov{L_{e;\si(v)}}\,.$$
The diffeomorphism~\eref{RtoCe_e} commutes with the evaluation morphisms
and naturally lifts to an isomorphism between the universal tangent line bundles.\\

\noindent
Choose $\nV_+^{\si}(\Ga)\!\subset\!\Ver$ and $\E_+^{\si}(\Ga)\!\subset\!\Edg$
satisfying~\eref{nVplus_e} and~\eref{nEplus_e}.
Let 
$$\nV_{+;3}^{\si}(\Ga)=\big\{v\!\in\!\nV_+^{\si}(\Ga)\!:\,\val_v(\Ga)\!\ge\!3\big\},
\quad
\nV_{+;2}^{\si}(\Ga)=\big\{v\!\in\!\nV_+^{\si}(\Ga)\!:\,
\val_v(\Ga),|\E_v(\Ga)|\!=\!2\big\}.$$
We define
\BE{FixedLocus_e}\begin{split}
\wt\cZ_{\Ga,\si}&\equiv
\prod_{v\in\nV_{+;3}^{\si}(\Ga)}\!\!\!\!\!\!\!\ov\fM_{\Ga;v}^{\phi;\bT} \times
\prod_{e\in\E_{\R}^{\si}(\Ga)\sqcup\E_+^{\si}(\Ga)}\hspace{-.34in}\ov\fM_{\Ga,\si;e}^{\phi;\bT}\\
&\approx \prod_{v\in\nV_{+;3}^{\si}(\Ga)}\!\!\!\!\!\!\!\ov\cM_{\g(v);S_v} \times
\prod_{e\in\E_{\R}^{\si}(\Ga)}\!\!\!\!\!\ov\fM_{\Ga,\si;e}^{\phi;\bT} \times
\prod_{e\in\E_+^{\si}(\Ga)}\!\!\!\!\!\ov\fM_{\Ga;e}^{\bT}\,.
\end{split}\EE
The fixed locus $\ov\cZ_{\Ga,\si}$ corresponding to   $(\Ga,\si)$
is then given~by
$$\ov\cZ_{\Ga,\si}=\wt\cZ_{\Ga,\si}\big/\Aut(\Ga,\si)$$
with the group $\Aut(\Ga,\si)$ acting trivially.
For example, in the case of the pair~$(\Ga,\si)$ represented by the first diagram
in Figure~\ref{discord_fig}
$$\cZ_{\Ga,\si}=\ov\cZ_{\Ga,\si}\approx \big(\ov\cM_{0,3}\!\times\!\ov\cM_{0,3}\big)\big/S_3
\approx \{\pt\}/S_3\subset \ov\fM_{2,0}(\P^{n-1},3)^{\phi}\,,$$
with the symmetric group~$S_3$ acting trivially.\\

\noindent
For each $e\!\in\!\Edg$, define 
$$e_{\bu}\in \E_{\R}^{\si}(\Ga)\!\cup\!\E_+^{\si}(\Ga) \qquad\hbox{by}\qquad
e_{\bu}\in\big\{e,\si(e)\big\}.$$
Let
\begin{gather*}
\ov\fM_{\Ga}^{\phi}=
\prod_{v\in\nV_{+;3}^{\si}(\Ga)}\!\!\!\!\!\!\!\ov\fM_{\Ga;v}^{\phi} \times
\prod_{e\in\E_{\R}^{\si}(\Ga)\cup \E_+^{\si}(\Ga)}\hspace{-.31in}\ov\fM_{\Ga,\si;e}^{\phi}\,,\\
\begin{split}
\ov\fM_{\Ga}'^{\phi}&=\big\{\big((\u_v)_{v\in\nV_{+;3}^{\si}(\Ga)},
(\u_e)_{e\in\E_{\R}^{\si}(\Ga)\sqcup\E_+^{\si}(\Ga)}\big)\!\in\!\ov\fM_{\Ga}^{\phi}\!:~
\ev_v(\u_v)\!=\!\ev_{e_{\bu},v}(\u_{e_{\bu}})~
\forall\,v\!\in\!\nV_{+;3}^{\si}(\Ga),\, e\!\in\!\E_v(\Ga),\\
&\hspace{2.3in} \ev_{e_{1\bu},v}(\u_{e_{1\bu}})\!=\!\ev_{e_{2\bu},v}(\u_{e_{2\bu}})~
\forall\,v\!\in\!\nV_{+;2}^{\si}(\Ga),\, e_1,e_2\!\in\!\E_v(\Ga)\big\}.
\end{split}\end{gather*}
For each $v\!\in\!\nV_{+;3}^{\si}(\Ga)\!\cup\!\nV_{+;2}^{\si}(\Ga)$, there is thus a well-defined
evaluation morphism
$$\ev_v\!:\ov\fM_{\Ga}'^{\phi} \lra \P^{n-1}, \quad
\ev_v\big((\u_{v'})_{v'\in\nV_{+;3}^{\si}(\Ga)},
(\u_e)_{e\in\E_{\R}^{\si}(\Ga)\sqcup\E_+^{\si}(\Ga)}\big)=
\ev_{e_{\bu},v}(\u_{e_{\bu}})~{}~\hbox{if}\,e\!\in\!\E_v(\Ga).$$
For $v\!\in\!\nV_{+;3}^{\si}(\Ga)$ and $e\!\in\!\E_{\R}^{\si}(\Ga)\!\sqcup\!\E_+^{\si}(\Ga)$, let
$$\pi_v\!:\ov\fM_{\Ga}'^{\phi}\lra \ov\fM_{\Ga;v}^{\phi}
\qquad\hbox{and}\qquad
\pi_e\!:\ov\fM_{\Ga}'^{\phi}\lra \ov\fM_{\Ga,\si;e}^{\phi}$$
be the projection maps.\\

\noindent
Define
\begin{alignat}{1}
\label{LGadfn_e}
L_{\Ga}&=\bigoplus_{v\in\nV_{+;3}^{\si}(\Ga)}\bigoplus_{e\in\E_v(\Ga)}\!\!\!\!
\pi_v^*L_{v;e}^{\phi}\!\otimes\!\pi_{e_{\bu}}^*L_{e_{\bu};v}^{\phi} \oplus
\bigoplus_{v\in\nV_{+;2}^{\si}(\Ga)}
\bigotimes_{e\in\E_v(\Ga)}\!\!\!\!\!\pi_{e_{\bu}}^*L_{e_{\bu};v}^{\phi}\,,\\
\label{cNGadfn_e}
\cN_{\Ga}\P&=
\bigoplus_{v\in\nV_{+;3}^{\si}(\Ga)}\!\!\!\!\!\!\!
\big(\ev_v^*T\P^{n-1}\big)^{\!\{v\}\sqcup\E_v(\Ga)}
\oplus
\bigoplus_{v\in\nV_{+;2}^{\si}(\Ga)}\!\!\!\!\!\!\!
\big(\ev_v^*T\P^{n-1}\big)^{\!\E_v(\Ga)}\,,\\
\label{cLGadfn_e}
\cL_{\Ga}&=
\bigoplus_{v\in\nV_{+;3}^{\si}(\Ga)}\!\!\!\!\!\!\!
\big(\ev_v^*\cL_{n;\a}\big)^{\!\{v\}\sqcup\E_v(\Ga)}
\oplus
\bigoplus_{v\in\nV_{+;2}^{\si}(\Ga)}\!\!\!\!\!\!\!
\big(\ev_v^*\cL_{n;\a}\big)^{\!\E_v(\Ga)}\,.
\end{alignat}
For $v\!\in\!\nV_{+;3}^{\si}(\Ga)\!\cup\!\nV_{+;2}^{\si}(\Ga)$, denote~by 
\begin{equation*}\begin{split}
\cN_{\Ga;v}^{\De}\P&\subset\begin{cases}
(\ev_v^*T\P^{n-1})^{\{v\}\sqcup\E_v(\Ga)},&\hbox{if}~v\!\in\!\nV_{+;3}^{\si}(\Ga);\\
(\ev_v^*T\P^{n-1})^{\E_v(\Ga)},&\hbox{if}~v\!\in\!\nV_{+;2}^{\si}(\Ga);
\end{cases}
\qquad\hbox{and}\\
\cL_{\Ga;v}^{\De}&\subset\begin{cases}
(\cL_{n;\a})^{\{v\}\sqcup\E_v(\Ga)},&\hbox{if}~v\!\in\!\nV_{+;3}^{\si}(\Ga);\\
(\cL_{n;\a})^{\E_v(\Ga)},&\hbox{if}~v\!\in\!\nV_{+;2}^{\si}(\Ga);
\end{cases}
\end{split}\end{equation*}
the small diagonals (all components are the same).
Let $\cN_{\Ga}'\P$ and $\cL_{\Ga}'$ be the quotients of $\cN_{\Ga}\P$ and $\cL_{\Ga}$
by the subbundles
$$\cN_{\Ga}^{\De}\P=
\bigoplus_{v\in\nV_{+;3}^{\si}(\Ga)\cup\nV_{+;2}^{\si}(\Ga)}\hspace{-.43in}\cN_{\Ga;v}^{\De}\P
\qquad\hbox{and}\qquad
\cL_{\Ga}^{\De}=
\bigoplus_{v\in\nV_{+;3}^{\si}(\Ga)\cup\nV_{+;2}^{\si}(\Ga)}\hspace{-.43in}\cL_{\Ga;v}^{\De},$$
respectively.
Since the vector bundles $L_{\Ga}$, $\cN_{\Ga}'\P$, and $\cL_{\Ga}'$  are complex, 
they are canonically oriented.\\

\noindent
The differentials of the evaluation morphisms $\ev_v$ and $\ev_{e_{\bu},v}$ 
induce a homomorphism
$$\wt\ev_{\Ga}^{\P}\!:\big(T\ov\fM_{\Ga}^{\phi}\big)^{\vrt}|_{\ov\fM_{\Ga}'^{\phi}}
\lra \cN_{\Ga}'\P\,.$$
The latter descends to an isomorphism
\BE{cNGaisom_e}
\wt\ev_{\Ga}^{\P}\!:
\cN_{\Ga}^{\phi}\fM\!\equiv\! 
\frac{(T\ov\fM_{\Ga}^{\phi})^{\vrt}|_{\ov\fM_{\Ga}'^{\phi}}}{(T\ov\fM_{\Ga}'^{\phi})^{\vrt}}
\stackrel{\approx}{\lra} \cN_{\Ga}'\P\EE
from the normal bundle of $\ov\fM_{\Ga}'^{\phi}$ in $\ov\fM_{\Ga}^{\phi}$.
The natural bundle homomorphisms
$$\cV_{n;\a}^{\wt\phi_{n;\a}}\lra \ev_v^*\cL_{n;\a} \qquad\hbox{and}\qquad
\cV_{n;\a}^{\wt\phi_{n;\a}}\lra \ev_{e_{\bu},v}^*\cL_{n;\a}$$
over  $\ov\fM_{\Ga;v}^{\phi}$ and $\ov\fM_{\Ga,\si;e}^{\phi}$,
respectively, given by the evaluations at the marked points similarly induce a bundle homomorphism
$$\wt\ev_{\Ga}^{\cL}\!:
\cV_{\Ga,\si}\equiv
\bigoplus_{v\in\nV_{+;3}^{\si}(\Ga)}\!\!\!
\bigg(\!\!\pi_v^*\cV_{n;\a}^{\wt\phi_{n;\a}}\oplus \!\!\!
\bigoplus_{e\in\E_v(\Ga)}\!\!\!\!\pi_{e_{\bu}}^*\cV_{n;\a}^{\wt\phi_{n;\a}}\!\!\bigg)
\oplus  \bigoplus_{v\in\nV_{+;2}^{\si}(\Ga)}
\bigoplus_{e\in\E_v(\Ga)}\!\!\!\!\pi_{e_{\bu}}^*\cV_{n;\a}^{\wt\phi_{n;\a}}
\lra  \cL_{\Ga}'$$
over $\ov\fM_{\Ga}'^{\phi}$.\\

\noindent
Denote by
$$\io_{\Ga}\!: \ov\fM_{\Ga}'^{\phi}\lra \ov\fM_{g,l}\big(\P^{n-1},d\big)^{\phi}$$
the natural node-identifying immersion which sends $\wt\cZ_{\Ga,\si}$ to~$\ov\cZ_{\Ga,\si}$.
For each $v\!\in\!\nV_{+;3}^{\si}(\Ga)$ and \hbox{$e\!\in\!\E_v(\Ga)$}, it identifies
the marked point of~$\u_v$ indexed by~$e$ with the marked point of
$\u_{e_{\bu}}$ indexed by~$v$.
For each $v\!\in\!\nV_{+;2}^{\si}(\Ga)$ and $\E_v(\Ga)\!=\!\{e_1,e_2\}$, 
$\io_{\Ga}$ identifies the marked points of~$\u_{e_{1\bu}}$ and~$\u_{e_{2\bu}}$ 
indexed by~$v$.
There is a natural isomorphism
\BE{cNLisom_e} \cN\io_{\Ga}\equiv 
\frac{\io_{\Ga}^*(T\ov\fM_{g,l}\big(\P^{n-1},d)^{\phi})^{\vrt}}{(T\ov\fM_{\Ga}'^{\phi})^{\vrt}}
\approx L_{\Ga}\EE 
of vector bundles over $\ov\fM_{\Ga}'^{\phi}$.
Since the exact sequence
$$0\lra  \io_{\Ga}^*\cV_{n;\a}^{\wt\phi_{n;\a}} \lra \cV_{\Ga,\si} 
 \stackrel{\wt\ev_{\Ga}^{\cL}}{\lra}
\cL_{\Ga}'\lra0$$
of vector bundles over $\ov\fM_{\Ga}'^{\phi}$ is $\bT^m$-equivariant,
\BE{cVsplit_e}
\io_{\Ga}^*\e\big(\cV_{n;\a}^{\wt\phi_{n;\a}}\big) =\frac{1}{\e(\cL_{\Ga}')}
\e\big(\cV_{\Ga,\si}\big)\,.\EE
\\

\noindent
For each $v\!\in\!\Ver$, the elements of the subset $\fm^{-1}(v)\!\subset\!S_v$
carry signs as elements of~$S_l$.
We decorate the elements in the complement 
$\E_v(\Ga)\!\subset\!S_v$ of~$\fm^{-1}(v)$ 
with the plus sign if $v\!\in\!\nV_+^{\si}(\Ga)$
and with the minus sign if $\si(v)\!\in\!\nV_+^{\si}(\Ga)$.
The involution~$\si$ on $S_v\!\sqcup\!S_{\si(v)}$ then interchanges the two subsets
and changes the sign of each element.
If $e\!\in\!\Edg$ and $v\!\in\!S_e\!\cap\!\Ver$, 
we decorate $v\!\in\!S_e$ with   the plus sign if $v\!\in\!\nV_+^{\si}(\Ga)$
and with the minus sign if $\si(v)\!\in\!\nV_+^{\si}(\Ga)$.
If $e\!\in\!\E_{\R}^{\si}(\Ga)$, $\si$ induces an involution on~$S_e$ that changes the sign of
each element.
If $e\!\in\!\E_{\C}^{\si}(\Ga)$, the involution~$\si$ on $S_e\!\sqcup\!S_{\si(e)}$ 
then interchanges the two subsets and changes the sign of each element.
Let $S_e^-\!\subset\!S_e$ be the subset of negatively marked elements.
In particular,
$$\big|S_e^-\big|=1 \quad\forall\,e\!\in\!\E_{\R}^{\si}(\Ga), \qquad
\big|S_e^-\big|\in\{0,1,2\} \quad\forall\,e\!\in\!\E_{\C}^{\si}(\Ga);$$
all three possibilities in the last case are  in general attainable.\\

\noindent
Denote by 
$$\wt\cN_{\Ga,\si}^{\phi}\equiv
\bigoplus_{v\in\nV_{+;3}^{\si}(\Ga)}\!\!\!
\bigg(\!\!\pi_v^*\cN_{\Ga;v}^{\phi}\oplus
\bigoplus_{e\in\E_v(\Ga)}\!\!\!\!\pi_{e_{\bu}}^*\cN_{\Ga,\si;e_{\bu}}^{\phi}\!\!\bigg)
\oplus  \!\!\!
\bigoplus_{v\in\nV_{+;2}^{\si}(\Ga)}
\bigoplus_{e\in\E_v(\Ga)}\!\!\!\!\pi_{e_{\bu}}^*\cN_{\Ga,\si;e_{\bu}}^{\phi}
\lra \wt\cZ_{\Ga,\si}$$
the normal bundle of $\wt\cZ_{\Ga,\si}$ in $\ov\fM_{\Ga}^{\phi}$.
Under the assumptions on~$(\Ga,\si)$ in Theorem~\ref{EquivLocal_thm},
$$\wt\cZ_{\Ga,\si}\subset \fM_{g,l}^{\star}(\P^{n-1},d)^{\phi}\,;$$
see \eref{fM0dfn_e} for the notation.
The real orientation on $(X_{n;\a},\phi_{n;\a})$ described 
in Section~\ref{CROprop_subs} determines an orientation on 
the moduli space $\ov\fM_{g,l}(X_{n;\a},d)^{\phi_{n;\a}}$.
Analogously to \cite[Section~5.3]{PSW}, it also determines orientations~on
the homology classes
\BE{cVcZclasses_e} \frac{\e(\cV_{n;\a}^{\wt\phi_{n;\a}})}{\e(\cN_{\Ga,\si}^{\phi})}
\cap\big[\ov\cZ_{\Ga,\si}\big]^{\vrt} 
\qquad\hbox{and}\qquad
\frac{\e(\cV_{\Ga,\si})}{\e(\wt\cN_{\Ga,\si}^{\phi})}
\cap\big[\wt\cZ_{\Ga,\si}\big]^{\vrt} \EE
that correspond to the orientations of the two moduli spaces of real maps
into $(X_{n;\a},\om_{n;\a},\phi_{n;\a})$ via the section~$s_{n;\a}$.
By~\eref{cNGaisom_e}, \eref{cNLisom_e}, and~\eref{cVsplit_e},
\BE{cVcZclasses_e2} \io_{\Ga}^*\bigg(\frac{\e(\cV_{n;\a}^{\wt\phi_{n;\a}})}{\e(\cN_{\Ga,\si}^{\phi})}\bigg)
=\bigg(\frac{\e(\cN_{\Ga}'\P)}{\e(\cL_{\Ga}')}\bigg)
\frac{1}{\e(L_{\Ga})}
\bigg(\frac{\e(\cV_{\Ga,\si})}{\e(\wt\cN_{\Ga,\si}^{\phi})}\bigg)\,.\EE
\\

\noindent
Since the orientations on the two classes in~\eref{cVcZclasses_e} arise from
the same real orientation on $(X_{n;\a},\phi_{n;\a})$,
\cite[Theorem~1.2]{RealGWsII} provides for a comparison between 
the two classes in~\eref{cVcZclasses_e}.
Since the section~$s_{n;\a}$ is holomorphic, 
the leading fraction on the right-hand side of~\eref{cVcZclasses_e2} 
corresponds to the Euler class of $TX_{n;\a}$ 
at the chosen node in each pair of the conjugate nodes considered 
for the purposes of \cite[Theorem~1.2]{RealGWsII}.
Since the number of such pairs~is
\BE{ndGadfn_e}\nd(\Ga)\equiv \sum_{v\in\nV_{+;3}^{\si}(\Ga)}\!\!\!\!\!\!\!\big|\E_v(\Ga)\big|
+\big|\nV_{+;2}^{\si}(\Ga)\big|\,,\EE
we find that 
\BE{cVcZclasses_e3}
\frac{\e(\cV_{n;\a}^{\wt\phi_{n;\a}})}{\e(\cN_{\Ga,\si}^{\phi})}\cap\big[\ov\cZ_{\Ga,\si}\big]^{\vrt} 
=\frac{(-1)^{\nd(\Ga)}}{|\Aut(\Ga,\si)|}  
\bigg(\frac{\e(\cN_{\Ga}'\P)}{\e(\cL_{\Ga}')}\bigg)
\io_{\Ga*}\bigg(\frac{1}{\e(L_{\Ga})}
\bigg(\frac{\e(\cV_{\Ga,\si})}{\e(\wt\cN_{\Ga,\si}^{\phi})}\bigg)
\cap\big[\wt\cZ_{\Ga,\si}\big]^{\vrt}\bigg).\EE
The automorphism factor above is the degree of the covering
$$\io_{\Ga}\!: \wt\cZ_{\Ga,\si} \lra\ov\cZ_{\Ga,\si}\,;$$
it corresponds to the first factor in~\eref{EquivLocal_e}.
By definition of $\wt\cN_{\Ga,\si}^{\phi}$, $\cV_{\Ga,\si}$, and $\wt\cZ_{\Ga,\si}$,
\BE{cVcZclasses_e5}\begin{split}
\bigg(\frac{\e(\cV_{\Ga,\si})}{\e(\wt\cN_{\Ga,\si}^{\phi})}\bigg)
\cap\big[\wt\cZ_{\Ga,\si}\big]^{\vrt}
&=\prod_{v\in\nV_{+;3}^{\si}(\Ga)}\!\!\!
\bigg(\!\!\bigg(\frac{\e(\cV_{n;\a}^{\wt\phi_{n;\a}})}{\e(\cN_{\Ga;v}^{\phi})}\bigg)
\cap\big[\ov\fM_{\Ga;v}^{\phi;\bT}\big]^{\vrt} \bigg)\\
&\qquad\times
\prod_{e\in\E_{\R}^{\si}(\Ga)\sqcup\E_+^{\si}(\Ga)}\!\!\!
\bigg(\!\!\bigg(\frac{\e(\cV_{n;\a}^{\wt\phi_{n;\a}})}{\e(\cN_{\Ga,\si;e}^{\phi})}\bigg)
\cap\big[\ov\fM_{\Ga,\si;e}^{\phi;\bT}\big]^{\vrt} \bigg)
\end{split}\EE
under the decomposition~\eref{FixedLocus_e}.
The identities~\eref{cVcZclasses_e3} and~\eref{cVcZclasses_e5} nearly split the equivariant
contribution of~$(\Ga,\si)$ to~\eref{PDeVdfn_e2}, 
i.e.~the left-hand side of~\eref{cVcZclasses_e3}, into contributions
from the components in~\eref{FixedLocus_e}.
The exceptional term~$\e(L_{\Ga})$ splits into products of Euler classes of 
$\bT^m$-equivariant line bundles, each of which involves only two components in~\eref{FixedLocus_e}.
We will associate each factor of~$\e(L_{\Ga})$ with the corresponding vertex~$v$
in $\nV_{+;3}^{\si}(\Ga)$ or in~$\nV_{+;2}^{\si}(\Ga)$.

\subsection{The edge and vertex contributions}
\label{NBsign_subs}

\noindent
It remains to compute the cap products in~\eref{cVcZclasses_e5} and to determine 
the factors of~$\e(L_{\Ga})$.
The first case of Lemma~\ref{VanV_lmm2} below justifies the alternative case of 
the fixed-edge contribution~\eref{ContrER_e}, i.e.~when it vanishes.
The second case of this lemma justifies the first restriction in~\eref{AdmissPairCond_e}
and in particular rules out contributions from the families $\ov\fM_{\Ga,\si;e}^{\phi;\bT;\circ}$
of real conics of Remark~\ref{ConicsSpaces_rmk}.
Lemma~\ref{AuxAct_lmm} expresses the contributions from the families $\ov\fM_{\Ga;e}^{\bT;\circ}$
of complex conics in terms of~\eref{ContrECb_e}.
The crucial fixed-edge contribution~\eref{ContrER_e} arises
from Proposition~\ref{REdgeCntr_prp}.
Lemmas~\ref{VanV_lmm2} and~\ref{AuxAct_lmm} are proved at the end of this section;
Proposition~\ref{REdgeCntr_prp} is established in Section~\ref{EdgeCntrPrp_subs}.
Analogously to~\eref{cVdfn_e}, let
$$\pi_{n;\a}\!:
\cV_{n;\a}=
\ov\fM_{g,l}\big(\cL_{n;\a},d\big)\lra  \ov\fM_{g,l}(\P^{n-1},d)\,.$$

\begin{lmm}\label{VanV_lmm2}
Suppose  $\Ga\!\in\!\ov\cA_{g,l}'(n,d)$ and $\E_{\R}^{\si}(\Ga)\!\neq\!\eset$.
If either $a_i\!\in\!2\Z$ for some $i\!\in\![k]$ or $\fd(e)\!\in\!2\Z$ for some 
$e\!\in\!\E_{\R}^{\si}(\Ga)$, then
$$\io_*\bigg(\frac{\e(\cV_{n;\a}^{\wt\phi_{n;\a}})}{\e(\cN_{\Ga,\si}^{\phi})}
\cap\big[\ov\cZ_{\Ga,\si}\big]^{\vrt}\bigg)=0
\in \cH_*^{\bT^m}\big(\ov\fM_{g,l}(\P^{n-1},d)^{\phi}\big),$$ 
where $\io\!:\cZ_{\Ga,\si}\!\lra\!\ov\fM_{g,l}(\P^{n-1},d)^{\phi}$ 
is the inclusion of the fixed locus.
\end{lmm}


\begin{lmm}\label{AuxAct_lmm}
Suppose $\Ga\!\in\!\ov\cA_{g,l}'(n,d)$, $e\!\in\!\E_{\C}^{\si}(\Ga)$,
\eref{ConicsCovCond_e} is satisfied, and $\b\!\in\!(\Z^{\ge0})^e$.
Then,
\BE{AuxAct_e}\begin{split}
&\bigg( \prod_{v\in e}\!c_1(L_{e;v})^{\b(v)}~
\frac{\e(\cV_{n;\a})}{\e(\cN_{\Ga;e})}\bigg)
\cap\big[\ov\fM_{\Ga;e}^{\bT}\big]\\
&\hspace{1in}
= -\bigg( \prod_{v\in e}\!\!\big(-\!\psi_{e;v}\big)^{\b(v)-|S_{e,v}|}\bigg)
\cdot
\begin{cases}
[\tn{RHS of~\eref{ContrECb_e}}],&\hbox{if}~k\!=\!1;\\
0,&\hbox{if}~k\!\ge\!2.
\end{cases}
\end{split}\EE
\end{lmm}

\begin{prp}\label{REdgeCntr_prp}
Suppose  $\Ga\!\in\!\ov\cA_{g,l}'(n,d)$ and $e\!\in\!\E_{\R}^{\si}(\Ga)$.
If $a_i\!\not\in\!2\Z$ for all $i\!\in\![k]$ and $\fd(e)\!\not\in\!2\Z$, then
\BE{REdgeCntr_e}
\frac{\e(\cV_{n;\a}^{\wt\phi_{n;\a}})}{\e(\cN_{\Ga,\si;e}^{\phi;\circ})}
\bigg|_{\ov\fM_{\Ga,\si;e}^{\phi;\bT;\circ}}
=\fd(e)\cdot\big[\tn{RHS of~\eref{ContrER_e}}\big]  \in \cH_{\bT^m}^*\,,\EE
with $\ov\fM_{\Ga,\si;e}^{\phi;\bT;\circ}$ as on the first line in~\eref{EdgeSpaces_e}.
\end{prp}

\begin{proof}[{\bf\emph{Proof of Theorem~\ref{EquivLocal_thm}}}]
By Lemmas~\ref{VanV_lmm} and~\ref{VanV_lmm2},
all nonzero $\bT^m$-equivariant contributions to~\eref{PDeVdfn_e2} come
from the $\bT^m$-fixed loci $\ov\cZ_{\Ga,\si}$ corresponding 
to the admissible pairs~$(\Ga,\si)$, i.e.~the elements of the collection 
$\cA_{g,l}(n,d)$ defined by~\eref{AdmissPairCond_e}.
Furthermore,  if a fixed locus~$\ov\cZ_{\Ga,\si}$ contributes to~\eref{PDeVdfn_e2},
then $\E_{\R}^{\si}(\Ga)\!=\!\eset$ if $a_i\!\in\!2\Z$ for some $i\!\in\![k]$ 
and $k\!=\!1$ if some $e\!\in\!\E_+^{\si}(\Ga)$ satisfies~\eref{ConicsCovCond_e}.
From now~on, we assume that $(\Ga,\si)$ is an admissible pair satisfying these two conditions.\\

\noindent
By~\eref{cNGadfn_e}, \eref{cLGadfn_e}, and the definitions of $\cN_{\Ga}'\P$ and $\cL_{\Ga}'$
just~below,
\BE{cNcLec_e}
\frac{\e(\cN_{\Ga}'\P)}{\e(\cL_{\Ga}')}\bigg|_{\wt\cZ_{\Ga,\si}}
=  \prod_{v\in\nV_{+;3}^{\si}(\Ga)}\!\!
\bigg(\frac{\e(T_{P_{\vt(v)}}\P^{n-1})}{\lr\a\al_{\vt(v)}^{k}}\bigg)^{\!|\E_v(\Ga)|}
\cdot
\prod_{v\in\nV_{+;2}^{\si}(\Ga)}\!\!\!
\bigg(\frac{\e(T_{P_{\vt(v)}}\P^{n-1})}{\lr\a\al_{\vt(v)}^{k}}\bigg).
\EE
By~\eref{LGadfn_e},
\BE{Lec_e}\begin{split}
\e(L_{\Ga})\big|_{\wt\cZ_{\Ga,\si}}
&=\prod_{v\in\nV_{+;3}^{\si}(\Ga)}\prod_{e\in\E_v(\Ga)}\!\!\!\!\!
\big(c_1(L_{e_{\bu};v}^{\phi})\!+\!c_1(L_{v;e}^{\phi})\big)
\cdot
\prod_{v\in\nV_{+;2}^{\si}(\Ga)}\!\!\!
\bigg(\!\!-\!\!\!\!\!\sum_{e\in\E_v(\Ga)}\!\!\!\!\!\psi_{e_{\bu};v}\bigg),\\
&=\prod_{v\in\nV_{+;3}^{\si}(\Ga)}
\Psi_{\Ga;v}^*\!\!\!\!\!\prod_{e\in\E_v(\Ga)}\!\!\!\!\!
\big(-\!\psi_{e;v}\!-\!\pi_v^*\psi_e\big)
\cdot
\prod_{v\in\nV_{+;2}^{\si}(\Ga)}\!\!\!
\bigg(\!\!-\!\!\!\!\!\sum_{e\in\E_v(\Ga)}\!\!\!\!\!\psi_{e;v}\bigg).
\end{split}\EE
The right-hand sides of~\eref{cNcLec_e} and~\eref{Lec_e} are elements of
$\cH_{\bT^m}^*$ and $\cH_{\bT^m}^*\!\otimes\!H^*(\wt\cZ_{\Ga,\si})$, respectively.\\

\noindent
By \cite[Theorem~1.4]{RealGWsII},
\BE{Csgn_e}\begin{split}
\Psi_{\Ga;v*}\bigg(\!\!
\bigg(\frac{\e(\cV_{n;\a}^{\wt\phi_{n;\a}})}{\e(\cN_{\Ga;v}^{\phi})}\bigg)
\!\cap\!\big[\ov\fM_{\Ga;v}^{\phi;\bT}\big]^{\vrt}\bigg) 
&=(-1)^{\g(v)-1+|S_v^-|} 
\bigg(\frac{\e(\cV_{n;\a})}{\e(\cN_{\Ga;v})}\bigg)
\!\cap\!\big[\ov\fM_{\Ga;v}^{\bT}\big]^{\vrt}\,,\\
\Psi_{\Ga,\si;e*}
\bigg(\!\!\bigg(\frac{\e(\cV_{n;\a}^{\wt\phi_{n;\a}})}{\e(\cN_{\Ga,\si;e}^{\phi})}\bigg)
\!\cap\!\big[\ov\fM_{\Ga,\si;e}^{\phi;\bT}\big]^{\vrt}\bigg)
&=(-1)^{-1+|S_e^-|} 
\bigg(\frac{\e(\cV_{n;\a})}{\e(\cN_{\Ga;e})}\bigg)\!\cap\!\big[\ov\fM_{\Ga;e}^{\bT}\big],
\end{split}\EE
for all $v\!\in\!\nV_{+;3}^{\si}(\Ga)$ and $e\!\in\!\E_+^{\si}(\Ga)$;
the right-hand sides above carry standard complex orientations.
Since 
$$\io_{\Ga}^*\big(\psi_i^{b_i},\ev_i^*\x^{p_i}\big)\big|_{\wt\cZ_{\Ga,\si}}=
\begin{cases}
(\pi_v^*\Psi_{\Ga;v}^*\psi_i^{p_i},\al_{\vt(v)}^{p_i}),&\hbox{if}~i^+\!\in\!S_v;\\
((-\pi_v^*\Psi_{\Ga;v}^*\psi_i)^{b_i},(-\al_{\vt(v)})^{p_i}),&\hbox{if}~i^-\!\in\!S_v;\end{cases}$$
the first statement in~\eref{Csgn_e} and \cite[Section~27.6]{MirSym} give
\BE{Ver3Contr_e}\begin{split}
&\Bigg(\frac{\prod\limits_{\begin{subarray}{c}1\le i\le l\\ i^{\pm}\in S_v\end{subarray}}
\!\!\!(\psi_i^{b_i}\ev_i^*\x^{p_i})}
{\Psi_{\Ga;v}^*\!\!\!\!\!\prod\limits_{e\in\E_v(\Ga)}\!\!\!\!\!(-\psi_{e;v}\!-\!\pi_v^*\psi_e)}
\frac{\e(\cV_{n;\a}^{\wt\phi_{n;\a}})}{\e(\cN_{\Ga;v}^{\phi})}
\Bigg)\cap \big[\ov\fM_{\Ga;v}^{\phi;\bT}\big]^{\vrt}\\
&\hspace{.8in}=(-1)^{\fs_v-|\E_v(\Ga)|}
\bigg(\frac{\e(T_{P_{\vt(v)}}\P^{n-1})}{\lr\a\al_{\vt(v)}^k}\bigg)^{\!-1}
\al_{\vt(v)}^{|\p|_v} \!\!\!\!\!\!\!\!\!
\int\limits_{\ov\cM_{\g(v),S_v}}\!\!
\frac{\e(\bE^*\!\otimes\!T_{P_{\vt(v)}}\P^{n-1})}
{\prod\limits_{e\in\E_v(\Ga)}\!\!\!\!\!\!\left(-\psi_{e;v}\!-\!\psi_e\right)}
\!\!\prod_{\begin{subarray}{c}1\le i\le l\\ i^{\pm}\in S_v\end{subarray}}\!\!\!\psi_i^{b_i}
\end{split}\EE
for all $v\!\in\!\nV_{+;3}^{\si}(\Ga)$.\\

\noindent
For any $e\!\in\!\Edg$ and $v\!\in\!e$, let 
$$|\b|_{e,v}=\sum_{\begin{subarray}{c}1\le i\le l\\ i^{\pm}\in S_{e,v}\end{subarray}}\!\!\!\!\!b_i\,,
\quad
|\b|_{e,v}^-=\sum_{\begin{subarray}{c}1\le i\le l\\ i^-\in S_{e,v}\end{subarray}}\!\!\!\!\!b_i\,,
\quad
|\p|_{e,v}=\sum_{\begin{subarray}{c}1\le i\le l\\ i^{\pm}\in S_{e,v}\end{subarray}}\!\!\!\!\!p_i\,,
\quad
|\p|_{e,v}^-=\sum_{\begin{subarray}{c}1\le i\le l\\ i^-\in S_{e,v}\end{subarray}}\!\!\!\!\!p_i\,.$$
Since $S_{e,v}\!\cap\![l]$ consists of at most one element,
\BE{alclass_e}\io_{\Ga}^*\ev_i^*\x^{p_i}\big|_{\wt\cZ_{\Ga,\si}}
= (-1)^{|\p|_{e,v}^-}\al_{\vt(v)}^{|\p|_{e,v}}
\qquad\forall~i\!\in\!S_{e,v}\!\cap\![l].\EE
If $e\!\in\!\E_+^{\si}(\Ga)$  does not satisfy~\eref{ConicsCovCond_e}, then
$$\frac{\e(\cV_{n;\a})}{\e(\cN_{\Ga;e})}=
\bigg(\!\prod_{v\in e}\!(-\psi_{e;v})^{|S_{e,v}|}\!\bigg)^{\!-1}
\ff_{\Ga;e}^*\bigg(\frac{\e(\cV_{n;\a})}{\e(\cN_{\Ga;e}^{\circ})}\bigg),\quad
\io_{\Ga}^*\psi_i^{b_i}\big|_{\wt\cZ_{\Ga,\si}}\!=
(-1)^{|\b|_{e,v}^-}\psi_{e;v}^{|\b|_{e,v}}\,.$$
The second statement in~\eref{Csgn_e} and \cite[Sections~27.2,27.6]{MirSym} thus give
\BE{EdgCContr_e}\begin{split}
&\Bigg(\prod_{\begin{subarray}{c}1\le i\le l\\ i^{\pm}\in S_e\end{subarray}}
\!\!\!(\psi_i^{b_i}\ev_i^*\x^{p_i})
\frac{\e(\cV_{n;\a}^{\wt\phi_{n;\a}})}{\e(\cN_{\Ga,\si;e}^{\phi})}
\Bigg)\cap \big[\ov\fM_{\Ga,\si;e}^{\phi;\bT}\big]^{\vrt}\\
&\hspace{.8in}
=(-1)^{|S_e^-|}\prod_{v\in e}\frac{
\prod\limits_{\begin{subarray}{c}1\le i\le l\\ i^{\pm}\in S_{e,v}\end{subarray}}\!\!\!\!
(-1)^{|\b|_{e,v}^-+|\p|_{e,v}^-}\psi_{e;v}^{|\b|_{e,v}}\al_{\vt(v)}^{|\p|_{e,v}}}
{(-\psi_{e;v})^{|S_{e,v}|}}
\big[\tn{RHS of~\eref{ContrEC_e}}\big]
\end{split}\EE
for all $e\!\in\!\E_+^{\si}(\Ga)$ not satisfying~\eref{ConicsCovCond_e}.\\

\noindent
If $e\!\in\!\E_+^{\si}(\Ga)$ satisfies~\eref{ConicsCovCond_e}, then
$$\io_{\Ga}^*\psi_i^{b_i}\big|_{\wt\cZ_{\Ga,\si}}\!=
(-1)^{|\b|_{e,v}^-}
\big(-c_1(L_{e;v})\big)^{|\b|_{e,v}}\,.$$
The second statement in~\eref{Csgn_e},
Lemma~\ref{AuxAct_lmm}, and~\eref{alclass_e} thus give
\BE{EdgCbContr_e}\begin{split}
&\Bigg(\prod_{\begin{subarray}{c}1\le i\le l\\ i^{\pm}\in S_e\end{subarray}}
\!\!\!(\psi_i^{b_i}\ev_i^*\x^{p_i})
\frac{\e(\cV_{n;\a}^{\wt\phi_{n;\a}})}{\e(\cN_{\Ga,\si;e}^{\phi})}
\Bigg)\cap \big[\ov\fM_{\Ga,\si;e}^{\phi;\bT}\big]^{\vrt}\\
&\hspace{.8in}
=(-1)^{|S_e^-|}\prod_{v\in e}\frac{
\prod\limits_{\begin{subarray}{c}1\le i\le l\\ i^{\pm}\in S_{e,v}\end{subarray}}\!\!\!\!
(-1)^{|\b|_{e,v}^-+|\p|_{e,v}^-}\psi_{e;v}^{|\b|_{e,v}}\al_{\vt(v)}^{|\p|_{e,v}}}{(-\psi_{e;v})^{|S_{e,v}|}}
\big[\tn{RHS of~\eref{ContrECb_e}}\big]
\end{split}\EE
for all $e\!\in\!\E_+^{\si}(\Ga)$ satisfying~\eref{ConicsCovCond_e}.\\

\noindent
If $e\!\in\!\E_{\R}^{\si}(\Ga)$, then precisely one of the vertices $v_e\!\in\!e$
belongs to~$\nV_+^{\si}(\Ga)$ and 
$$\io_{\Ga}^*\ev_i^*\x^{p_i}\big|_{\wt\cZ_{\Ga,\si}} =
(-1)^{|\p|_{e,v_e}^-}\al_{\vt(v_e)}^{|\p|_{e,v_e}}\,, 
\quad
\io_{\Ga}^*\psi_i^{b_i}\big|_{\wt\cZ_{\Ga,\si}} =
(-1)^{|\b|_{e,v_e}^-}\psi_{e;v_e}^{|\b|_{e,v_e}}\,.$$
Since $\ov\fM_{\Ga,\si;e}^{\phi;\bT}$ consists of a single point with 
the automorphism group~$\Z_{\fd(e)}$,
Proposition~\ref{REdgeCntr_prp} thus~gives
\BE{EdgRContr_e}\begin{split}
&\Bigg(\prod_{\begin{subarray}{c}1\le i\le l\\ i^{\pm}\in S_e\end{subarray}}
\!\!\!(\psi_i^{b_i}\ev_i^*\x^{p_i})
\frac{\e(\cV_{n;\a}^{\wt\phi_{n;\a}})}{\e(\cN_{\Ga,\si;e}^{\phi})}
\Bigg)\cap \big[\ov\fM_{\Ga,\si;e}^{\phi;\bT}\big]^{\vrt}\\
&\hspace{.8in}
=\frac{\prod\limits_{\begin{subarray}{c}1\le i\le l\\ i^{\pm}\in S_{e,v}\end{subarray}}\!\!\!\!
(-1)^{|\b|_{e,v_e}^-+|\p|_{e,v_e}^-}\psi_{e;v_e}^{|\b|_{e,v_e}}
\al_{\vt(v_e)}^{|\p|_{e,v_e}}}{(-\psi_{e;v_e})^{|S_{e,v_e}|}}
\big[\tn{RHS of~\eref{ContrER_e}}\big]
\end{split}\EE
for all $e\!\in\!\E_{\R}^{\si}(\Ga)$.\\

\noindent
By~\eref{cVcZclasses_e3}, \eref{cVcZclasses_e5},  and~\eref{Lec_e},
the left-hand side of~\eref{EquivLocal_e} is the product~of 
the leading fraction in~\eref{cVcZclasses_e3} with the right-hand sides 
of~\eref{cNcLec_e}, \eref{Ver3Contr_e}, 
\eref{EdgCContr_e}, \eref{EdgCbContr_e}, and~\eref{EdgRContr_e}.
The factor on the right-hand side of~\eref{EquivLocal_e} corresponding to
a vertex $v\!\in\!\nV_{+;3}^{\si}(\Ga)$ is the product~of
\begin{enumerate}[label=$\bu$,leftmargin=*]

\item the corresponding factor on RHS of~\eref{cNcLec_e},

\item RHS of~\eref{Ver3Contr_e},

\item the non-edge factors on RHSs 
of~\eref{EdgCContr_e} and~\eref{EdgCbContr_e} corresponding to either~$v\!\in\!e$
or $\si(v)\!\in\!e$,

\item the non-edge factors on  RHS of~\eref{EdgRContr_e} with $v_e\!=\!e$,
and 

\item $(-1)$ to the power of the summand $|\E_v(\Ga)|$ in~\eref{ndGadfn_e}.

\end{enumerate}
The factor on the right-hand side of~\eref{EquivLocal_e} corresponding to
a vertex $v\!\in\!\nV_{+;2}^{\si}(\Ga)$ is obtained similarly,
except the product of the contributions from the second and last bullets above is replaced~by
$$\bigg(\!\!-\!\!\!\!\!\sum_{e\in\E_v(\Ga)}\!\!\!\!\!\psi_{e;v}\bigg)^{-1}\cdot (-1)^1\,.$$
The role of~\eref{Contr3V_e} now is played by~\eref{Contr2V_e} with 
$$\val_v(\Ga),\big|\E_v(\Ga)\big|=2, \qquad |\b|_v,|\p|_v=0\,.$$
There is no contribution associated with $v\!\in\!\nV_+^{\si}(\Ga)$ such that $\val_v(\Ga)\!=\!1$;
this is consistent~with
$$(-1)^{\fs_v}\big[\tn{RHS of~\eref{Contr2V_e}}\big] =1$$
in this case.
For $v\!\in\!\nV_+^{\si}(\Ga)$ such that $\val_v(\Ga)\!=\!2$ and $|\E_v(\Ga)|\!=\!1$,
there is a unique edge~$e$ in $\E_{\R}^{\si}(\Ga)\!\cup\!\E_+^{\si}(\Ga)$ containing either~$v$ or~$\si(v)$.
The factor on the right-hand side of~\eref{EquivLocal_e}  
corresponding to such~$v$ is the non-edge factor on the right-hand side of
either \eref{EdgCContr_e}, \eref{EdgCbContr_e}, or \eref{EdgRContr_e}
corresponding to the associated flag~$(e,v)$.
This is the only case for which the numbers $|\b|_{e,v}^-$ and $|\p|_{e,v}^-$  may be nonzero.
\end{proof}

\begin{proof}[{\bf\emph{Proof of Lemma~\ref{VanV_lmm2}}}]  
Suppose $e\!\in\!\E_{\R}^{\si}(\Ga)$.
If $k\!\in\!\Z^+$ and either $a_i\!\in\!2\Z$ for some $i\!\in\![k]$ or $\fd(e)\!\in\!2\Z$, then
$$\cV_{n;\a}^{\wt\phi_{n;\a}}\lra \ov\fM_{\Ga,\si;e}^{\phi}$$
contains a subbundle of odd rank and thus $\e(\cV_{n;\a}^{\wt\phi_{n;\a}})\!=\!0$.
In the first case, this is the subbundle associated with the $i$-th factor in
the vector bundle~$\cL_{n;\a}$.
In the second case,  the subbundle associated with every factor in~$\cL_{n;\a}$
has odd rank.\\

\noindent
It remains to consider the case $\fd(e)\!\in\!2\Z$ and $k\!=\!0$. 
The latter implies that $X_{n;\a}\!=\!\P^{2m-1}$.
If $\phi\!=\!\eta_{2m}$, then 
$$\ov\fM_{\Ga,\si;e}^{\phi;\bT} \subset
    \ov\fM_{\Ga,\si;e}^{\phi}\!\equiv\!\ov\fM_{0,S_e}\big(\P^{2m-1},\fd(e)\big)^{\phi}=\eset;$$
see \cite[Lemma~1.9]{Teh}.
Suppose $\phi\!=\!\tau_{2m}'$.
The set $\ov\fM_{\Ga,\si;e}^{\phi;\bT}$ then consists of two equivalence classes of
real maps: one with the standard involution $\tau\!=\!\tau_2$ on the domain
and the other with the fixed-point-free involution $\eta\!=\!\eta_2$.
We denote the associated uncompactified moduli spaces of real maps~by
$$\fM_{\Ga,\si;e}^{\phi,\tau;\bT} \subset\fM_{\Ga,\si;e}^{\phi,\tau}
\qquad\hbox{and}\qquad  \fM_{\Ga,\si;e}^{\phi,\eta;\bT} \subset\fM_{\Ga,\si;e}^{\phi,\eta}\,,$$
respectively.\\

\noindent
A real orientation on $(\P^{2m-1},\tau_{2m})$ directly determines orientations on 
$\fM_{\Ga,\si;e}^{\phi,\tau}$ and $\fM_{\Ga,\si;e}^{\phi,\eta}$;
see \cite[Corollary~5.10]{RealGWsI}.
In order to extend some set of orientations across the common boundary of these two moduli spaces,
we reverse the orientation of the second moduli space;
see the end of \cite[Section~3.2]{RealGWsI}.
Thus, it is sufficient to show~that 
\BE{VanV_lmm2e3}\e\big(\cN_{\Ga,\si;e}^{\phi}\big)\big|_{\fM_{\Ga,\si;e}^{\phi,\tau;\bT}}
=\e\big(\cN_{\Ga,\si;e}^{\phi}\big)\big|_{\fM_{\Ga,\si;e}^{\phi,\eta;\bT}}
\in \cH_{\bT^m}^*\EE 
before the orientation reversal.
Since adding a pair of conjugate points has the same effect on the two sides of~\eref{VanV_lmm2e3},
we can assume that $S_e$ consists of a pair of conjugate points
and the value of~$\vt$ on the positive vertex is some $i\!\in\!\Z^+\!-\!2\Z$.\\

\noindent
By \cite[Proposition~5.5]{Teh}, the algebraic orientations on 
$\fM_{\Ga,\si;e}^{\phi,\tau}$ and $\fM_{\Ga,\si;e}^{\phi,\eta}$
defined in \cite[Section~5.2]{Teh} also do not extend across the common boundary.
By \cite[Lemma~5.1]{Teh}, the moduli space $\ov\fM_{\Ga,\si;e}^{\phi}$ is connected.
Thus, the algebraic orientations on 
$\fM_{\Ga,\si;e}^{\phi,\tau}$ and $\fM_{\Ga,\si;e}^{\phi,\eta}$ are either both the same
or both opposite of the orientations induced by a real orientation on $(\P^{2m-1},\tau_{2m})$.
In either case, it is sufficient to establish~\eref{VanV_lmm2e3} 
for the algebraic orientations on the two moduli spaces.
By \cite[Remark~6.9]{Teh}, both classes in~\eref{VanV_lmm2e3}
are then the negative of the right-hand side of \cite[(6.6)]{Teh} with $d_0\!=\!\fd(e)$
and $\la_i\!=\!\al_i$.
In particular, they are equal.
\end{proof}

\begin{rmk}\label{VanV2_rmk}
Suppose $\fd(e)\!\in\!2\Z$ as above.
If $m\!\in\!2\Z$, the orientation on $\fM_{\Ga,\si;e}^{\phi,\tau}$
induced by the canonical real orientation of Section~\ref{CROprop_subs}
is the same as the orientation induced by the canonical spin structure
of \cite[Section~5.5]{Teh};
see  \cite[Theorem~1.5]{RealGWsII}.
The two classes in~\eref{VanV_lmm2e3}
with respect to the canonical real orientation (and before the orientation reversal)
are thus given by~\cite[(6.6)]{Teh}.
If $m\!\not\in\!2\Z$, the orientation on $\fM_{\Ga,\si;e}^{\phi,\tau}$
induced by the canonical real orientation of Section~\ref{CROprop_subs}
is the same as the orientation induced by the associated relative spin structure; 
see  \cite[Theorem~1.5]{RealGWsII}.
The latter is the same as the orientation induced by the relative spin structure
of \cite[Remark~6.5]{Teh}.
By the beginning of \cite[Remark~6.9]{Teh} and the middle of the preceding paragraph in~\cite{Teh},
the two classes in~\eref{VanV_lmm2e3}
with respect to the canonical real orientation (and before the orientation reversal)
are thus again given by~\cite[(6.6)]{Teh}.
\end{rmk}

\begin{proof}[{\bf\emph{Proof of Lemma~\ref{AuxAct_lmm}}}] 
We use the $\bT^2$-action on $\P^2_{\vt(v_1),\vt(v_2),n}$ induced by the $\bT^2$-action on~$\C^3$
with weights~$\al_{\vt(v_1)}$, $\al_{\vt(v_2)}$, and~0 and set $\al_{\vt(v_1)}\!=\!-\al_{\vt(v_2)}$
{\it after} computing the equivariant contribution to the left-hand side of~\eref{AuxAct_e}
from the fixed loci of this action on~$\ov\fM_{\Ga;e}^{\bT}$.\\

\noindent
The fixed locus of the $\bT^2$-action on $\ov\fM_{\Ga;e}^{\bT}$ consists of two points:
the $\bT^2$-invariant degree~$\fd(e)$ cover of the line $\P^1_{\vt(v_1),\vt(v_2)}$
and
the $\bT^2$-invariant degree~$\fd(e)/2$ cover of $\P^1_{\vt(v_1),n}\!\cup\!\P^1_{\vt(v_2),n}$.
Let 
$$\cZ_{\Ga;e}'^{\bT},\cZ_{\Ga';e}^{\bT}
\subset\ov\fM_{\Ga;e}^{\bT}\subset\ov\fM_{\Ga;e}\equiv \ov\fM_{0,S_e}\big(\P^{n-1},\fd(e)\big) $$
be the corresponding one-element subspaces.
Denote~by 
$$\cZ_{\Ga;e}'^{\bT;\circ}\subset\ov\fM_{\Ga;e}^{\bT;\circ}
\subset 
\ov\fM_{\Ga;e}^{\circ}\equiv \ov\fM_{0,0}\big(\P^{n-1},\fd(e)\big)$$
the image of $\cZ_{\Ga;e}'^{\bT}$ under the forgetful morphism and by 
$$\cN\cZ_{\Ga;e}'^{\bT},\cN_{\Ga;e}'\lra \cZ_{\Ga;e}'^{\bT}
\qquad\hbox{and}\qquad
\cN\cZ_{\Ga;e}'^{\bT;\circ},\cN_{\Ga;e}'^{\circ} \lra \cZ_{\Ga;e}'^{\bT;\circ}$$
the normal bundles in the intermediate subspaces and in the ambient moduli spaces,
respectively.\\

\noindent
By the proof of Lemma~\ref{VanV_lmm}, $\e(\cV_{n;\a})\big|_{\cZ_{\Ga';e}^{\bT}}\!=\!0$.
From the classical Atiyah-Bott Localization Theorem~\cite{AB},
we thus find~that 
\BE{AuxAct_e1}
\bigg( \prod_{v\in e}\!c_1(L_{e;v})^{\b(v)}~
\frac{\e(\cV_{n;\a})}{\e(\cN_{\Ga;e})}\bigg)
\cap\big[\ov\fM_{\Ga;e}^{\bT}\big]
=\bigg( \prod_{v\in e}\!c_1(L_{e;v})^{\b(v)}~
\frac{\e(\cV_{n;\a})}{\e(\cN_{\Ga;e})}\bigg)\bigg|_{\cZ_{\Ga;e}'^{\bT}}
\cdot\frac{1}{\e(\cN\cZ_{\Ga;e}'^{\bT})}\EE
with respect to the specified $\bT^2$-action.\\

\noindent
By Exercise~27.2.4 and~(27.8) in \cite{MirSym},
\BE{AuxAct_e5}\begin{split}
&\bigg(\frac{1}{\e(\cN\cZ_{\Ga;e}'^{\bT;\circ})}\frac{\e(\cV_{n;\a})}{\e(\cN_{\Ga;e}^{\circ})}
\bigg) \cap\big[\cZ_{\Ga;e}'^{\bT;\circ}\big]
=\bigg(\frac{\e(\cV_{n;\a})}{\e(\cN_{\Ga;e}^{'\circ})}\bigg) \cap\big[\cZ_{\Ga;e}'^{\bT;\circ}\big]\\
&=-\frac{\lr\a}{\fd(e)\,(\fd(e)!)^2}
\cdot\frac{\left(\frac{\al_{\vt(v_1)}+\al_{\vt(v_2)}}{2}\right)^{\!k-1}}
{\prod\limits_{\begin{subarray}{c}1\le j<n\\ j\neq\vt(v_1),\vt(v_2)\end{subarray}}\hspace{-.25in}
\left(\frac{\al_{\vt(v_1)}+\al_{\vt(v_2)}}{2}\!-\!\al_j\right)}\\
&\hspace{.5in}\times
\frac{\prod\limits_{i=1}^k\prod\limits_{r=1}^{a_i\fd(e)/2}\!\!
\left(\! a_i^2\!\left(\frac{\al_{\vt(v_1)}+\al_{\vt(v_2)}}{2}\right)^{\!2} \!-\! 
r^2\!\left(\frac{\al_{\vt(v_1)}-\al_{\vt(v_2)}}{\fd(e)}\right)^{\!2}\right)}
{ \left(\frac{\al_{\vt(v_1)}-\al_{\vt(v_2)}}{\fd(d)}\right)^{2\fd(e)-2} \!\!\!\!\!\!\!\!\!\!\! 
\prod\limits_{j\neq\vt(v_1),\vt(v_2)}
\!\!\prod\limits_{r=1}^{\fd(e)/2}\!\!
\left(\!\!\left(\frac{\al_{\vt(v_1)}+\al_{\vt(v_2)}}{2}\!-\!\al_j\right)^{\!2} \!-\! 
r^2\left(\frac{\al_{\vt(v_1)}-\al_{\vt(v_2)}}{\fd(e)}\right)^{\!2}\right)}\,,
\end{split}\EE
with $\al_n\!\equiv\!0$.
For $\al_{\vt(v_1)}\!=\!-\al_{\vt(v_2)}$, the right-hand side of this expression
reduces to the negative of the right-hand side of~\eref{ContrECb_e} if $k\!=\!1$
and to~0 if $k\!\ge\!2$.
Since 
$$\e\big(\cN\cZ_{\Ga;e}'^{\bT}\big)=
\ff_{\Ga;e}^*\e\big(\cN\cZ_{\Ga;e}'^{\bT;\circ}\big), \qquad
\frac{\e(\cV_{n;\a})}{\e(\cN_{\Ga;e})}=
\bigg(\!\prod_{v\in e}\!c_1(L_{e;v})^{|S_{e,v}|}\!\bigg)^{\!-1}
\ff_{\Ga;e}^*\bigg(\frac{\e(\cV_{n;\a})}{\e(\cN_{\Ga;e}^{\circ})}\bigg),$$
and $c_1(L_{e;v})|_{\cZ_{\Ga;e}'^{\bT}}\!=\!-\psi_{e;v}$,  
the claim follows from~\eref{AuxAct_e1}  and~\eref{AuxAct_e5}.
\end{proof}

\subsection{Proof of Proposition~\ref{REdgeCntr_prp}}
\label{EdgeCntrPrp_subs}

\noindent
Let $n'\!=\!n\!-\!k$,
$$c=\begin{cases} \tau,&\hbox{if}~\phi\!=\!\tau_n';\\
\eta,&\hbox{if}~\phi\!=\!\eta_n;
\end{cases}
\qquad\hbox{and}\qquad 
G_c=\Aut(\P^1,c).$$
The 3-dimensional Lie group~$G_c$ is oriented by the positive rotation around $0\!\in\!\P^1$
and  the complex orientation of $T_0\P^1$;
see \cite[Section~1.4]{RealGWsII}.
This choice of orientation  is not directly relevant to the present proof, as it is contained
in the formulas from~\cite{Teh} we cite.
Denote~by \hbox{$\P^{n'-1}\!\subset\!\P^{n-1}$}
the span of the first $n'$ homogenous coordinates and~by
$$\big(V_c,\vph_c\big)\lra \big(\P^{n'-1},\phi'\big),$$
where $\phi'\!=\!\phi|_{\P^{n'-1}}$,
the normal bundle of $\P^{n'-1}$ in~$\P^{n-1}$.
This holomorphic real bundle pair is the restriction of the last $k$~components 
of the middle term in~\eref{Pnses_e} to~$\P^{n'-1}$.
If $(V,\vph)$ is a real bundle pair over $(\P^1,c)$, let 
$$\Ga(\P^1;V)^{\vph}\equiv\big\{\xi\!\in\!\Ga(\P^1;V)\!:\,\xi\!\circ\!c\!=\!\vph\!\circ\!\xi\big\}$$
denote the space of \sf{real sections}.\\

\noindent
Let $e\!=\!\{v_1,v_2\}$ so that $\vt(v_1)\!\not\in\!2\Z$  and 
\begin{gather*}
\fM_{\Ga,\si;e}^{\phi,c;\circ}=\fM_{0,0}\big(\P^{n-1},\fd(e)\big)^{\phi,c}, \\
\fM(\P^{n'-1})=\fM_{0,0}\big(\P^{n'-1},\fd(e)\big)^{\phi',c}, \qquad
\fM(X_{n;\a})=\fM_{0,0}\big(X_{n;\a},\fd(e)\big)^{\phi_{n;\a},c}.
\end{gather*}
By \cite[Lemma~5.1]{Teh},  the space 
$$\ov\fM_{\Ga,\si;e}^{\phi;\bT;\circ}=\fM_{\Ga,\si;e}^{\phi,c;\bT;\circ}$$
consists of one element: the equivalence class of the unmarked
real degree~$\fd(e)$ covering 
$$f_0\!:\big(\P^1,c,0,\i\big)\lra 
\big(\P^1_{\vt(v_1)\vt(v_2)},\phi,P_{\vt(v_1)},P_{\vt(v_2)}\big)$$
branched only over $P_{\vt(v_1)}$ and $P_{\vt(v_2)}\!=\!P_{\phi(\vt(v_1))}$.
We denote~by
$$G_c(f_0)\subset H^0\big(\P^1;f_0^*T\P^1_{\vt(v_1)\vt(v_2)}\big)^{f_0^*\tnd\phi}$$
the tangent space to the orbit of the $G_c$-action on the space of 
parametrized branched covers (by the composition with the inverse of each automorphism as usual).
The orientation on $G_c$ induces an orientation on~$G_c(f_0)$.\\

\noindent
Let $\dbar_{\cL_{n;\a}}$ denote
the standard  $\dbar$-operator on the real bundle pair 
$$f_0^*\big(\cL_{n;\a},\wt\phi_{n;\a}\big)\lra (\P^1,c).$$
The evaluations of real holomorphic sections and their derivatives at $z\!=\!0$ induce 
a $\bT^m$-equivariant isomorphism
\BE{Corienttau_e1}H^0\big(\P^1;f_0^*\cL_{n;\a}\big)^{f_0^*\wt\phi_{n;\a}}\lra
\bigoplus_{i=1}^k\bigoplus_{r=0}^{(a_i\fd(e)-1)/2}\!\!\!\!\!\!\!\!\! 
\cO_{\P^{n-1}}(a_i)\big|_{P_{\vt(v_1)}}\!\otimes\!\big(T_0^*\P^1\big)^{\otimes r}\,.\EE
The orientation on 
$$\det\big(\dbar_{\cL_{n;\a}}\big)=
\La_{\R}^{\top}\big(H^0\big(\P^1;f_0^*\cL_{n;\a}\big)^{f_0^*\wt\phi_{n;\a}}\big)$$
induced by the isomorphism~\eref{Corienttau_e1}
is called the \sf{complex orientation} in \cite[Section~3.3]{RealGWsII}.
Using
\BE{Corienttau_e2a}
\e\big(\cO_{\P^{n-1}}(a_i)|_{P_{\vt(v_1)}}\big)=a_i\al_{\vt(v_1)},
\qquad  \e\big(T_0^*\P^1\big)=\frac{\al_{\vt(v_2)}\!-\!\al_{\vt(v_1)}}{\fd(e)}
=-\frac{2\al_{\vt(v_1)}}{\fd(e)},\EE
we find that 
\BE{Corienttau_e2}\begin{split}
\e\big(H^0\big(\P^1;f_0^*\cL_{n;\a}\big)^{f_0^*\wt\phi_{n;\a}}\big)
&=\prod_{i=1}^k\prod_{r=0}^{(a_i\fd(e)-1)/2}\!\!
\bigg(a_i\al_{\vt(v_1)}-r\frac{2\al_{\vt(v_1)}}{\fd(e)}\bigg)\\
&=\prod_{i=1}^k\!\!\left(a_i\fd(e)\right)!!\,
\bigg(\frac{\al_{\vt(v_1)}}{\fd(e)}\bigg)^{\!\!\frac{|\a|\fd(e)+k}{2}}
\end{split}\EE
with respect to the complex orientation; see \cite[Section~27.2]{MirSym}.\\

\noindent
By Section~\ref{CROprop_subs}, 
the canonical real orientation on $(X_{n;\a},\phi_{n;\a})$ 
does not depend on the ordering of pairs $(2i\!-\!1,2i)$ of homogeneous coordinates on~$\P^{n-1}$.
Thus, we can assume that $\vt(v_1)\!\le\!n'$.
Denote~by $\dbar_{\P^{n'-1}}$ and $\dbar_{V_c}$ the standard $\dbar$-operators 
on the real bundle pairs
$$f_0^*\big(T\P^{n'-1},\tnd\phi\big),
f_0^*\big(V_c,\vph_c\big)\lra (\P^1,c),$$
respectively.
Combining
$$\e\big(T_{P_{\vt(v_1)}}\P^1_{\vt(v_1)j}\big)=\al_{\vt(v_1)}-\al_j
\qquad\forall~j\!\in\![n]\!-\!\vt(v_1)$$
with the second statement in~\eref{Corienttau_e2a}, we~obtain
\BE{Corienttau_e5}\begin{split}
\e\big(H^0\big(\P^1;f_0^*V_c\big)^{f_0^*\vph_c}\big)
&=\prod_{j=n'+1}^n \!\!\!\!
\prod_{r=0}^{(\fd(e)-1)/2}\!\!\!
\bigg(\al_{\vt(v_1)}\!-\!\al_j-r\frac{2\al_{\vt(v_1)}}{\fd(e)}\bigg)\\
&=\prod_{j=n'+1}^n
\!\!\!\!\!\prod\limits_{r=0}^{(\fd(e)-1)/2}\!
\left(\frac{(\fd(e)\!-\!2r)\al_{\vt(v_1)}}{\fd(e)}\!-\!\al_j\right)
\end{split}\EE
with respect to the complex orientation on the determinant of the standard  $\dbar$-operator 
on~$(V_c,\vph_c)$.\\

\noindent
Let
\BE{RealOrientCI_e19}\begin{split}
0\lra (V,\vph) &\lra 
\big(T\P^{n'-1},\tnd\phi|_{T\P^{n'-1}}\big)\big|_{\P^1_{\vt(v_1)\vt(v_2)}}\oplus
(V_c,\vph_c)\big|_{\P^1_{\vt(v_1)\vt(v_2)}} \\
&\hspace{1.5in}\lra
\big(\cL_{n;\a},\wt\phi_{n;\a}\big)\big|_{\P^1_{\vt(v_1)\vt(v_2)}} \lra 0
\end{split}\EE
be an exact sequence of holomorphic real bundle pairs over 
$(\P^1_{\vt(v_1)\vt(v_2)},\phi)$ such~that 
$$T\P^1_{\vt(v_1)\vt(v_2)}\subset V.$$
Let $\dbar_V$ denote the standard $\dbar$-operator on the real bundle pair 
$$f_0^*(V,\vph)\lra (\P^1,c).$$
The real Cauchy-Riemann operators~$\dbar_V$ and $\dbar_{\P^{n'-1}}$ descend
to operators~$\dbar_V'$ and $\dbar_{\P^{n'-1}}'$ on the quotients of their domains~by
$$G_c(f_0)\subset H^0\big(\P^1;f_0^*T\P^1_{\vt(v_1)\vt(v_2)}\big)^{f_0^*\tnd\phi}
\subset \Ga\big(\P^1;f_0^*V\big)^{f_0^*\vph},\Ga\big(\P^1;f_0^*T\P^{n'-1}\big)^{f_0^*\tnd\phi}\,.$$
The exact sequence~\eref{RealOrientCI_e19}  gives rise to an exact sequence 
\BE{RealOrientCI_e21}
0\lra \dbar_V' \lra \dbar_{\P^{n'-1}}'\!\oplus\!\dbar_{V_c} \lra \dbar_{\cL_{n;\a}} \lra 0\EE
of Fredholm operators over~$(\P^1,c)$.
The operators $\dbar_{\P^{n'-1}}'$, $\dbar_{V_c}$, and $\dbar_{\cL_{n;\a}}$ are surjective
and the kernel of $\dbar_{\P^{n'-1}}'$ is canonically isomorphic to the tangent space of
$\fM(\P^{n'-1})$ at~$[f_0]$.\\

\noindent
{\bf\emph{The case $\phi\!=\!\tau_n'$ and $n\!-\!|\a|\!\in\!4\Z$.}}
By Lemma~\ref{CIorient_lmm}, $n'\!\in\!4\Z$.
The canonical spin structure on the real locus~of 
\BE{Corienttau_e6}\frac{n'}{2}\big(2\cO_{\P^{n-1}}(1),\wt\tau_{n;1,1}'^{(1)}\big)
\lra \big(\P^{n-1},\tau_n'\big)\EE
as in Section~\ref{RealOrientCI_subs} and the exact sequence~\eref{Pnses_e}
determine a spin structure on 
\BE{RPdfn_e}\R\P^{n'-1}=\Fix\big(\tau_{n'}\big)\subset\P^{n'-1}.\EE
This is the same spin structure as in \cite[Section~5.5]{Teh}. 
The equivariant Euler class of the tangent space 
of $[f_0,(0,\i)]$ in $\fM_{0,1}(\P^{n'-1},\fd(e))^{\phi'}$
with respect to the orientation induced by this spin structure is provided
by \cite[Proposition~6.2]{Teh} with $d_0\!=\!\fd(e)$, $i\!=\!\vt(v_1)$, and $\la_j\!=\!\al_j$.
The Euler class of  $[f_0]$ in $\fM(\P^{n'-1})$ is obtained by dividing the expression
in~\cite{Teh} by~$\e(T_0\P^1)$. Thus,
\BE{Corienttau_e9}\begin{split}
\e\big(T_{[f_0]}\fM(\P^{n'-1})\big)
 =(-1)^{\fd(e)}\fd(e)!\bigg(\frac{2\al_{\vt(v_1)}}{\fd(e)}\bigg)^{\!\!\fd(e)-1}
\!\!\!\!\!\!\!\!\!\!\!\!\!\!
\prod_{\begin{subarray}{c}1\le j\le n'\\ j\neq\vt(v_1),\, 2|(j-\vt(v_1))\end{subarray}}
\!\!\!\!\!\!\!\!\prod_{r=0}^{\fd(e)}
\!\!\bigg(\!\frac{\fd(e)\!-\!2r}{\fd(e)}\al_{\vt(v_1)}-\al_j\!\bigg)&\\
=(-1)^{\frac{\fd(e)-1}{2}}
2^{\fd(e)-1}\fd(e)!
\bigg(\frac{\al_{\vt(v_1)}}{\fd(e)}\bigg)^{\!\!\fd(e)-1} \!\!\!\!\!\!\!\!\!\!
\prod_{\begin{subarray}{c}1\le j\le n'\\ j\neq\vt(v_1),\vt(v_2)\end{subarray}}
\!\!\!\!\!\!\!\!\!\!\prod_{r=0}^{(\fd(e)-1)/2}\!\!\!
\bigg(\frac{(\fd(e)\!-\!2r)\al_{\vt(v_1)}}{\fd(e)}\!-\!\al_j\bigg)&.
\end{split}\EE\\

\noindent
Similarly to Section~\ref{RealOrientCI_subs},
the canonical spin structure on~\eref{RPdfn_e} induces a trivialization of~$V^{\vph}$
via the exact sequence~\eref{RealOrientCI_e19}
and thus an orientation on $\det \dbar_V'$ as in  \cite[Section~3.3]{RealGWsII}.
By \cite[Corollary~3.13]{RealGWsII}, 
the isomorphism
\BE{RealOrientCI_e23}\begin{split}
\big(\!\det\dbar_V'\big)\otimes \big(\!\det\!\big(\cV_{n;\a}^{\wt\phi_{n;\a}}|_{[f_0]}\big)\!\big)
&=\big(\!\det\dbar_V'\big)\otimes \big(\!\det\dbar_{\cL_{n;\a}}\big)\\
&\approx 
\La_{\R}^{\top}\big(T_{[f_0]}\fM(\P^{n'-1})\big)
\otimes \big(\!\det\dbar_{V_c}\big) 
= \det\!\big(\cN_{\Ga,\si;e}^{\phi;\circ}|_{[f_0]}\big)
\end{split}\EE
induced by~\eref{RealOrientCI_e21} is orientation-preserving with respect to the orientations
on the first terms on the two sides induced by the canonical spin structure on~\eref{RPdfn_e}
and the complex orientations on the other terms.
By~\eref{Corienttau_e2}, \eref{Corienttau_e5}, and~\eref{Corienttau_e9}, 
\eref{REdgeCntr_e} thus holds with respect to the  orientation on the left-hand side
induced by the orientation on~$\det\dbar_V'$ corresponding
to the canonical spin structure on the real locus of~\eref{Corienttau_e6}
via the exact sequences~\eref{Pnses_e} and~\eref{RealOrientCI_e19}.\\

\noindent
The orientation on $\fM(X_{n;\a})$ in this case
is determined by the distinguished homotopy class of
isomorphisms~\eref{RealOrientCI_e9a}
with $\ell_0(\a)\!=\!0$ and thus $\ell_1(\a)\!=\!k$.
Since the real line bundle $(L^*)^{\wt\phi^*}$ is orientable in this case,
the distinguished homotopy class of isomorphisms~\eref{RealOrientCI_e9a} 
determines a homotopy class of isomorphisms 
\begin{equation*}\begin{split}
 \big(X_{n;\a}^{\phi_{n;\a}}\!\times\!\R\big)\oplus
TX_{n;\a}^{\phi_{n;\a}}\oplus\cL_{n;\a}^{\wt\phi_{n;\a}}\big|_{X_{n;\a}^{\phi_{n;\a}}}
&\approx
 \big(X_{n;\a}^{\phi_{n;\a}}\!\times\!\R\big)\oplus
TX_{n;\a}^{\phi_{n;\a}}
\oplus k\cO_{\R\P^{n-1}}(1)\big|_{X_{n;\a}^{\phi_{n;\a}}}\\
&\approx n'\cO_{\R\P^{n-1}}(1)\big|_{X_{n;\a}^{\phi_{n;\a}}}
\oplus k\cO_{\R\P^{n-1}}(1)\big|_{X_{n;\a}^{\phi_{n;\a}}}\,.
\end{split}\end{equation*}
The canonical spin structure on the real locus of~\eref{Corienttau_e6}  thus determines 
a spin structure on~$TX_{n;\a}^{\phi_{n;\a}}$,
an orientation on $\fM(X_{n;\a})$, and 
an orientation on the left-hand side of~\eref{REdgeCntr_e} via 
the short exact sequence~\eref{CIprp_e4} pulled back to $\fM(X_{n;\a})$.
Since $\fM_{\Ga,\si;e}^{\phi,c;\circ}$ is connected, 
the last orientation agrees with the orientation described below~\eref{RealOrientCI_e23}.
By \cite[Theorem~1.5]{RealGWsII}, the orientation on $\fM(X_{n;\a})$
 induced  by the real orientation of Section~\ref{RealOrientCI_subs} is 
the same as the orientation induced by the canonical spin structure
on $(X_{n;\a},\phi_{n;\a})$.
Thus, the orientation on the left-hand side of~\eref{REdgeCntr_e} with the respect to
the canonical orientation on $\ov\fM_0(X_{n;\a},\fd(e))^{\phi_{n;\a}}$ 
arising from
Section~\ref{RealOrientCI_subs} is described by the right-hand side of~\eref{REdgeCntr_e}
in this~case.\\

\noindent
{\bf\emph{The case $\phi\!=\!\tau_n'$ and $n\!-\!|\a|\!\not\in\!4\Z$.}}
By Lemma~\ref{CIorient_lmm}, $n'\!+\!2\!\in\!4\Z$.
The canonical spin structure on the real locus~of 
\BE{Corienttau_e36}
\frac{n'\!+\!2}{2}\big(2\cO_{\P^{n-1}}(1),\wt\tau_{n;1,1}'^{(1)}\big)
\lra \big(\P^{n-1},\tau_n'\big)\EE
as in Section~\ref{RealOrientCI_subs} and the exact sequence~\eref{Pnses_e}
determine a relative spin structure on~\eref{RPdfn_e}.
This is the same relative spin structure as in \cite[Remark~6.5]{Teh}. 
The equivariant Euler class of the tangent space 
of $[f_0,(0,\i)]$ in $\fM_{0,1}(\P^{n'-1},\fd(e))^{\tau_{n'}'}$
with respect to the orientation induced by this relative spin structure is provided
by \cite[(6.21)]{Teh}; see \cite[Remark~6.6]{Teh}.
Thus,
\BE{Corienttau_e39}\begin{split}
&\e\big(T_{[f_0]}\fM_0(\P^{n'-1})\big)
=2^{\fd(e)-1}\fd(e)!
\bigg(\frac{\al_{\vt(v_1)}}{\fd(e)}\bigg)^{\!\!\fd(e)-1} \!\!\!\!\!\!\!\!\!\!
\prod_{\begin{subarray}{c}1\le j\le n'\\ j\neq\vt(v_1),\vt(v_2)\end{subarray}}
\!\!\!\!\!\!\!\!\!\!\prod_{r=0}^{(\fd(e)-1)/2}\!\!\!
\bigg(\frac{(\fd(e)\!-\!2r)\al_{\vt(v_1)}}{\fd(e)}\!-\!\al_j\bigg).
\end{split}\EE\\

\noindent
The canonical relative spin structure on~\eref{RPdfn_e} induces a relative spin structure on~$V^{\vph}$
via the exact sequence~\eref{RealOrientCI_e19}
and thus an orientation on $\det\dbar_V'$ as in  \cite[Section~3.3]{RealGWsII}.
By \cite[Corollary~3.13]{RealGWsII}, 
the isomorphism~\eref{RealOrientCI_e23}
induced by~\eref{RealOrientCI_e21} is orientation-preserving with respect to the orientations
on the first terms on the two sides induced by the canonical relative spin structure on~\eref{RPdfn_e}
and the complex orientations on the other terms.
By~\eref{Corienttau_e2}, \eref{Corienttau_e5}, and~\eref{Corienttau_e39}, 
\eref{REdgeCntr_e} without the leading sign in~\eref{ContrER_e}
thus holds with respect to the  orientation on the left-hand side
induced by the orientation on~$\det\dbar_V'$ corresponding
to the canonical relative spin structure on the real locus of~\eref{Corienttau_e36}
via the exact sequences~\eref{Pnses_e} and~\eref{RealOrientCI_e19}.\\

\noindent
The orientation on $\fM(X_{n;\a})$ in this case
is determined by the distinguished homotopy class of
isomorphisms~\eref{RealOrientCI_e9b}
with $\ell_0(\a)\!=\!0$ and thus $\ell_1(\a)\!=\!k$.
Since the real line bundle $(L^*)^{\wt\phi^*}$ is not orientable in this case,
the distinguished homotopy class of isomorphisms~\eref{RealOrientCI_e9b} 
determines a homotopy class of isomorphisms 
\begin{equation*}\begin{split}
 &\big(X_{n;\a}^{\phi_{n;\a}}\!\times\!\R\big)\oplus
\big(TX_{n;\a}^{\phi_{n;\a}}\!\oplus\!2(L^*)^{\wt\phi^*}\big|_{X_{n;\a}^{\phi_{n;\a}}}\big)
\oplus\cL_{n;\a}^{\wt\phi_{n;\a}}\big|_{X_{n;\a}^{\phi_{n;\a}}}\\
&\hspace{1.5in}\approx
 \big(X_{n;\a}^{\phi_{n;\a}}\!\times\!\R\big)\oplus
\big(TX_{n;\a}^{\phi_{n;\a}}\!\oplus\!2\cO_{\R\P^{n-1}}(1)\big|_{X_{n;\a}^{\phi_{n;\a}}}\big)
\oplus k\cO_{\R\P^{n-1}}(1)\big|_{X_{n;\a}^{\phi_{n;\a}}}\\
&\hspace{1.5in}\approx (n'\!+\!2)\cO_{\R\P^{n-1}}(1)\big|_{X_{n;\a}^{\phi_{n;\a}}}
\oplus k\cO_{\R\P^{n-1}}(1)\big|_{X_{n;\a}^{\phi_{n;\a}}}\,.
\end{split}\end{equation*}
The canonical spin structure on the real locus of~\eref{Corienttau_e36} thus determines 
a relative spin structure on~$TX_{n;\a}^{\phi_{n;\a}}$,
an orientation on $\fM(X_{n;\a})$, and 
an orientation on the left-hand side of~\eref{REdgeCntr_e} via 
the short exact sequence~\eref{CIprp_e4} pulled back to $\fM(X_{n;\a})$.
Since $\fM_{\Ga,\si;e}^{\phi,c;\circ}$ is connected (unless $n\!=\!2$), 
the last orientation agrees with the orientation described
at the end of the previous paragraph.\\

\noindent
By \cite[Theorem~1.5]{RealGWsII}, the orientation on $\fM(X_{n;\a})$
 induced  by the real orientation of Section~\ref{RealOrientCI_subs} differs from
 the orientation induced by the associated relative spin structure
on $(X_{n;\a},\phi_{n;\a})$ by~$(-1)$ to the power~of
$$\flr{\frac{\lr{c_1(X_{n;\a}),\fd(e)\ell}+2}{4}}=\frac{(n\!-\!|\a|)\fd(e)\!+\!2}{4}\,,$$
where $\ell\!\in\!H_2(\P^{n-1};\Z)$ is the homology class of a line.
Replacing the real bundle bundle pair $(L^*,\wt\phi^*)$ with 
$(\cO_{\P^{n-1}}(1),\wt\tau_{n;1}')$ above,
we obtain the  relative spin structure on $(X_{n;\a},\phi_{n;\a})$
induced by the canonical relative spin structure on~\eref{RPdfn_e}.
The orientations on $\fM_0(X_{n;\a},\fd(e))^{\phi_{n;\a},c}$ induced by  
the two relative spin structures differ  by~$(-1)$ to the power~of
$$\frac{1}{2}\blr{1\!-\!c_1(L^*),\fd(e)\ell}=\frac{(n\!-\!|\a|)\fd(e)\!+\!2\fd(e)}{4}\,.$$
Adding up the right-hand sides of the last two equations, 
we conclude
the orientation on $\fM(X_{n;\a})$
 induced  by the real orientation of Section~\ref{RealOrientCI_subs} differs from
 the orientation induced by the canonical relative spin structure 
by the leading sign in~\eref{ContrER_e}.
Thus, the orientation on the left-hand side of~\eref{REdgeCntr_e} with respect to
the canonical orientation on $\ov\fM_0(X_{n;\a},\fd(e))^{\phi_{n;\a}}$ 
arising from
Section~\ref{RealOrientCI_subs} is described by the right-hand side of~\eref{REdgeCntr_e}
in this~case as well.\\

\noindent 
{\bf\emph{The case $\phi\!=\!\eta_n$.}}
The top exterior power of the real bundle pair
\BE{Corienteta_e6}\frac{n'}{2}\big(2\cO_{\P^{n-1}}(1),\eta_{n;1,1}^{(1)}\big)
\lra \big(\P^{n-1},\eta_n\big)\EE
is canonically the square of a  rank~1 real bundle pair as in Section~\ref{RealOrientCI_subs}.
Thus, the restriction of~\eref{Corienteta_e6} to the equator 
$$S_{\vt(v_1)\vt(v_2)}^1\subset\P^1_{\vt(v_1)\vt(v_2)}$$
has a canonical homotopy class of trivializations.
This homotopy class of trivializations determines an orientation on each moduli space 
 $\fM_{0,l}(\P^{n'-1},\fd(e))^{\phi'}$ as in the proof 
of \cite[Lemma~2.5]{Teh}.
The equivariant Euler class of the tangent space of $[f_0,(0,\i)]$ in 
$\fM_{0,1}(\P^{n'-1},\fd(e))^{\phi'}$ with respect to this orientation
is again provided by \cite[Proposition~6.2]{Teh} and given by the same expression
as in the first case above.
Thus,
\BE{Corienteta_e9}\begin{split}
&\e\big(T_{[f_0]}\fM(\P^{n'-1})\big)\\
&\qquad
=(-1)^{\frac{\fd(e)-1}{2}}
2^{\fd(e)-1}\fd(e)!
\bigg(\frac{\al_{\vt(v_1)}}{\fd(e)}\bigg)^{\!\!\fd(e)-1} \!\!\!\!\!\!\!\!\!\!
\prod_{\begin{subarray}{c}1\le j\le n'\\ j\neq\vt(v_1),\vt(v_2)\end{subarray}}
\!\!\!\!\!\!\!\!\!\!\prod_{r=0}^{(\fd(e)-1)/2}\!\!\!
\bigg(\frac{(\fd(e)\!-\!2r)\al_{\vt(v_1)}}{\fd(e)}\!-\!\al_j\bigg).
\end{split}\EE\\

\noindent
The canonical square root structure on~\eref{Corienteta_e6} induces 
a homotopy class of trivializations of $(V,\vph)$ over $S_{\vt(v_1)\vt(v_2)}^1$
via the exact sequence~\eref{RealOrientCI_e19}
and thus an orientation on $\det\dbar_V'$ as in  \cite[Section~3.3]{RealGWsII}.
By \cite[Corollary~3.16]{RealGWsII},  the isomorphism~\eref{RealOrientCI_e23}
induced by~\eref{RealOrientCI_e21} is orientation-preserving with respect to the orientations
on the first terms on the two sides induced by the canonical square root structure on~\eref{Corienteta_e6}
and the complex orientations on the other terms.
By~\eref{Corienttau_e2}, \eref{Corienttau_e5}, and~\eref{Corienteta_e9}, 
\eref{REdgeCntr_e} without $|\phi|\!=\!1$ in the leading exponent 
thus holds with respect to the  orientation on the left-hand side
induced by the orientation on~$\det\dbar_V'$ corresponding
to the canonical square root structure on~\eref{Corienteta_e6}
via the exact sequences~\eref{Pnses_e} and~\eref{RealOrientCI_e19}
and the orienting procedure of \cite[Lemma~2.5]{Teh}.\\

\noindent
The orientation on $\fM(X_{n;\a})$ in this case
is determined by a real orientation on $(X_{n;\a},\phi_{n;\a})$
associated with the second isomorphism in~\eref{CIprp_e7}.
This isomorphism also determines an orientation on 
the tangent space of $\fM(X_{n;\a})$ at an element~$[u]$
by trivializing $u^*(TX_{n;\a},\phi_{n;\a})$ along the equator $S^1\!\subset\!\P^1$
as in the proof of \cite[Lemma~2.5]{Teh}.
By \cite[Corollary~3.8(2)]{RealGWsII}, the two orientations are the same.
They determine an orientation on the left-hand side of~\eref{REdgeCntr_e} via 
the short exact sequence~\eref{CIprp_e4b} pulled back to $\fM(X_{n;\a})$.
Since $\fM_{\Ga,\si;e}^{\phi,c;\circ}$ is connected, 
the last orientation agrees with the orientation described 
at the end of the previous paragraph.
Thus, the orientation on the left-hand side of~\eref{REdgeCntr_e} with the respect to
the  orientation on $\fM(X_{n;\a})$ 
arising from the canonical real orientation on $(X_{n;\a},\phi_{n;\a})$
is given by the right-hand side of~\eref{REdgeCntr_e} multiplied 
by $(-1)^{|\phi|}\!=\!(-1)^{|c|}$.
As stated at the end of \cite[Section~3.2]{RealGWsI},
the canonical orientation~of 
$$\fM(X_{n;\a})\equiv \fM_{0,0}(X_{n;\a},\fd(e))^{\phi_{n;\a},c}$$ 
with $c\!=\!\eta$
is reversed when it is viewed as a subspace of $\ov\fM_0(X_{n;\a},\fd(e))^{\phi_{n;\a}}$.
This accounts for the extra factor of $(-1)^{|\phi|}\!=\!-1$ 
in~\eref{ContrER_e} in this case.\\

\vspace{.5in}

\noindent
{\it  Institut de Math\'ematiques de Jussieu - Paris Rive Gauche,
Universit\'e Pierre et Marie Curie, 
4~Place Jussieu,
75252 Paris Cedex 5,
France\\
penka.georgieva@imj-prg.fr}\\

\noindent
{\it Department of Mathematics, Stony Brook University, Stony Brook, NY 11794\\
azinger@math.sunysb.edu}\\


\vspace{.2in}

\begin{thebibliography}{99}


\bibitem{AB} M.~Atiyah and R.~Bott, 
\emph{The moment map and equivariant cohomology}, Topology~23 (1984), 1--28

\bibitem{BHH} I.~Biswas, J.~Huisman, and J.~Hurtubise, 
\emph{The moduli space of stable vector bundles over a real algebraic curve}, 
Math.~Ann.~347 (2010), no.~1, 201--233


\bibitem{Teh} M.~Farajzadeh Tehrani,
\emph{Counting genus zero real curves in symplectic manifolds},
math/1205.1809v4, to appear in Geom.~Top.

\bibitem{Teh2} M.~Farajzadeh Tehrani,
\emph{Notes on genus one real Gromov-Witten invariants},
math/1406.3786

\bibitem{FO} K.~Fukaya and K.~Ono,
\emph{Arnold Conjecture and Gromov-Witten Invariant},
Topology 38 (1999), no.~5, 933--1048


\bibitem{FOOO}
K.~Fukaya, Y.-G.~Oh, H.~Ohta, and K.~Ono,
\emph{Lagrangian Intersection Theory: Anomaly and Obstruction},
AMS Studies in Advanced Mathematics~46, 2009

\bibitem{growi} A.~Gathmann, {\it GROWI},
available on the author's website 

\bibitem{Ge} P.~Georgieva,
\emph{The orientability problem in open Gromov-Witten theory},
Geom.~Top.~17 (2013), no.~4, 2485--2512

\bibitem{Ge2} P.~Georgieva,
\emph{Open Gromov-Witten invariants in the presence of an anti-symplectic involution},
math/1306.5019v2 


\bibitem{XCapsSetup} P.~Georgieva and A.~Zinger,
\emph{The moduli space of maps with crosscaps: Fredholm theory 
and orientability},  Comm.~Anal.~Geom.~23 (2015), no.~3, 81--140


\bibitem{RealGWsI} P.~Georgieva and A.~Zinger,
{\it Real Gromov-Witten theory in all genera and real enumerative geometry: construction},
math/1504.06617v3


\bibitem{RealGWsII} P.~Georgieva and A.~Zinger,
{\it Real Gromov-Witten theory in all genera and real enumerative geometry: properties},
math/1507.06633v2 

\bibitem{RealGWsApp} P.~Georgieva and A.~Zinger,
{\it Real Gromov-Witten theory in all genera and real enumerative geometry: appendix}, 
available from the authors' websites


\bibitem{Getzler} E.~Getzler, \emph{The elliptic Gromov-Witten invariants of~$\P^3$},
math.AG/9612009


\bibitem{GP} T.~Graber and R.~Pandharipande, {\it Localization of virtual classes},
Invent.~Math.~135 (1999), no.~2, 487--518

\bibitem{GH}  P.~Griffiths and J.~Harris,
\emph{Principles of Algebraic Geometry}, Wiley, 1994

\bibitem{Gr} M.~Gromov, \emph{Pseudoholomorphic curves in symplectic manifolds},  
Invent.~Math.~82 (1985), no~ 2, 307--347 


\bibitem{MirSym} K.~Hori, S.~Katz, A.~Klemm, R.~Pandharipande, 
R.~Thomas, C.~Vafa, R.~Vakil, and E.~Zaslow, {\it Mirror Symmetry},
Clay Math.~Inst., AMS, 2003 

\bibitem{Kollar} J.~Koll\'ar,
\emph{Examples of vanishing Gromov--Witten--Welschinger invariants},
math/1401.2387 

\bibitem{LT}  J.~Li and G.~Tian, 
\emph{Virtual moduli cycles and Gromov-Witten invariants of general symplectic manifolds}, 
Topics in Symplectic \hbox{$4$-Manifolds},
47-83, First Int.~Press Lect.~Ser., I, Internat.~Press, 1998


\bibitem{LZ}  J.~Li and A.~Zinger, 
\emph{On the genus-one Gromov-Witten invariants of complete intersections}, 
J. Differential Geom.~82 (2009), no.~3, 641–-690


\bibitem{Melissa} C.-C.~Liu,
\emph{Moduli of $J$-holomorphic curves with Lagrangian boundary condition and open
Gromov-Witten invariants for an $S^1$-pair},
math/0210257v2

\bibitem{Loo} E.~Looijenga, 
\emph{Smooth Deligne-Mumford compactifications by means of Prym level structures}, 
J.~Algebraic Geom.~3 (1994), 283--293 

\bibitem{McSa94} D.~McDuff and D.~Salamon, 
\emph{$J$-Holomorphic Curves and Quantum Cohomology},
University Lecture Series~6, AMS,~1994

\bibitem{MiSt} J.~Milnor and J.~Stasheff, 
\emph{Characteristic classes}, 
Annals of Mathematics Studies, No.~76. Princeton University Press, 1974


\bibitem{NZ} J.~Niu and A.~Zinger,
{\it Lower bounds for the enumerative geometry of positive-genus real curves},
in preparation

\bibitem{NZapp} J.~Niu and A.~Zinger,
{\it Lower bounds for the enumerative geometry of positive-genus real curves, appendix},
available from the authors' websites

\bibitem{PSW} R.~Pandharipande, J.~Solomon, and J.~Walcher,
\emph{Disk enumeration on the quintic 3-fold}, 
J.~Amer.~Math.~Soc. 21 (2008), no.~4, 1169--1209

\bibitem{bcov0_ci} A.~Popa and A.~Zinger,
\emph{Mirror symmetry for closed, open, and unoriented Gromov-Witten invariants},
Adv.~Math.~259 (2014), 448-–510


\bibitem{RT} Y.~Ruan and G.~Tian, \emph{A mathematical theory of quantum cohomology}, 
J.~Differential Geom.~42 (1995), no.~2, 259--367

\bibitem{RT2} Y.~Ruan and G.~Tian,
{\it Higher genus symplectic invariants and sigma models coupled with gravity},
Invent.~Math.~130 (1997), no.~3, 455--516

\bibitem{Wal} J.~Walcher, \emph{Evidence for tadpole cancellation in the
topological string}, Comm.~Number Theory Phys.~3 (2009), no.~1, 111--172


\bibitem{g2n2and3} A.~Zinger,
\emph{Enumeration of genus-two curves with a fixed complex structure in $\P^2$ and~$\P^3$},
J.~Diff.~Geom.~65 (2003), no.~3, 341--467

\bibitem{g1cone} A.~Zinger, 
\emph{On the structure of certain natural cones over moduli spaces of genus-one holomorphic maps},
Adv.~Math.~214 (2007), no.~2, 878–-933

\bibitem{PseudoCycles} A.~Zinger,
\emph{Pseudocycles and integral homology}, 
Trans.~AMS~360 (2008), no.~5, 2741–-2765

\bibitem{g1diff} A.~Zinger, 
\emph{Standard vs.~reduced genus-one Gromov-Witten invariants},
Geom.~Top.~12 (2008), no.~2, 1203--1241

\bibitem{g1comp} A.~Zinger,
\emph{A sharp compactness theorem for genus-one pseudo-holomorphic maps},
Geom.~Top.~13 (2009), no.~5, 2427-2522

\bibitem{g1comp2} A.~Zinger,  \emph{Reduced genus-one Gromov-Witten invariants}
J.~Differential Geom.~83 (2009), no.~2, 407--460



\bibitem{FanoGV} A.~Zinger, 
\emph{A comparison theorem for Gromov-Witten invariants in the symplectic category},
Adv.~Math.~228 (2011), no.~1, 535--574


\end{thebibliography}
\end{document}